\documentclass[11pt,reqno]{amsart} 
\usepackage[utf8]{inputenc} 
\usepackage[left=2.65cm,right=2.65cm,top=2.7cm,bottom=2.7cm]{geometry}
\usepackage{graphicx}
\usepackage{epstopdf}
\usepackage{color}
\usepackage{cite}
\usepackage{mathtools}
\mathtoolsset{showonlyrefs}
\usepackage{xcolor}
\usepackage{tikz}

\usepackage{array} 
\usepackage{paralist} 
\usepackage{verbatim} 
\usepackage{subfig} 

\usepackage{esint}
\usepackage{amsfonts}
\usepackage{amsmath}
\usepackage{amsthm}
\usepackage{amssymb}
\usepackage{hyperref}
\usepackage{dsfont}

\hypersetup{
	colorlinks=true,
	linkcolor=blue, 
	citecolor=blue,
	filecolor=blue,
	urlcolor=blue,
}

\newcommand{\mel}{\MoveEqLeft}

\newtheorem{theorem}{Theorem}[section]
\newtheorem{proposition*}{Proposition\textsuperscript{*}}
\newtheorem{corollary}[theorem]{Corollary}
\newtheorem{corollary*}{Corollary\textsuperscript{*}}
\newtheorem{proposition}[theorem]{Proposition}

\newtheorem{lemma}[theorem]{Lemma}

\theoremstyle{definition}

\newtheorem{remark}[theorem]{Remark} 

\newtheorem{example*}{Example\textsuperscript{*}}

\numberwithin{equation}{section}

\usepackage{bbm}

\def\Limes#1#2 {\lim\limits_{#1\rightarrow #2}}
\renewcommand{\Re}{\operatorname{Re}}

\def\eps{\epsilon}

\def\R{\mathbb{R}}
\def\C{\mathbb{C}}

\def\T{\mathbb{T}}
\def\Z{\mathbb{Z}}

\newcommand{\K}{\mathcal{K}}

\def\N{\mathbb{N}}

\def\XXint#1#2#3{{\setbox0=\hbox{$#1{#2#3}{\int}$ }
\vcenter{\hbox{$#2#3$ }}\kern-.59\wd0}}

\renewcommand{\div}{\grad\cdot}
\renewcommand{\epsilon}{\varepsilon}

\def\scalar#1#2{\langle #1,#2 \rangle}
\def\de{\partial}
\renewcommand{\div}{\operatorname{div}}

\def\dx{\,\mathrm{d}x}

\def\dl{\,\mathrm{d}l}

\newcommand{\david}[1]{{#1}}

\newcommand{\be}{\beta}
\newcommand{\ep}{\varepsilon}

\newcommand{\ga}{\gamma}
\newcommand{\ka}{\kappa}
\newcommand{\la}{\lambda}
\newcommand{\om}{\omega}
\newcommand{\si}{\sigma}
\newcommand{\te}{\theta}

\newcommand{\ze}{\zeta}
\newcommand{\Ga}{\Gamma}

\newcommand{\bx}{\mathbf{x}}

\newcommand{\bs}{\mathbf{s}}

\newcommand{\bn}{\mathbf{n}}

\newcommand{\tom}{\widetilde{\omega}_\ep}
\newcommand{\tu}{\widetilde{u}_\ep}
\newcommand{\tes}{\widetilde{e}_\bs}
\newcommand{\ten}{\widetilde{e}_\bn}

\def\CC{\mathbb{C}}

\def\RR{\mathbb{R}}

\def\ZZ{\mathbb{Z}}

\def\TT{\mathbb{T}}

\newcommand{\cC}{{\mathcal C}}

\newcommand{\cH}{{\mathcal H}}

\newcommand{\cX}{{\mathcal X}}
\newcommand{\cY}{{\mathcal Y}}

\newcommand{\cW}{{\mathcal W}}
\newcommand\cZ{\mathcal Z}

\newcommand{\pd}{\partial}
\newcommand\minus\backslash

\newcommand\lan\langle
\newcommand\ran\rangle

\newcommand{\sign}{\operatorname{sign}}

\newcommand{\supp}{\operatorname{supp}}

\newcommand{\dd}{{\,\mathrm d}}

\newcommand{\norm}[1]{\left\lVert#1\right\rVert}

\renewcommand\leq\leqslant
\renewcommand\geq\geqslant

\title[Desingularization of vortex sheets]{Desingularization of vortex sheets\\ for the 2D Euler equations}
 
\author[A. Enciso]{Alberto Enciso}
\thanks{Corresponding author: A.\ Enciso.}
\address{   \vspace{-0.2cm} 
\newline
\textbf{{\small Alberto Enciso}} 
\vspace{0.15cm}
\newline \indent Instituto de Ciencias Matem\'aticas, Consejo Superior de Investigaciones Cient\'\i ficas, 28049 Madrid, \indent Spain}
\email{aenciso@icmat.es}

 \author[A. J. Fern\'andez]{Antonio J.\ Fern\'andez}
 \address{ \vspace{-0.5cm}
\newline 
\textbf{{\small Antonio J. Fern\'andez}} 
\vspace{0.15cm}
\newline \indent Departamento de Matem\'aticas, Universidad Aut\'onoma de Madrid, 28049 Madrid, Spain}
 \email{antonioj.fernandez@uam.es}

 \author[D. Meyer]{David Meyer}
 \address{ \vspace{-0.5cm}
\newline 
\textbf{{\small David Meyer}} 
\vspace{0.15cm}
\newline \indent Max Planck Institute for Mathematics in the Sciences, Inselstrasse 22-26, 04103 Leipzig,
Germany}
 \email{david.meyer@mis.mpg.de}

\subjclass[2020]{}

\begin{document}
\begin{abstract} \vspace{-0.2cm}
We show how to regularize vortex sheets by means of smooth, compactly supported vorticities that asymptotically evolve according to the Birkhoff--Rott vortex sheet dynamics. More precisely, consider a vortex sheet initial datum~$\om^0_{\mathrm{sing}}$, which is a signed Radon measure supported on a closed curve. We construct a family of initial vorticities~$\om^0_\ep\in C^\infty_c(\RR^2)$ converging to~$\om^0_{\mathrm{sing}}$ distributionally as $\ep\to 0^+$, and show that the corresponding solutions $\om_\ep(x,t)$ to the 2D incompressible Euler equations converge to the measure defined by the Birkhoff--Rott system with initial datum $\om^0_{\mathrm{sing}}$. The regularization relies on a layer construction designed to exploit the key observation that the Kelvin--Helmholtz instability has a strongly anisotropic effect: while vorticities must be analytic in the ``tangential'' direction, the way layers can be arranged in the ``normal'' direction is essentially arbitrary.  \vspace{-0.7cm}
\end{abstract}

\maketitle

\section{Introduction}

The 2D incompressible Euler equations describe the movement of ideal fluids on the plane. In vorticity form, this system of equations can be written as
\begin{equation} \label{eq 2dEuler}
    \left\{
    \begin{aligned}
\ & \pd_t \om + u \cdot \nabla \om = 0 \quad && \textup{in } \R^2 \times (0,T)\,,\\
  & u  := K * \om && \textup{in } \R^2 \times [0,T)\,, \\      
  & \om(\cdot,0) = \om_0 && \textup{in } \R^2\,,
    \end{aligned}
    \right.
\end{equation}
where $\om: \R^2 \times[0,T) \to \R$ is the vorticity of the fluid, and $u: \R^2 \times [0,T ) \to \R^2$ is the corresponding velocity field, which is linked to the vorticity through the Biot--Savart law. Here, $\om_0: \R^2 \to \R$ is the initial value of the vorticity and
$$
K(x):= \frac{x^{\perp}}{2\pi|x|^2}\,, \quad x^{\perp} := (-x_2,x_1)\,,
$$
denotes the Biot--Savart kernel, so that the velocity field~$u$ is divergence-free and $\nabla^\perp\cdot u =\omega$.

In many physical situations, the vorticity tends to concentrate on a lower dimensional subset of physical space. Consequently, a classical topic in the study of the incompressible Euler equations is the analysis of solutions whose vorticity is concentrated on small sets, and the justification of effective models for these situations. The simplest case is that of vorticities supported on $N$ planar domains with small diameter. By letting these domains shrink down to points~$\xi_1^0,\dots,\xi_N^0\in\RR^2$, one formally arrives at the {\em $N$-vortex model}. This model posits that the evolution of the limit vorticity, which is the collection of deltas
\begin{equation}\label{E.om0sing_pt}
\omega^0_{N,\, \mathrm{sing}}(x):=\sum_{n=1}^N \varpi_n\, \delta(x-\xi_n^0)\,,
\end{equation}
with $\varpi_n\in\RR$,
should be morally given by 
\[
\omega_{N,\,\mathrm{sing}}(x,t):=\sum_{n=1}^N \varpi_n\, \delta(x-\xi_n(t))\,.
\]
Here $(\xi_n)_{n=1}^N:[0,T_0)\to\RR^{2N}$ is the only solution to the Kirchhoff--Routh system of ODEs
\begin{equation}\label{E.vortex}
\dot\xi_n(t)=\sum_{m\neq n}\varpi_m \, K(\xi_{n}(t)-\xi_{m}(t))\,,\qquad \xi_n(0)=\xi_n^0\,,\qquad n \in \{1, \ldots,N\}\,,
\end{equation}
which is defined for some~$T_0>0$.

To connect these heuristic considerations with actual Euler flows, one proceeds by desingularizing the system of $N$~point vortices. To this end, let us fix an arbitrary cutoff function, say $\chi\in C^\infty_c(\RR^2)$, which we normalize so that $\int_{\RR^2}\chi(x)\dx=1$. Using that the function $\chi_\epsilon(x):=\epsilon^{-2}\chi(x/\ep)$ converges weakly to the delta distribution as $\ep\to0^{+}$, one can consider the following regularization of~\eqref{E.om0sing_pt}:
\begin{equation}\label{E.om0_pt}
	\omega^0_{N,\epsilon}(x):=\sum_{n=1}^N \varpi_n\, \chi_\epsilon(x-\xi_n^0)\,.
\end{equation}
It is standard that $\om^0_{N,\ep}\to \om^{0}_{N,\,\mathrm{sing}}$ in the sense of distributions  as $\ep\to0^+$, and that the Euler equations~\eqref{eq 2dEuler} admit a unique global solution with initial datum~\eqref{E.om0_pt}, which we denote by $\omega_{N,\epsilon}\in C^\infty(\RR^2\times[0,\infty))$. However, the well-posedness theory for the Euler equations does not provide any a priori information on the structure of the solution.

A rigorous connection between Euler flows and the $N$-vortex model~\eqref{E.vortex} was first established in two classical results of Marchioro and Pulvirenti~\cite{MP, MP2}.  Their main result in \cite{MP2} ensures that there exists some time $T>0$, independent of the small parameter~$\epsilon$, such that the unique Euler flow with initial datum~\eqref{E.om0_pt}, $\om_{N,\epsilon}$, satisfies
\[
\omega_{N,\epsilon}(\cdot, t) \to \omega_{N,\,\mathrm{sing}}(\cdot,t)
\]
in the sense of distributions as $\ep \to 0^+$, for all $t \in [0,T)$. With a suitably chosen cutoff, much more precise estimates for the convergence of~$\om_\epsilon$ to $\omega_{\mathrm{sing}}$ were derived by D\'avila, del Pino, Musso and Wei in~\cite{DDMW} using gluing techniques. Vortex desingularization problems, both for the incompressible Euler equations and for related models such as the Navier--Stokes equations, have garnered significant attention in recent years; see, for example, \cite{Marchioro, Gallay, DG, CJY, DDMP, DDMP2, HHM}.

An important yet still insufficiently understood class of solutions with concentrated vorticity are vortex sheets. These are Euler flows whose vorticity is concentrated on a curve, which corresponds to considering an initial vorticity given by a measure of the form
\begin{equation}\label{E.om0_sing}
	\omega^0_{\mathrm{sing}}(x):=\int \varpi^0(s)\, \delta(x-\Gamma(s)) \,\mathrm{d}s\,.
\end{equation}
Here, $\Gamma(s)$ is a sufficiently smooth curve parametrized by arclength, and the vorticity amplitude $\varpi^0(s)$ is a smooth function{. The integral above is taken over the parameter domain of the curve. Note that the same applies to the next two integrals with their corresponding curves.}

A well-known formal computation (see e.g.\ \cite[Sec.\ 9.2]{Majda} for details) suggests that the evolution of this vorticity should morally be of the form
\begin{equation}\label{E.om_sing}
	\omega_{\mathrm{sing}}(x,t):=\int \varpi(s,t)\, \delta(x-\gamma(s,t)) \,\mathrm{d}s\,,
\end{equation}
where the curve $\gamma(\cdot,t)$ and the vorticity amplitude $\varpi(\cdot,t)$ at time~$t$ are obtained by solving the Birkhoff--Rott equations
\begin{equation}\label{E.BRintro}
\left\{
\begin{aligned}
    & \pd_t \ga(s,t)+ c(s,t) \pd_s\ga(s,t) = {\rm BR}[\ga;\, \varpi](s,t)\,,\\
    & \pd_t \varpi(s,t) + \pd_s (c(s,t) \varpi(s,t)) = 0\,,
\end{aligned}
\right.
\end{equation}
with initial conditions
$$
\ga(\cdot,0) = \Ga\,, \quad \textup{ and } \quad \varpi(\cdot,0) = \varpi^0\,.
$$
Here
\begin{equation}\label{E.BR_intro}
{\rm BR}[\ga;\, \varpi](s,t) := {\rm p.v.} \int K(\ga(s,t) - \ga(\varsigma,t)) \varpi(\varsigma,t) \dd \varsigma
\end{equation}
is the Birkhoff--Rott operator and $c(s,t)$ is an arbitrary fixed function whose choice is immaterial. Indeed, the appearance of the function~$c$ is associated with the fact that the measure~\eqref{E.om_sing} does not depend on how one parametrizes the curve (and the vorticity density), and, in a natural way, solutions to the Birkhoff--Rott equations with different~$c$ give rise to the same measure.
Lopes Filho, Nussenzveig Lopes and Schochet~\cite{Nussenzveig} {showed} that, under minimal assumptions, the vorticity defined by~\eqref{E.om_sing} is a weak solution to the Euler equations \eqref{eq 2dEuler} if and only if $(\ga,\varpi)$ satisfies~\eqref{E.BRintro}.

The well-posedness theory for vortex sheets is remarkably hard. This is mainly due to the fact that they are subject to Kelvin--Helmholtz instability~\cite{Moore}, so that higher Fourier modes of a small perturbation of a flat profile can get exponentially amplified in the evolution. While the Birkhoff--Rott equations are ill-posed on Sobolev spaces, as shown by Caflisch and Orellana in the seminal paper \cite{CO}, they are locally well-posed on analytic spaces~\cite{SS}. This can be heuristically understood as a consequence of the fact that the Birkhoff--Rott equations are elliptic in nature, a property that Lebeau~\cite{Lebeau} and Wu~\cite{Wu} have exploited to show that, under suitable assumptions, solutions to~\eqref{E.BRintro} become analytic instantaneously. It is known~\cite{CO} that analytic initial data can develop singularities in finite time. The important case of self-similar spirals (see \cite[Chapter 8]{Sa}) was analyzed in detail in \cite{SWZ,JS,CK}.  In addition, various special types of solutions, such as translating and rotating sheets, have been constructed in \cite{DR, Cao, Cao2, CW, CG, GP, MS}.

From the perspective of the 2D Euler equations, a celebrated result of Delort~\cite{Delort} ensures the existence of weak solutions to~\eqref{eq 2dEuler} with initial datum of the form~\eqref{E.om0_sing}, provided that the vorticity density~$\varpi^0$ does not change sign. 
However, even with this distinguished sign assumption, the uniqueness and structure of such solutions remain unknown. On the other hand, in the framework of Euler flows whose vorticity is no longer a measure, Sz\'ekelyhidi \cite{Sz} used convex integration techniques to show that there exist infinitely many bounded weak solutions to \eqref{eq 2dEuler} with a specific flat vortex sheet as initial data. Further developments are given in~\cite{Ambrose,Caflisch2,Mengual,Extension,Elgindi,Hou,DeRosaPark}.

In this paper, we are concerned with the dynamics of regularized vortex sheets and their connection with the Birkhoff--Rott equations. As a consequence of the aforementioned analytic subtleties, the situation here is much less clear-cut than in the case of point vortices, {and}, {to the best of our knowledge, the following two works} are {the only preceding} results on {convergence of thin Euler vortex layers to analytic Birkhoff--Rott evolution}. To analyze the vortex sheets, both impose periodic boundary conditions in the first coordinate (i.e., $x\in\TT\times\RR$) and assume that the curve $\Gamma(s):=(s,h_\Gamma(s))$ is a graph.

First, a classical result due to Benedetto and Pulvirenti~\cite{BP} considers the evolution of a thin, almost flat vortex layer of constant vorticity. This corresponds to taking a  patch-type vorticity of the form $\om_\ep^{0,\mathrm{BP}}(x):=\ep^{-1}\,\mathds{1}_{\Lambda_\ep}(x)$, where
$$
\Lambda_\ep:=\Big\{x=(x_1,x_2)\in\TT\times\RR: |x_2-h_\Gamma(x_1)|<\tfrac12\ep\varpi^0(x_1)\Big\}\,,
$$
is a small tubular neighborhood of the curve~$\Gamma$ whose width, of order~$\ep$, is proportional to the vorticity amplitude. Since the initial vorticity is bounded and compactly supported, the Euler equations with this datum have a unique solution $\om_\ep(x,t)$ by the Yudovich theory~\cite{Yudovich}. The authors show that, if $(\Gamma,\varpi)$ is a small analytic perturbation of a flat sheet of constant vorticity (i.e{.}, the functions~$\varpi^0$ and $h_\Gamma$ are analytic and close enough to $1$ and $0$ respectively), then there exists some positive time $T$ independent of~$\ep$ such that, for all $t \in [0,T)$, 
$$
\om_\ep^{\mathrm{BP}}(\cdot,t) \to \om_{\mathrm{sing}}(\cdot,t)\,,
$$
in the sense of distributions, as $\ep \to 0^+$. Note that the Birkhoff--Rott evolution~\eqref{E.BRintro} does determine $(\gamma,\varpi)$ for some positive time because the initial data is analytic. A drawback of this construction is that the regularized vorticity is discontinuous. It is not obvious a priori how to adapt this approach to achieve higher regularity; in fact, the authors explicitly point out that the construction should \textit{not} work if one replaces the indicator function by a general cutoff function~\cite[Section 4]{BP}.

The second result, {obtained} by Caflisch, Lombardo, and Sammartino~\cite{Caflisch}, considers the setting of reasonably small analytic perturbations of a flat sheet, where the slope satisfies certain quantitative bounds. In this paper, the regularized vorticity is no longer an indicator function; instead, these authors develop a functional setting which enables them to consider analytic vorticities $\om_\ep^{0,{\mathrm{CLS}}}$ that decay exponentially away from $\Gamma$ with a decay length of order~$\ep$. Hence, the vorticity does not have compact support, but it is smooth and sharply peaked around the curve~$\Gamma$ for small~$\ep$. The central result of their paper establishes a similar convergence result in this setting. To control the exponential decay of the solutions, the fact that~$\Gamma$ is a horizontal graph on the cylinder $\TT\times\RR$ is crucially used.

Our objective in this paper is to achieve the regularization of vortex sheets by smooth, compactly supported data on~$\RR^2$, in a manner that applies to sheets that are not necessarily graphs. Consequently,  we let $\Gamma:\TT\to\RR^2$ be an analytic planar curve parametrized by arclength, and $\varpi^0:\TT\to\RR$ be any analytic function. The fact that these functions are defined on~$\TT:=\RR/\ZZ$ instead of on an open interval simply reflects that we want the curve~$\Gamma$ to be closed. Without any loss of generality, we assume that $\Ga$ has unit length.
{We restrict our attention to closed curves in order to use periodic parametrizations and avoid endpoints issues. Noncompact analytic curves are outside the present setting, and curves with endpoints would require a separate treatment of the density near those endpoints}.
 
To state our main result, let us denote by $\dot\Gamma(s)$ the derivative of~$\Gamma(s)$ with respect to the arclength parameter~$s$.
For the sake of concreteness\footnote{Although throughout the paper we will consider significantly more general initial vorticities of the form $ \om^0_\ep(x):=\int_{-1}^1\int_{\TT}\varpi_{\ep,l}^0(s)\, \delta(x-\Gamma(s)-\ep (l+\eta^0_{\ep}(s))\dot\Gamma(s)^{\perp}) \,\mathrm{d}s \dl$, which offers greater flexibility for constructing solutions, the simpler choice of initial vorticity considered in the introduction suffices to present the key features of our construction.}, consider the initial vorticity
\begin{equation}\label{E.myom0_intro}
    \om^0_\ep(x):=\int_{-1}^1\int_{\TT}\varpi^0(s)\,\chi(l)\, \delta\big(x-\Gamma(s)-\ep l \dot\Gamma(s)^\perp\big) \,\mathrm{d}s \dl\,,
\end{equation}
where $\ep$ is a small positive constant.
Assuming that $\chi\in C^\infty_c((-1,1))$ is a nonnegative cutoff function, normalized so that $\int_\RR \chi(l)\dl=1$ {and such that $\chi(l) = 0$ if $|l| > 1/2$}, it is not hard to see that $\om_\ep^0\in C_c^\infty(\RR^2)$, that its support is contained {in} a tubular neighborhood around the curve~$\Ga$ of {total} width~$2\ep$, and that $\om_\ep^0\to\om^0_{\mathrm{sing}}$ in the sense of distributions as $\ep\to0^+$.

Our main result asserts that the regularization of vortex sheets given by~\eqref{E.myom0_intro} (or, indeed, by the considerably more general expressions discussed in Section~\ref{S.setup}) does possess all the properties that one may desire. A somewhat informal statement of this fact is the following; for a precise, more general statement, see Theorem~\ref{main thm}. It is worth mentioning that a minor modification of the proof (which amounts to replacing the Birkhoff--Rott operator~\eqref{E.BR_intro} by its periodic counterpart) yields an analogous result for the Euler equations on~$\TT^2$ or on~$\TT\times\RR$.

\begin{theorem}\label{T.main_intro}
Consider the initial vorticities~$\om^0_\ep\in C^\infty_c(\RR^2)$ given by~\eqref{E.myom0_intro}, which converge distributionally to the vortex sheet~\eqref{E.om0_sing} as $\ep \to 0^+$, and let $\om_\ep\in C^\infty(\RR^2\times [0,\infty))$ denote the unique solution to the Euler equations~\eqref{eq 2dEuler} with initial datum~$\om_\ep^0$.

If the curve~$\Gamma$ and the initial vorticity density~$\varpi^0$ are analytic, there exists some time $T>0$, which does not depend on~$\ep$, such that, for all $t\in(0,T)$, $\om_\ep(\cdot,t)$ converges distributionally as $\ep \to 0^+$ to the singular measure $\om_{\mathrm{sing}}(\cdot,t)$ defined by~\eqref{E.om_sing} in terms of the solution to the Birkhoff--Rott equations~\eqref{E.BRintro} with initial datum $(\Gamma,\varpi^0)$. 
\end{theorem}

The key insight underlying our construction is that the instability induced by the Birkhoff--Rott equations affects the dynamics of regularized vortex sheets in a highly anisotropic manner. Roughly speaking, the mechanism at play is as follows. Writing the regularized initial datum in terms of a continuum of parallel sheets, we need to impose that all these sheets are analytic to establish a relation with the Birkhoff--Rott equations, which are only locally well-posed on analytic spaces. However, it turns out that one can get by with much less regularity in the normal direction, as suggested by the result of Benedetto and Pulvirenti \cite{BP}. As we will see, this distinction between the tangential and normal regularity is essential and remains valid for vorticities much more general than that in~\eqref{E.myom0_intro}. 

{In this paper, we regularize the initial vorticity and then evolve it under the inviscid Euler equations. A related question is whether the same analytic Birkhoff--Rott evolution can be recovered as a vanishing-viscosity limit of Navier--Stokes solutions. One could start from the vortex sheet itself or from regularized initial data, in which case the layer thickness must also tend to zero. Our theorem does not address these limits or provide estimates uniform in both thickness and viscosity. The layer formulation developed here may offer a useful starting point, as it separates tangential transport from concentration in the normal direction. Extending the analysis to viscous flows would, however, require modified evolution equations and new estimates to account for diffusion, which transfers vorticity between layers and destroys compact support.}

The paper is organized as follows. In Section~\ref{S.setup}, we introduce the geometric and functional setting we will use, and state our main result, namely Theorem~\ref{main thm}. Then, in Section~\ref{S.proofMain}, we present the proof of Theorem~\ref{main thm}. To streamline the presentation, the proofs of most results stated there are postponed to later sections. Specifically, in Section~\ref{S.vortexSheets}, we derive an effective system of evolution equations used to obtain the family of solutions we seek. The existence of this family is established using Nishida’s Cauchy--Kovalevskaya theorem. The key estimates required for applying this theorem are proved in Sections~\ref{S.Geo} and~\ref{S.Kernel}. First, in Section~\ref{S.Geo}, we prove the necessary geometric estimates; then, in Section~\ref{S.Kernel}, we analyze the kernel appearing in the effective system of equations. Finally, in Section~\ref{Sec 7}, we address the results that imply convergence to the solution of the Birkhoff--Rott equations. 

\section{Geometric setup and main result} \label{S.setup}

For $\rho>0$, we set $\T_\rho:=\TT\times[-\rho,\rho]$ and denote by $\cH(\T_\rho)$ the space of {functions}~$f$, {continuous} on the {closed} strip 
\[
\big\{z=s+i\beta\in\CC : s\in\R,\; -\rho\leq \beta\leq \rho\big\}\,,
\]
{and holomorphic in its interior} satisfying the periodicity condition
\[
f(z)=f(z+1)\,.
\]
We denote by 
$$
\cX_\rho:=\big\{f\in \cH(\T_\rho): f(\T)\subset\R,\ \norm{f}_{\rho}<\infty\big\}\,,
$$
the space of holomorphic functions $f:\T_\rho\to\C$ that arise as complexifications of real-valued analytic functions on $\T:= \R / \Z$. We regard $\cX_\rho$ as a Banach space, equipped with the norm 
\begin{align} \label{E.norm rho}
\norm{f}_{\rho}:=\sup_{|\beta|\leq \rho}\norm{f(\cdot+i\beta)}_{C^{\frac{1}{2}}(\T)}.
\end{align}
{Note that the horizontal traces are uniformly} $C^{1/2}$ {up to the boundary. We regard these as real Banach spaces and set} $\cX_{0}=C^{1/2}(\TT,\RR)${. When a real function on} $\TT$ {is said to belong to} $\cX_{\rho}${, this means that its unique holomorphic extension belongs to that space. We use the same symbol for the function and its extension.}

\indent We fix some non-self-intersecting closed planar curve $\Gamma$, parametrized by arc-length. Rescaling it if necessary, without loss of generality, we can assume that the curve~$\Gamma=(\Gamma_1,\Gamma_2)$ has unit length, which yields a map $\Gamma:\T \to\R^2$. We suppose that the curve is analytic, so its complexification, which we still denote by $\Gamma$, is in $\cX_{\rho_0}^2$ for some fixed $\rho_0>0$. For convenience, we can assume that $\Gamma^{(5)}\in \cX_{\rho_0}^2$. We also assume that none of the ($\C^2$-valued) curves $\Gamma(\cdot+i\beta)$ is self-intersecting for $|\beta|\leq \rho_0$. In fact, for technical reasons, without any loss of generality, we can assume the stronger condition that for every $\delta>0$ we have \begin{align}
\min_{|\beta|\leq \rho_0,\, |s-\alpha|\geq\delta}\, \Re\left\{ \left[\Gamma_1(s+i\beta)-\Gamma_1(\alpha+i\beta)\right]^2+\left[\Gamma_2(s+i\beta)-\Gamma_2(\alpha+i\beta)\right]^2\right\}>0 \,. \label{tec no self int}
\end{align}
Indeed, by continuity and the fact that $\Gamma|_{\T}$ does not self-intersect, this condition always holds provided that $\rho_0$ is sufficiently small. Let us also recall that, possibly after reducing $\rho_0 > 0$, reparametrizing a curve by its arc-length preserves holomorphy, provided that $\dot{\Gamma} \neq 0$ at every point.

Next, observe that there is a neighborhood of the graph of $\Gamma$ in~$\R^2$ where every point $x$ can be uniquely written as 
\begin{equation}
x=\Gamma(s)+n\dot{\Gamma}(s)^\perp\,. \label{def coord}
\end{equation}
More precisely, by the analytic inverse function theorem~\cite[Theorem 2.5.3]{Krantz-Parks}, there is some constant $R_0>0$ and some tubular neighborhood $U_\Gamma\subset\R^2$ of the curve~$\Gamma$ such that~\eqref{def coord} defines an analytic diffeomorphism 
$
\TT\times(-R_0,R_0)\ni (s,n)\mapsto x\in U_\Gamma\,.
$
We will denote this diffeomorphism and its inverse by
\[
(s,n)\mapsto \bx(s,n)\qquad \text{and} \qquad x\mapsto (\bs(x),\bn(x))\,,
\]
respectively. Note that
\begin{align}
e_\bs(x):=\frac{\nabla \bs(x)}{|\nabla \bs(x)|^2}\,, \quad e_\bn(x):=\nabla \bn(x)\,,\label{def es}
\end{align}
is an analytic orthogonal (but not orthonormal) frame at every point $x\in U_\Gamma$. Since $|\nabla\bn|=1$, we have chosen the normalization factors so that
\[
e_\bs\cdot\nabla\bs=e_\bn\cdot\nabla\bn=1\quad \text{and}\quad e_\bs\cdot\nabla\bn=e_\bn\cdot\nabla\bs=0 \qquad \textup{in } U_\Gamma\,.
\]
As the curve $\Ga$ is analytic, the vector fields $e_\bs$, $e_{\bn}$ can be uniquely extended to complex-valued fields in $\cX_{\rho_0}^2$. {The complexification is performed in} the {parameter} $s${, after pulling back by} $\bx$.
{For each fixed real} $n$ {in a sufficiently small interval these are elements of} $\cX_{\rho_0}^2${. They are also jointly holomorphic in both variables in the domains used below}.
Note as well that, for every $x\in U_\Gamma$, \begin{align}
e_\bn(x)=\dot{\Gamma}(\bs(x))^\perp\qquad   \text{and}\qquad e_\bs(x)=\dot{\Gamma}(\bs(x))+\bn(x)\ddot{\Gamma}(\bs(x))^\perp, \label{expr tang}
\end{align}
as can be seen by differentiating \eqref{def coord} and taking the dot product with $\dot{\Gamma}(\bs(x))^\perp$ and $\dot{\Gamma}(\bs(x))+\bn(x)\ddot{\Gamma}(\bs(x))^\perp$, respectively.
{We also write}
\[
J(s,n):=\det(\partial_{s}\bx,\partial_{n}\bx)
       =1+n\,\ddot\Gamma(s)^{\perp}\cdot\dot\Gamma(s)>0,
\qquad \dd x=J(s,n)\,\dd s\,\dd n,
\]
{after reducing} $R_{0}$ {if necessary.}

We consider an initial vorticity for \eqref{eq 2dEuler} which is given as an infinite superposition of almost parallel vortex sheets. To make this precise, we start by choosing a family of analytic functions $(\eta_{\ep}^0)\subset \cX_{\rho_0}$ depending on a small parameter~$\ep>0$, which we will eventually use to ensure that our initial vorticity concentrates around the curve~$\Gamma$. Then, we define the functions
\begin{align}
\nu_{\ep,l}^0(s):=\eps[l+\eta_{\ep}^0(s)]\in \cX_{\rho_0}\,, \label{def eta}
\end{align}
which we parametrize by a real variable $l\in [-1,1]$, and use them to define a parametrized family of analytic curves as
\begin{align}
\gamma_{\ep,l}^0(s):=\Gamma(s)+\nu_{\ep,l}^0(s)\dot{\Gamma}(s)^\perp\in \cX_{\rho_0}^2\,. \label{def gamma}
\end{align}
Next, we take another parametrized family of functions $(\varpi_{\ep,l}^0)_{l\in [-1,+1]}\subset \cX_{\rho_0}$, which are analytic in $s$ but not necessarily in $l$, and which will play the role of vorticity intensities in our construction. We want this family to have compact support in the parameter~$l$, so we impose that $\varpi_{\ep,l}^0 \equiv 0$ for $|l| \geq 1/2$. Concerning the dependence on the parameter~$l$, we shall assume that $l\mapsto \varpi_{\ep,l}^0$ is in $C^0((-1,1),\cX_{\rho_0})$. 

Having this family of curves and intensities at hand, we consider an initial datum of the form 
\begin{equation} \label{def omega0}
\om_{\ep}^0(x) := \begin{cases}
  \tom^0(\bs(x),\bn(x)) &\text{if } x\in U_\Gamma\,,\\
  0 & \text{if } x\not\in U_\Gamma\,,
\end{cases}
\end{equation}
where  $\tom^0$ is such that for any test function $\Phi\in C^0(\R^2)$, one has \begin{equation}
\langle \omega_\ep^0, \Phi \rangle =\int_{\T}\int_{-1}^1\varpi_{\eps,l}^0(s)\Phi(\bx(s,\nu_{\eps,l}^0(s)))\dd s\dd l\,.
\end{equation}
It is straightforward to check that $\om_\ep^0 \in C_c^0(\R^2)${, since its coordinate expression is}
\[
\widetilde\omega_{\ep}^{0}(s,n)=
\frac{\varpi^0_{\ep,n/\ep-\eta_\ep^0(s)}(s)}{\ep J(s,n)},
\]
{with the layer density $\varpi_{\eps,l}^0$ extended by zero outside} $[-1,1]$, and that for small enough~$\ep$, its support is contained in a thin tubular neighborhood of the curve $\Gamma$:
\begin{equation} \label{E.suppom0}
\supp \om_\ep^0 \subset \left\{ x\in\RR^2: \mathrm{dist}(x,\Gamma) \leq \ep \Big(1+ \|\eta_{\ep}^0\|_{L^\infty(\TT)} \Big) \right\}\,.
\end{equation}
Moreover, $\om_\ep^0$ is as regular as the dependence on $l$ allows. In other words, if we assume that $\varpi_{\ep,l}^0$ is a $C^r$ {function} of~$l$ for some $0\leq r \leq \infty$, the initial datum $\om_\ep^0$ given in \eqref{def omega0} belongs to $C_c^r(\R^2)$. 
 
Yudovich's theory (see~\cite{Yudovich} or~\cite[Theorem 7.27]{Bahouri11}) guarantees the existence and uniqueness of weak solutions to \eqref{eq 2dEuler} for any initial vorticity in $L^1(\RR^2) \cap L^{\infty}(\R^2)$. Moreover, if the initial vorticity is in $ C^{r}_c(\R^2)$ for some $1\leq r \leq \infty$, then the unique solution to the Euler equation~\eqref{eq 2dEuler} is in $C_w([0,+\infty), C_c^r(\R^2))$ (where $C_w$ denotes the space of functions which are weakly-$*$ continuous). Although the arguments we will use work directly with weaker regularity in the variable~$l$, we are primarily interested in the case where  $\varpi_{\ep,l}^0$ is a smooth function of~$l$, which ensures that the vorticity is a $C^\infty$~function of~$x\in\RR^2$ and~$t\in[0,+\infty)$.

For any initial vorticity $\om_0 \in L^1(\R^2) \cap L^{\infty}(\R^2)$, the unique solution to \eqref{eq 2dEuler} is given by \cite{Yudovich}
$$
\om(x,t) := \om_0(X^{-1}_t(x))\,,
$$
where $X_t^{-1}$ denotes the inverse of the particle-trajectory map at time~$t$, namely the homeomorphism $X_t(x):=X(x,t)$ defined by the value at time~$t$ of the unique solution to the ODE
$$
\frac{\partial}{\partial t} X(x,t) = u(X(x,t),t)\,, \qquad X(x,0) = x\,.
$$
Thus, if we denote by $\om_\ep$ the unique solution to \eqref{eq 2dEuler} with initial datum $\om_\ep^0$ given in \eqref{def omega0}, then
\begin{equation} \label{def omega}
    \begin{aligned}
    \om_\ep(x,t) & = \om_\ep^0(X^{-1}_t(x))  \,.
    \end{aligned}
\end{equation}
Therefore, the support of~$\om_\ep(\cdot,t)$ is contained in the set
\[
\big\{X_t(\ga_{\eps,l}^0(s)): s\in\TT,\; l\in[- 1,1]\big\}\,.
\]
{For fixed} $\ep$, {boundedness} of $u_{\ep}$ {gives} $|X_{t}(x)-x|\leq t\sup_{0\leq\tau\leq t}\|u_{\ep}(\cdot,\tau)\|_{L^{\infty}}${, so} the {transported curves remain in} $U_\Gamma$ for {a} sufficiently short time. {This preliminary time may depend on} $\ep${. The uniformness of this time will follow from the geometric bounds for the forthcoming effective system.} 

{
\begin{remark}
For these data, the flow map depends on $\ep$ and should actually be denoted by $X_{\ep,t}$. However, we omit the $\ep$-dependence for the ease of notation. Likewise, the parametrization $\sigma_t$ below should be $\sigma_{\ep,l,t}$, but we also omit this dependence. 
\end{remark}
}

We can now assume that, for any small enough~$t$ and each $s \in \TT$, there exists a unique point $\si_t(s) \in \TT$ (depending on $\ep$ and $l$) satisfying
\begin{equation} \label{E.sgraph}
  \bs(X_t(\ga_{\ep,l}^0(\si_t(s)))) = s\,.  
\end{equation}
We shall assume that the map $\si_t:\TT\to\TT$ is a homeomorphism for each small enough~$t$. Geometrically, this assumption amounts to saying that the transported curve $X_t(\ga_{\ep,l}^0)$ remains a graph over the reference curve~$\Gamma$ (which will be an immediate consequence of the equations for small enough times). Note that $\si_0$ is the identity. { This graph property, which is automatic (for short times) if the velocity field $u_\ep$ is $C^1$, will be
justified in Step 2 of the proof of the main theorem.}

We reparametrize the transported curves as 
\begin{equation} \label{def gamma evolved}
    \gamma_{\ep,l}(s,t) := X_t(\ga_{\ep,l}^0(\si_t(s)))\,.
\end{equation}
Similarly as in the case of the family of initial curves, let us introduce the notation 
\begin{equation}  \label{def nu eps l}
\nu_{\eps,l}(s,t) :=\bn(X_t(\ga_{\ep,l}^0(\si_t(s))))
\end{equation}
for the $n$-coordinate of these curves so that $\gamma_{\eps,l}(s,t)=\Gamma(s)+\nu_{\eps,l}(s,t)\dot{\Gamma}(s)^\perp${. Define}
\[
\eta_{\ep,l}(s,t):=\ep^{-1}\nu_{\ep,l}(s,t)-l.
\]
{and observe that, whenever this derivative exists}, {this gives}
\begin{equation}  \label{def nu eps l2}
\nu_{\ep,l}(s,t) - \nu_{\ep,\ell}(s,t) = \ep \int_\ell^l \big( 1+ \pd_\mu \eta_{\ep,\mu}(s,t) \big) \dd \mu\,.
\end{equation}
In terms of this family of evolved curves, it is natural to make the ansatz that the vorticity~\eqref{def omega} at a small enough time~$t$ can be written as
\begin{equation}\label{E.ansatzomega}
    \om_\ep(x,t):=\begin{cases}
     \tom(\bs(x),\bn(x),t)   &\text{if } x\in U_\Gamma\,,\\
  0 & \text{if } x\not\in U_\Gamma\,,
    \end{cases}
\end{equation}
with
\begin{equation}\label{E.ansatzomegabis}
    \tom(s,n,t):=\frac{1}{J(s,n)}\int_{-1}^1 \varpi_{\ep,l}(s,t) \, \delta(n-\nu_{\ep,l}(s,t))   \, \mathrm{d}l\,.
\end{equation}
Moreover, as we will see later on, the density carried by each layer is explicitly given by
$$
{
\varpi_{\ep,l}(s,t) = \varpi_{\ep,l}^0(\sigma_t(s)) \pd_s \sigma_t(s)\,.
}
$$

The vorticity {is} therefore the compactly supported time-dependent distribution that acts on $\Phi\in C^0(\RR^2)$ as
\begin{equation}\label{def Phi om ep}
\langle  \om_\ep(\cdot,t), \Phi\rangle :=\int_\TT \int_{-1}^1 \varpi_{\ep,l}(s,t) \,\Phi(\bx(s,\nu_{\ep,l}(s,t)))\, \mathrm{d}l\,\mathrm{d}s\,.
\end{equation}
{The validity of this representation requires}, {in addition to boundedness}, {persistence} of the {graph coordinates and invertibility} in $l$ {and} $s${. These properties are proved on a uniform time interval in Step}~{2 of Section}~{\ref{S.proofMain}}. 

Of course, the vorticity intensity $\varpi_{\ep,l}$ at time~0 satisfies
\begin{equation}\label{E.initialcond}
\varpi_{\ep,l}(s,0)=\varpi_{\ep,l}^0(s)
\end{equation}
and, by definition,
\begin{equation}\label{E.initialcond2}
\nu_{\eps,l}(s,0)=\nu_{\ep,l}^0(s)\,,\qquad \eta_{\eps,l}(s,0)=\eta^0_\ep(s)\,.
\end{equation}

\begin{figure}[h]
\includegraphics[width=0.42\textwidth]{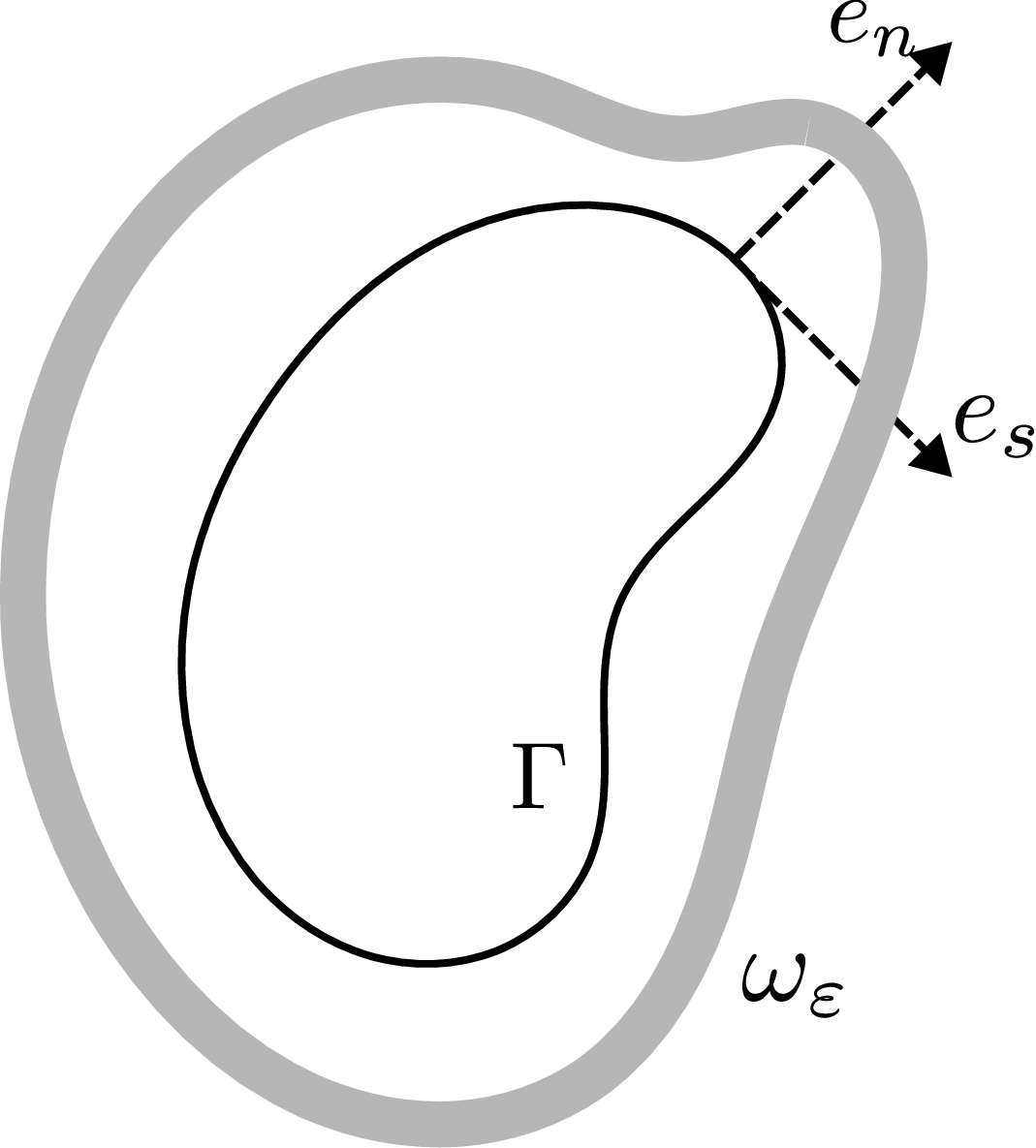}  \hspace{1.5cm}\includegraphics[width=0.42\textwidth]{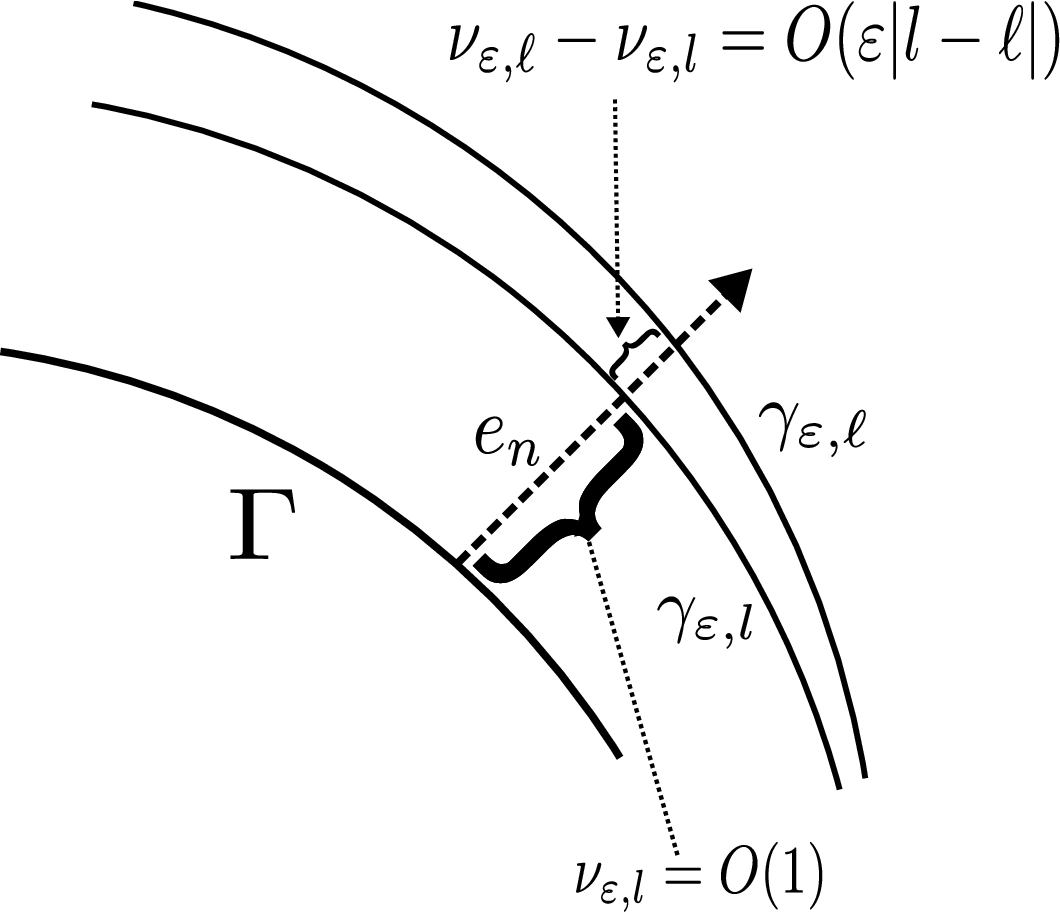}
\caption{In the left picture, we can see a schematic representation of the vortex and the reference curve, viewed from above. In the right one, a representation of the arrangement of the different objects at time $t\geq 0$ is depicted.}
\end{figure}

Having this notation and this geometric setting at hand, we can now give a more precise version of our main result. Here and in what follows, we will say that a quantity $\mu_l(s)$ is in the space $C^k\mathcal \cX_\rho$ with $\rho>0$ and $k$ a nonnegative integer if $l\mapsto \mu_l(\cdot)$ is in $C^k((-1,1),\mathcal \cX_\rho)$ and
\[
\|\mu_l\|_{C^k\mathcal \cX_\rho}:=\sum_{j=0}^ k\sup_{|l|\leq 1}\|\partial_l^j\mu_l\|_\rho<\infty\,.
\]

\begin{theorem}\label{main thm}
Let $\rho_0 > 0$ be a real number. We fix a reference curve $\Gamma\in\cX_{\rho_0}^2$ satisfying \eqref{tec no self int} and $\Gamma^{(5)}\in \cX_{\rho_0}^2$ and, for $\ep \in (0,1)$, consider a family of initial data  $(\eta_{\ep}^0)_\ep\subset  \cX_{\rho_0}$ and $ (\varpi_{\ep,l}^0)_\ep \subset C^0 \cX_{\rho_0}$, with $\varpi^0_{\ep,l}\equiv0$ for all $|l|> 1/2$, satisfying the uniform bound
\[
C_0:=\sup_{0<\ep<1} \left(\, \|\eta_{\ep}^{0}\|_{ \cX_{\rho_0}}+ \|\dot{\eta}_{\ep}^{0}\|_{ \cX_{\rho_0}} + \|\varpi_{\ep,l}^0\|_{C^0\mathcal \cX_{\rho_0}}\,\right)<\infty\,.
\]
Then, for any $\rho \in (0,\rho_0)$, there {exist} $\ep_{0}\in(0,1)$ {and} $T>0$, depending on $\Gamma${,} $\rho_0${,} $\rho$ and $C_{0}$, such that {the following conclusions hold for every} $\ep\in(0,\ep_{0}]$:
\begin{itemize}
    \item[\rm (i)] The solution $\om_\ep$ to the Euler equation~\eqref{eq 2dEuler} with initial datum \eqref{def omega0} is of the form \eqref{E.ansatzomega}, \eqref{E.ansatzomegabis} for all $t \in [0,T)$, and $\nu_{\eps,l}(\cdot,t),\,\eta_{\eps,l}(\cdot,t),\,\varpi_{\eps,l}(\cdot,t)$ belong to $\cX_\rho$. In particular, $\omega_\ep \in C([0,T), C_c^0(\R^2))$.\smallbreak
    \item[\rm (ii)] If $(\varpi_{\ep,l}^0)_\ep \subset C^r \cX_{\rho_{0}}$ for some non-integer $r > 0$, then  $\omega_\ep \in C_w([0,T), C_c^r(\R^2))$. \smallbreak
    \item[\rm (iii)] If $(\varpi_{\ep,l}^0)_\ep \subset C^r \cX_{\rho_{0}}$ for every $r > 0$, then     
    $\om_\ep\in C^\infty([0,T)\times\RR^2)$. \smallbreak
    \item[\rm (iv)] $\varpi_{\ep,l}(\cdot,t)\equiv0$  for all $|l|>3/4$ and all $t \in [0,T)$. In particular, {its support} is contained in
\[
\{x:\operatorname{dist}(x,\gamma_{\ep,0}(\TT,t))\leq C_{1}\ep\},
\]
{where} $C_{1}$ {depends} on {the fixed geometry and the uniform data bounds}, but not on $\ep$. 
   \item[\rm (v)] If we assume that
   \begin{equation} \label{E.convergenceInitialdata}
        \int_{-1}^1 \varpi_{\ep,l}^0 \dd l \to \varpi_0^0 \quad \textup{in } \cX_{\rho_0}\,, \quad \textup{as } \ep \to 0^{+}\,,
   \end{equation}
and denote by $(\nu_0, \varpi_0)$ the unique solution to the Birkhoff--Rott equations (written as a graph over the curve $\Ga$ as in \eqref{ev pi 2}) with initial datum $(0, \varpi_0^0)$, then this solution exists in $\cX_\rho^2$ up to time $T$. If we  set 
\begin{equation} \label{E.omegaSingTh}
\omega_{\rm sing}(x,t) := \left\{
\begin{aligned}
&\widetilde{\om}_0(\bs(x),\bn(x),t)   &&\text{if } x\in U_\Gamma\,,\\
& 0 && \text{if } x\not\in U_\Gamma\,,
\end{aligned}
\right.
\end{equation}
with
\begin{equation*}
    \widetilde{\omega}_0(s,n,t):= \frac{\varpi_{0}(s,t)}{J(s,n)} \, \delta(n-\nu_{0}(s,t))\,,
\end{equation*}
then it holds that
$$
\om_\ep(\cdot,t) \to \om_{\rm sing}(\cdot,t)\,, \quad \textup{as } \ep \to 0^{+}\,,
$$
for all $t \in [0,T)$, in the sense of distributions. Furthermore, we have the following quantitative convergence estimate
\begin{equation}
{ \sup_{|l|\leq 1}}
\norm{\nu_{\ep,l}(\cdot,t) - \nu_0(\cdot,t)}_{\rho} +\norm{ \int_{-1}^1 \varpi_{\ep,l}(\cdot,t) \dd l - \varpi_0(\cdot,t)}_{\rho}\lesssim \eps+\norm{\int_{-1}^1\varpi_{\eps,l}^0\dd l-\varpi_0^0}_{\rho_0}\,,\label{quant conv}
\end{equation}
where the implicit constant {is independent of} $\ep${,} $l$ {and} $t\in[0,T)$ {and} depends on $\Gamma${,} $C_{0}${,} $\rho_0$ {and the chosen radius} $\rho$. 
\end{itemize}
\end{theorem}

\begin{remark}
{
Our proof also applies to bounded layer densities and initial
graphs
$
\nu_{\ep,l}^0=\ep(c_\ep+l b_\ep),
$
provided $c_\ep,b_\ep$ and their first derivatives are uniformly
bounded in $\cX_{\rho_0}$, and $b_\ep-1$ is sufficiently small
in that space. Taking
$
\varpi_{\ep,l}^0
=b_\ep J(\cdot,\nu_{\ep,l}^0)\,
\mathds{1}_{(-1/2,1/2)}(l)
$
gives a patch of constant vorticity $\ep^{-1}$ between the
normal graphs $\ep(c_\ep\pm b_\ep/2)$.
The patches of Benedetto and Pulvirenti~\cite{BP} fit this
construction by expressing their boundaries in normal
coordinates (their small analytic data assumptions give the
required bounds). Thus, the proof recovers their patch
desingularization result, while allowing arbitrary embedded
analytic reference curves.}
\end{remark}

\section{Proof of the main result} \label{S.proofMain}

In this section, we present the proof of Theorem \ref{main thm}. To streamline the presentation, we will state several key auxiliary results whose proofs are relegated to later sections. Before we start the proof, let us provide a very brief outline of it. 

Very roughly speaking, our approach to show the preservation of the structure \eqref{E.ansatzomega}-\eqref{E.ansatzomegabis} is to rephrase \eqref{eq 2dEuler} as an evolution equation for the quantities $\nu_{\eps,l}, \, \partial_l\eta_{\eps,l}$ and $\varpi_{\eps,l}$, (cf.\ Proposition \ref{P.Evolution}), and then to show existence, uniqueness and boundedness of solutions to this effective system using an abstract Cauchy-Kovalevskaya type theorem, namely Theorem \ref{Nishida}. 

Since the effective system \eqref{ev gamma}--\eqref{ev pi} gives a PDE in the $s$ and $t$ variables only, our approach merely requires analyticity in the $s$-variable. This anisotropic effect is what allows us to deal with data which is analytic only in one direction. 

The key step for showing the applicability of Theorem \ref{Nishida} to this system is to show bounds on the velocity $u_\ep$, restricted to the curves $\gamma_{\eps,l}$ and taken as a function of $s$ in the analytic function spaces $\cX_\rho$ (cf.\ Corollary \ \ref{C.estimateUepsilon}). The crucial observation for showing these estimates is that we can use the prescribed structure of the vorticity as a superposition of vortex sheets to decompose the velocity into the contributions of the individual sheets (cf.\ \eqref{E.uep1}) and therefore we only need to show boundedness of a variant of the usual Birkhoff--Rott operator (cf.\ Proposition \ref{main est}), which describes the velocity that a vortex sheet induces on an almost parallel nearby sheet.

Once we have the desired structure \eqref{E.ansatzomega}, \eqref{E.ansatzomegabis}, and thus the proof of Theorem \ref{main thm} (i), the proofs of Theorem \ref{main thm} (ii), (iii) and (iv) are essentially immediate.

Finally, the proof of the convergence to the solution to the Birkhoff--Rott equation essentially amounts to proving the convergence to the actual Birkhoff--Rott operator. Once this result (cf. Proposition~\ref{Prop K conv}) is available, all that is left is to write the Birkhoff--Rott equation as a similar system and to use a variant of Gronwall's Lemma for the effective equations.

\subsection*{Step 1: Derivation of the effective evolution equations}

As a first step towards the proof of Theorem \ref{main thm}, we shall derive the evolution equations for $\nu_{\ep,l}$, $\pd_l \eta_{\ep,l}$ and $\varpi_{\ep,l}$ (see \eqref{def nu eps l}--\eqref{E.ansatzomegabis}). To do so, let us introduce some more notation. First, it is convenient to consider the expression of the velocity field $u_\ep := K* \om_\ep$ inside the tubular neighborhood~$U_\Gamma$ in terms of the coordinates $(s,n)$, introduced in \eqref{def coord}. We therefore set
\begin{equation} \label{def tilde}
\tu (s, n, t): = u_\ep (\bx (s, n), t)\,.
\end{equation}
Likewise, we define the orthogonal basis
$$
\tes(s,n) := e_\bs(\bx(s,n))\,, \quad \textup{and} \quad \ten(s,n) :=  e_\bn (\bx(s,n))\,.
$$
Note that the unit field $\ten(s,n)$ is independent of~$n$, so henceforth we will simply write $\ten(s)$. Now, let
\begin{equation}\label{E.usun}
    \tu^{\,\bs}(s,n,t):=\tu(s,n,t)\cdot \frac{\tes(s,n)}{|\tes(s,n)|^2}\,,\qquad  \tu^{\,\bn}(s,n,t):=\tu(s,n,t)\cdot {\ten(s)}
\end{equation}
denote the components of the velocity field in this basis, in terms of which one can write
\[
\tu(s,n,t)=\tu^{\,\bs}(s,n,t) \, \tes(s,n)+ \tu^{\,\bn}(s,n,t) \, \ten(s)\,.
\]

 We are now ready to give the system of evolution equations that $\varpi_{\ep,l}(s,t)$ and $\nu_{\eps,l}(s,t)$ must satisfy for~\eqref{def omega} to be a solution to the Euler equations. The proof of this result is postponed to Section \ref{S.vortexSheets}.

\begin{proposition} \label{P.Evolution} Let 
\begin{equation}\label{E.kappa}
\kappa(s,n) := \frac{\pd}{\pd s} \Big( \frac{\tes(s,n)}{|\tes(s,n)|^2}\Big)\,.
\end{equation}
Assume that the solution $\omega_\eps$ to the Euler equations has the structure \eqref{def gamma evolved}--\eqref{E.ansatzomegabis} and that $$\norm{\de_l\eta_{\eps,l}}_{L^\infty(\T \times [0,T))}<1$$ for all $|l| \leq 1$. Then, for $s\in\TT$ and $t\in (0,T)$:
\begin{align} \label{ev gamma}
    & \textup{\rm a) }\  \pd_t \nu_{\eps,l}(s,t) + \tu^{\, \bs} (s,\nu_{\eps,l}(s,t),t)\, \pd_s \nu_{\ep,l}(s,t) =  \tu^{\, \bn}(s,\nu_{\eps,l}(s,t),t) \,, \\
    & \label{ev eta} \textup{\rm b) }\ \pd_t \pd_l \eta_{\ep,l}(s,t) + \pd_s\Big( \tu^{\, \bs} (s,\nu_{\eps,l}(s,t),t) \big(1+ \pd_l \eta_{\ep,l}(s,t) \big) \Big) \\ \nonumber & \hspace{2.55cm} =  \big( \kappa(s,\nu_{\ep,l}(s,t)) \cdot \tu(s,\nu_{\eps,l}(s,t),t) \big)  \big(1+ \pd_l \eta_{\ep,l}(s,t) \big)\,, \\
    & \label{ev pi} \textup{\rm c) }\ \pd_t \varpi_{\ep,l}(s,t) + \pd_s \big(\tu^{\, \bs} (s, \nu_{\ep,l}(s,t),t)  \varpi_{\ep,l}(s,t)\big)= 0\,.
\end{align}
\end{proposition}

Having this result at hand, and before proceeding further, let us point out that, by arguing as in the proof of Proposition~\ref{P.Evolution}, one can verify (under the assumption of sufficient regularity) that the Birkhoff--Rott equations, when written as a graph over the curve $\Gamma$, read as
\begin{equation}\left\{\begin{aligned}  
 & \pd_t \nu_{0}(s,t) + \widetilde{u}_0^{\, \bs} (s,\nu_{0}(s,t),t)\, \pd_s \nu_{0}(s,t) =  \widetilde{u}_0^{\, \bn}(s,\nu_{0}(s,t),t) \,, \\
    & \label{ev pi 2} \pd_t \varpi_{0}(s,t) + \pd_s \big(\widetilde{u}_0^{\, \bs} (s, \nu_{0}(s,t),t)  \varpi_{0}(s,t)\big)= 0\,.
\end{aligned}
\right.
\end{equation}
Here the vortex sheet $\omega_0$ is written as \begin{equation}\label{E.ansatzomega2}
    \om_0(x,t):=\begin{cases}
     \widetilde{\om}_0(\bs(x),\bn(x),t)   &\text{if } x\in U_\Gamma\,,\\
  0 & \text{if } x\not\in U_\Gamma\,,
    \end{cases}
\end{equation}
with
\begin{equation}\label{E.ansatzomegabis2}
    \widetilde{\omega}_0(s,n,t):= \frac{\varpi_{0}(s,t)}{J(s,n)} \, \delta(n-\nu_{0}(s,t))\,.
\end{equation}
{ Equivalently, for $\Phi \in C^0(\RR^2)$} 
$$
{
\scalar{\omega_0(\cdot,t)}{\Phi} = \int_{\TT} \varpi_0(s,t) \Phi(\bx(s,\nu_0(s,t))) \dd s\,.
}
$$

Also, note that the velocity $\widetilde{u}_0$ at the curve $\gamma_0(s,t):=\Gamma(s)+\nu_0(s,t)\dot{\Gamma}(s)^\perp$ can be recovered from $\nu_0$ and $\varpi_0$ through the Biot--Savart law:
\begin{align}
\widetilde{u}_0(s,\nu_0(s,t),t) & =\frac{1}{2\pi}\, {\rm p.v.}\int_{\T} \frac{\left(\gamma_0(s,t)-\gamma_0(\varsigma,t)\right)^\perp}{|\gamma_0(s,t)-\gamma_0(\varsigma,t)|^2}\varpi_0(\varsigma,t)\dd\varsigma\\
& = \frac{1}{2\pi} {\rm p.v.}\,\int_{\T} \frac{\left(\gamma_0(s,t)-\gamma_0(s+\varsigma,t)\right)^\perp}{|\gamma_0(s,t)-\gamma_0(s+\varsigma,t)|^2}\varpi_0(s+\varsigma,t)\dd\varsigma\,.
\label{sheet velo}
\end{align} 

\begin{remark}
    Let us point out that, once we have a solution to the effective system \eqref{ev gamma}-\eqref{ev pi}, we can always ``come back'' to \eqref{eq 2dEuler}. Let $(\tu(\cdot,t),\ \nu_{\ep,l}(\cdot,t),\ \pd_l\eta_{\ep,l}(\cdot,t),\ \varpi_{\ep,l}(\cdot,t))$ be $C^1$-functions satisfying \eqref{ev gamma}-\eqref{ev pi} on some time interval $[0,T)$, and assume that:
\begin{itemize}
    \item[\rm (i)] The initial data $(\nu_{\eps,l}^0,\,\eta_{\eps}^0,\,\varpi_{\eps,l}^0)$ is of the form discussed in Section~\ref{S.setup}.
    \item[\rm (ii)] $\sup_{|l| \leq 1}\norm{\varpi_{\eps,l}}_{L^\infty(\T \times (0,T))}<\infty$.
    \item[\rm (iii)] $\sup_{|l| \leq 1}\norm{\pd_l\eta_{\eps,l}}_{L^\infty(\T \times (0,T))}<1$.
\end{itemize}
Then, $\omega_\ep$ given by \eqref{E.ansatzomega}-\eqref{E.ansatzomegabis} belongs to $C^1([0,T), \mathcal{D}'(\RR^2))$. Moreover, if its associated velocity vector field $u_\ep := K * \omega_\ep$ satisfies that
\begin{equation} \label{E.back}
 u_\ep({\bf{x}}(s,\nu_{\ep,l}(s,t)),t) = \tu(s,\nu_{\ep,l}(s,t),t) \,,
\end{equation}
then $\om_\ep$ is a solution to \eqref{eq 2dEuler}. Indeed, by \eqref{def Phi om ep}, \eqref{ev gamma} and \eqref{ev pi}, it follows that
$$
\scalar{\pd_t \omega_\ep(\cdot,t)}{\Phi} = \pd_t \scalar{\omega_\ep(\cdot,t)}{\Phi} = \int_{\TT} \int_{-1}^1 \varpi_{\ep,l}(s,t) \nabla \Phi(\bx(s,\nu_{\ep,l}(s,t))) \cdot \tu(s,\nu_{\ep,l}(s,t),t) \dd l \dd s\,.
$$
for all $t \in (0,T)$ and all $\Phi \in C^1(\RR^2)$. Likewise, by \eqref{def Phi om ep} and \eqref{E.back},
\begin{align*}
 & \scalar{\div(u_\ep(\cdot,t)  \omega_\ep(\cdot,t))}{\Phi}  =- \scalar{\omega_\ep(\cdot,t)}{u_\ep(\cdot,t)\cdot \nabla \Phi} \\ & \qquad = -\int_{\TT} \int_{-1}^1  \varpi_{\ep,l}(s,t) \nabla \Phi(\bx(s,\nu_{\ep,l}(s,t))) \cdot \tu(s,\nu_{\ep,l}(s,t),t) \dd l \dd s\,,
\end{align*}
for all $t \in (0,T)$ and all $\Phi \in C^1(\RR^2)$. 
Combining these identities, we conclude that $\omega_\ep$ is a distributional solution to \eqref{eq 2dEuler}. Note as well that, for $\omega_\ep^0$ as in \eqref{def omega0}, $\omega_\ep^0 \in L^1(\RR^2) \cap L^{\infty}(\RR^2)$ so we are in Yudovich's framework and we have uniqueness. 
\end{remark}

\subsection*{Step 2: Local existence for the effective system}

We aim to prove that the effective equations derived in Proposition~\ref{P.Evolution} are locally well-posed. Essentially as a consequence of the Kelvin--Helmholtz instability inherent to vortex sheets, we will need to consider solutions that are analytic in the~$s$ variable and use a theorem of Cauchy--Kovalevskaya type.
 
For this, we start by decomposing the velocity $\tu(s,\nu_{\ep,l}(s,t),t)$ into the contributions of the individual curves in \eqref{E.ansatzomega} by making use of \eqref{def nu eps l} and \eqref{def nu eps l2}. With $s\in\TT$, as assumed in the derivation of the effective equations, this amounts to the formula
\begin{align}
\tu(s,\nu_{\ep,l}(s,t),t) & =\frac{1}{2\pi}\int_{-1}^1\int_{\T} \frac{(\gamma_{\ep,l}(s,t)-\gamma_{\ep,\ell}(\varsigma,t))^\perp}{|\gamma_{\ep,l}(s,t)-\gamma_{\ep,\ell}(\varsigma,t)|^2}\varpi_{\ep,\ell}(\varsigma,t) \dd \varsigma\dd \ell \\
& = \frac{1}{2\pi}\int_{-1}^1\int_{\T} \frac{(\gamma_{\ep,l}(s,t)-\gamma_{\ep,\ell}(s+\varsigma,t))^\perp}{|\gamma_{\ep,l}(s,t)-\gamma_{\ep,\ell}(s+\varsigma,t)|^2}\varpi_{\ep,\ell}(s+\varsigma,t) \dd \varsigma\dd \ell\,.\label{E.uep1}
\end{align}

Given functions $\nu_1,\nu_2,g,\varpi\in \cX_\rho$, and under the assumption that $\nu_1,\nu_2,g$ are small enough, let us define the operator
\begin{equation}\label{E.defKep}
    K_\ep[\nu_1,\nu_2,g,\varpi](z):=\frac1{2\pi}\int_\TT \frac{\Xi^\ep[\nu_1,\nu_2,g](z,\varsigma)^\perp}{[\Xi_1^\ep[\nu_1,\nu_2,g](z,\varsigma)]^2+[\Xi_2^\ep[\nu_1,\nu_2,g](z,\varsigma)]^2}\varpi(z+\varsigma)\, \dd\varsigma\,.
\end{equation}
Here, $\Xi^\ep[\nu_1,\nu_2,g]\equiv(\Xi_1^\ep[\nu_1,\nu_2,g],\Xi_2^\ep[\nu_1,\nu_2,g])$ is defined as
\begin{gather*}
\Xi^\ep[\nu_1,\nu_2,g](z,\varsigma):=\Ga(z) - \Ga(z+\varsigma)
    +  [\nu_1(z)-\ep g(z)]\,\dot{\Ga}(z)^{\perp}-\nu_1(z+\varsigma)\dot{\Ga}(z+\varsigma)^{\perp}\,,
\end{gather*}
for $z\in\TT_\rho$ and $\varsigma \in \T$. {The argument} $\nu_{2}$ {is retained as a bookkeeping variable for} $\partial_{s}\nu_{1}$ {in the first-order system and in its norm estimates. The value of} $K_{\ep}$ {itself does not depend on $\nu_2$.}  

Extending the action of the operator~$K_\ep$ to functions that depend on~$(z,t)$ in the obvious way, our next goal is to consider the effective system~\eqref{ev gamma}--\eqref{ev pi} as an evolution problem for functions that are holomorphic in the variable $z\in\TT_\rho$. For this purpose, for $z\in\TT_\rho$ one can {\em define}\/ the field~$\tu$ appearing in this system as
\begin{align}
\tu(z,\nu_{\ep,l}(z,t),t) & := \int_{-1}^1 K_{\ep}\left[\nu_{\ep,\ell},\, \pd_s \nu_{\ep,\ell},\, \int_l^\ell (1+\pd_\mu \eta_{\ep,\mu}) \dd \mu,\,\varpi_{\ep,\ell}\right](z,t) \dd \ell\,.
\end{align}
By construction, this coincides with~\eqref{E.uep1} for $z = s\in\TT$. 

This formulation will allow us to apply Nishida's Cauchy--Kovalevskaya theorem~\cite{Nishida} to the evolution equations found in Proposition \ref{P.Evolution}. For the benefit of the reader, we state here the precise version of the theorem we are going to use:

\begin{theorem}[Nishida~\cite{Nishida}]\label{Nishida}
Let $  \{Y_\rho\}_{\rho\in[0,\infty)}$ be a scale of Banach spaces, and let $B^R_\rho:=\{w\in Y_\rho: \|w\|_{Y_\rho}<R\}$ denote the ball in~$Y_\rho$ of radius $R$. Assume that, for any $0\leq \rho'<\rho$, $Y_\rho$ is a linear subspace of $ Y_{\rho'}$ and $\| w \|_{Y_{\rho'}} \leq \| w \|_{Y_\rho}$ for all $w\in Y_\rho$. Consider the initial value problem 
\begin{equation}\label{E.Nishida}
    w'(t) = F(w(t), t)\,, \qquad 
    w(0) = w^{0}\,,
\end{equation}
and assume the following conditions on $F$ and $w^{0}$,  for  some positive numbers $R ,\delta ,\rho_0 $:
\begin{itemize}
    \item[(i)]  For every $0 < \rho' < \rho < \rho_0$, the map $F: B_\rho^R \times(-\delta,\delta)\to Y_{\rho'}, \ (w, t) \mapsto F(w, t),$ is continuous.
    \item[(ii)] For any $0<\rho' < \rho < \rho_0$ and all $w^1, w^2 \in B_\rho^R$, and for any $t\in (-\delta,\delta)$, $F$ satisfies
    \begin{equation*}
        \| F(w^1, t) - F(w^2, t) \|_{Y_{\rho'}} \leq \frac{C}{\rho - \rho'}\, \| w^1 - w^2 \|_{Y_\rho}\,,
    \end{equation*}
    where $C > 0$ is a fixed constant independent of $t, w^1, w^2, \rho$ or $\rho'$.
    \item[(iii)] $w^0 \in B_{\rho_0}^{R^0}$ for some $0 < R^0 < R$ and $F(w^0, \cdot ): (- \delta,\delta)\to Y_\rho$, $\,t \mapsto F(w^0,t)$, is a continuous function for every $0<\rho < \rho_0$. Moreover,  it satisfies, with a fixed constant $K > 0$,
    \begin{equation}
        \sup_{|t|<\delta}\| F(w^0, t) \|_{Y_\rho} \leq \frac{K}{\rho_0 - \rho}.
    \end{equation}
\end{itemize}
Under these hypotheses, there is a positive constant $\delta_0$ (depending only on $R-R^0$, $C$, $K$, $\rho_0$  {and} $\delta$) such that there exists a unique function $w(t)$ which, for every $0\leq \rho < \rho_0$, is a continuously differentiable function of $t\in (-(\rho_0 - \rho)\delta_0,\,(\rho_0 - \rho)\delta_0)$ with values in $B_\rho^R$ that satisfies the initial value problem~\eqref{E.Nishida}.
\end{theorem}

\begin{remark}
{In \cite{Nishida}, this result is only stated for $w^0=0$. }{ To obtain this formulation here, set} $v=w-w^{0}$ {and} $G(v,t)=F(w^{0}+v,t)$ {and apply it on the ball of radius} $R-R^{0}${. Inspection of the iteration in \cite[p.~632]{Nishida} shows that the contraction constraint uses only the Lipschitz constant} $C${, while the constraint keeping the iterates in the ball uses only} $K$ {and the available radius} $R-R^{0}${. Finally choose} $\delta_{0}\rho_{0}\leq\delta${. In particular, uniform bounds for these quantities give a uniform lifespan.}
\end{remark}

As already announced, we would like to apply Theorem \ref{Nishida} to solve the evolution equations found in Proposition \ref{P.Evolution}. To that end, we first set 
$$
\nu_\ep[l,s](t):= \nu_{\ep,l}(s,t)\,, \quad \eta_{\ep}[l,s](t) := \eta_{\ep,l}(s,t)\,, \quad \textup{and} \quad \varpi_{\ep}[l,s](t):= \varpi_{\ep,l}(s,t)\,,
$$
and define
\begin{equation} \label{E.defw}
\begin{aligned}
    w[l,s](t) & = \big(w_1[l,s](t),\, w_2[l,s](t),\, w_3[l,s](t),\, w_4[l,s](t)\big) \\
    &:=\big(\nu_\ep[l,s](t),\, \pd_s \nu_\ep[l,s](t),\, 1+\pd_l\eta_\ep[l,s](t),\, \varpi_{\ep}[l,s](t) \big)\,.
\end{aligned}
\end{equation}
{To measure geometric smallness, we write}
\[
\widehat w:=(w_{1},w_{2},w_{3}-1),\qquad
\mathfrak g_{\rho}(w):=\sum_{j=1}^{3}\sup_{|l|\leq1}\|\widehat w_{j}[l,\cdot]\|_{\rho}.
\]
\david{Note that, if ${\mathfrak g}_{\rho}(w)$ is small, then}
 $w_{3}$ {is close to} $1${, whereas the vorticity component} $w_{4}$ {only needs to be bounded.}
Also, we introduce
\begin{equation} \label{E.Uepsilon}
U_\ep\big(w[\cdot,\cdot]\big)(l,s) :=  \int_{-1}^1 K_{\ep}\left[\,w_1[\ell,\cdot],\, w_2[\ell,\cdot],\, \int_l^\ell w_3[\mu,\cdot] \dd \mu,\,w_4[\ell,\cdot]\right](s) \dd \ell
\end{equation}
and we further define:
\begin{equation}
\begin{aligned}
    \ U_\ep^{\bn}\big(w[\cdot,\cdot]\big)(l,s) &:= U_\ep\big(w[\cdot,\cdot]\big)(l,s) \cdot \ten(s)\,, \\
 U_\ep^{\bs}\big(w[\cdot,\cdot]\big)(l,s) &:= U_\ep\big(w[\cdot,\cdot]\big)(l,s) \cdot \frac{\tes(s,w_1[l,s])}{|\tes(s,w_1[l,s])|^2}\,,\\
     \K\big(w[\cdot,\cdot]\big)(l,s) &:=  \kappa(s,w_1[l,s])\,.
\end{aligned}
\end{equation}
Since $\K\big(w[\cdot,\cdot]\big)$ only depends on the first component of~$w$, in what follows we will also use the notation
\[
\K\big(w_1[\cdot,\cdot]\big)\,.
\]

Considering functions that depend on $[l,s](t)$ in the obvious way, we then consider the initial value problem
\begin{equation} \label{E.initialValueProblem}
w'(t) = F(w(t))\,, \quad 
    w(0) = w^0\,,
\end{equation}
where
\begin{align*}
    & F_1(w) := U_\ep^{\bn}(w) - U_\ep^{\bs}(w) w_2\,, \\
    & F_2(w) := \pd_s \big(U_\ep^{\bn}(w) - U_\ep^{\bs}(w) w_2 \big) \,,  \\
    & F_3(w) := \Big( \K(w) \cdot U_\ep(w)\Big) w_3 - \pd_s\Big(U_\ep^{\bs}(w) w_3\Big)\,,\\
    & F_4(w) := - \pd_s \Big( U_\ep^{\bs}(w) w_4 \Big)\,,
\end{align*}
and where
\begin{equation} \label{E.dataIntialValueProblem}
w^0\,[l,s]:= (\nu_{\ep,l}^0(s),\, \de_s\nu^0_{\ep,l}(s),\, 1,\,\varpi_{\ep,l}^0(s))\,,
\end{equation}
so that \eqref{E.initialValueProblem} is the system introduced in Proposition \ref{P.Evolution}.

We want to use Theorem \ref{Nishida} to solve \eqref{E.initialValueProblem}--\eqref{E.dataIntialValueProblem}. Hence, for $\rho \geq 0$, let us introduce the scale of Banach spaces 
$$
\cY_\rho^k:=\bigg\{f:\T_\rho\times [-1,1]\rightarrow \C^k\bigg|\ l\mapsto f_j(\cdot,l)\in L^\infty([-1,1],\cX_\rho)\, \textup{ \& }\sup_{|l| \leq 1} \|f_j(\cdot,l)\|_\rho 
 <\infty\,,\ 1\leq j\leq k\bigg\},
$$
endowed with the norm
$$
\|f\|_{\cY_\rho^k} := \sum_{j=1}^k \, \sup_{|l| \leq 1} \|f_j(\cdot,l)\|_\rho \,.
$$

Of course, the variables $w_1,w_2,w_3$ in \eqref{E.defw} are not independent of each other, and we cannot expect the system to behave well for arbitrary data in $\cY^4_\rho$. We define the set of admissible data as \begin{align}
\mkern-45mu\mathcal{Z}_\rho=\mathcal{Z}_\rho^\eps:=\left\{ w\in\cY^4_\rho\,\bigg|\, \de_sw_1[l,\cdot]=w_2[l,\cdot] \textup{ and } w_1[l,s]-w_1[\ell,s]=\eps \int_\ell^l w_3[\mu,s] \dd \mu, \  \forall\  l,\ell\right\}.
\label{def z}\end{align}
Note that, for every $\rho \geq 0$, $\cZ_\rho$ is a Banach subspace of $\cY_\rho^4$, endowed with the same norm as $\cY_\rho^4$.
{Indeed, the two defining constraints are linear and closed.}

Moreover, if $w\in\cZ_\rho$ {satisfies the smallness assumptions imposed below} and $F(w)\in\cY_{\rho'}^4$ for some $0<\rho'<\rho$, then $F(w)\in\cZ_{\rho'}$.
{The first compatibility condition follows directly from the definition $F_2(w)=\partial_sF_1(w)$. For the second, the same incompressibility calculation used to derive \eqref{ev eta}, together with $\partial_lw_1=\epsilon w_3$, gives
\[
F_1(w)[l,s]-F_1(w)[\ell,s]
=
\epsilon\int_\ell^l
\left[
\big(\mathcal K(w)\cdot U_\epsilon(w)\big)w_3
-\partial_s\big(U_\epsilon^{\mathbf s}(w)w_3\big)
\right][\mu,s]\,\dd\mu
=
\epsilon\int_\ell^l F_3(w)[\mu,s]\,\dd\mu.
\]
Here, the smallness assumptions ensure that $w_3>0$ on real arguments, so the change of variables between the layer parameter and the normal coordinate is invertible. As in the proof of \eqref{ev eta}, the calculation can be performed directly in integrated form and therefore also applies to bounded measurable layer densities.}

Having this functional setting at hand, in the next two results, whose proofs are postponed to Sections \ref{S.Geo} and \ref{S.Kernel}, we prove the required estimates to verify that \eqref{E.initialValueProblem} fits the framework of Theorem \ref{Nishida}.

\begin{proposition}\label{geo est}
There exists a constant $\mathcal{C}_1>0$, depending only on $\Ga$, such that, for every $0 \leq \rho \leq \rho_0\,,$  if $w_1^1, w_1^2 \in \cY^1_\rho$ satisfy
$$
 \sup_{|l| \leq 1} \Big\{ \norm{w_1^1[l,\cdot]}_{\rho} +  \norm{w_1^2[l,\cdot]}_{\rho} \Big\}  < \mathcal{C}_1\,,
$$
then
\begin{align}
 \sup_{|l| \leq 1} \bigg\{ \norm{\K\big(w_1^i[\cdot,\cdot]\big)(l,\cdot)}_{\rho}+ \norm{\frac{e_\bs(\cdot,w_1^i[l,\cdot])}{|e_\bs(\cdot,w_1^i[l,\cdot])|^2}}_{\rho} \bigg\}  \lesssim 1\,, \quad i \in \{1,2\}\,,
\end{align}
and
\begin{align}
     \sup_{|l| \leq 1} \bigg\{ & \norm{\K\big(w_1^1[\cdot,\cdot]\big)(l,\cdot)-\K\big(w_1^2[\cdot,\cdot]\big)(l,\cdot)}_{\rho} \\
     & + \norm{\frac{\tes(\cdot,w_1^1[l,\cdot])}{|\tes(\cdot,w_1^1[l,\cdot])|^2}-\frac{\tes(\cdot,w_1^2[l,\cdot])}{|\tes(\cdot,w_1^2[l,\cdot])|^2}}_{\rho} \bigg\} \lesssim \, \sup_{|l| \leq 1} \norm{w_1^1[l,\cdot]-w_1^2[l,\cdot]}_\rho \,,
\end{align}
with implicit constants depending on $\Ga$, but not on $\rho$.  
\end{proposition}

\begin{proposition}\label{main est}
There {exist constants} $\mathcal{C}_2>0$ {and} $\ep_{*}>0$, depending only on $\Ga$, such that{, for} $0<\ep\leq\ep_{*}$, for every $0 \leq \rho \leq \rho_0\,,$ if $w^1, w^2 \in \mathcal{Z}_\rho$ satisfy
\begin{equation} \label{ass small}
\mathfrak{g}_\rho(w^1)+\mathfrak{g}_\rho(w^2)
< \mathcal{C}_2\,,  \quad \textup{and} \quad \mathcal{C}_3 := \sup_{|l|\leq 1} \big\{  \norm{w_4^1[l,\cdot]}_\rho + \norm{w_4^2[l,\cdot]}_\rho \big\} < \infty \,,
\end{equation}
then
\begin{align}\label{bd K}
 \sup_{l,\, \ell \in [-1,1],\ l\neq \ell} \, \bigg\|K_{\ep}\bigg[w_1^i[\ell,\cdot],\,w_2^i[\ell,\cdot],\, \int_l^\ell w_3^i[\mu,\cdot] \dd \mu,\,w_4^i[\ell,\cdot]\bigg](\cdot) \bigg\|_\rho \lesssim 1\,, \quad i \in \{1,2\}\,, \quad 
\end{align}
and
\begin{equation}
\begin{aligned}
\sup_{l,\, \ell \in [-1,1],\ l\neq \ell} \,     \bigg\|&K_{\ep}\bigg[w_1^1[\ell,\cdot],\,w_2^1[\ell,\cdot],\, \int_l^\ell w_3^1[\mu,\cdot] \dd \mu,\,w_4^1[\ell,\cdot]\bigg](\cdot) \\
& - K_{\ep}\bigg[w_1^2[\ell,\cdot],\,w_2^2[\ell,\cdot],\, \int_l^\ell w_3^2[\mu,\cdot] \dd \mu,\,w_4^2[\ell,\cdot]\bigg](\cdot) \bigg\|_\rho  \lesssim \|w^1-w^2\|_{\cY^4_\rho}\,,  \label{E.KLipschitz}
\end{aligned}
\end{equation}
with implicit constants depending on $\Ga$ and~$\cC_3$, but not on $\rho$.
\end{proposition}

As an immediate corollary, we also obtain the following estimates. 

\begin{corollary} \label{C.estimateUepsilon}
{
Let $\mathcal{C}_2 > 0$ and $\ep_* > 0$ be as in Proposition \ref{main est}, depending only on $\Ga$,} and assume $0 < \ep \leq \ep_*$. For every $0 \leq \rho \leq \rho_0\,,$ and every $w^1, w^2 \in \mathcal{Z}_\rho$ satisfying \eqref{ass small}, it follows that
    \begin{align} \label{E.UepsBounded}
    \sup_{|l| \leq 1} \norm{U_\ep\big(w^i[\cdot,\cdot]\big)(l,\cdot)}_\rho  \lesssim 1\,, \quad i \in \{1,2\}\,,
\end{align}
and that
\begin{equation}\label{E.UepsLipschitz}
\begin{aligned} 
    \sup_{|l| \leq 1}  \norm{U_\ep\big(w^1[\cdot,\cdot]\big)(l,\cdot) - U_\ep\big(w^2[\cdot,\cdot]\big)(l,\cdot)}_\rho 
     \lesssim \|w^1-w^2\|_{\cY^4_\rho}\,.
\end{aligned}
\end{equation}
\end{corollary}

We can now prove the existence of a solution to \eqref{E.initialValueProblem} with $w(0) = w^0$ given in \eqref{E.dataIntialValueProblem}. We set $\mathcal{C}_0:= \min\{\mathcal{C}_1, \mathcal{C}_2\}$, depending only on $\Ga$, and, for $0 \leq \rho \leq \rho_0\,,$ consider $w^1, w^2 \in \mathcal{Z}_\rho$ satisfying
\begin{equation}
\mathfrak{g}_\rho(w^1)+\mathfrak{g}_\rho(w^2) < \mathcal{C}_0  \quad \textup{and} \quad \mathcal{C}_3 := \sup_{|l|\leq 1} \big\{  \norm{w_4^1[l,\cdot]}_\rho + \norm{w_4^2[l,\cdot]}_\rho \big\} < \infty \,.
\end{equation}

{We first check that the corollary implies that this $F$ fulfills the assumptions in Theorem \ref{Nishida}.}
Using that $\cX_\rho$ is an algebra (since $C^\frac12(\T)$ is), we see that
\begin{align}
    & \sup_{|l|\leq 1} \norm{F_1(w^i[\cdot,\cdot])(l,\cdot)}_\rho \\
    & \qquad \leq  \sup_{|l| \leq 1} \norm{U_\ep(w^i[\cdot,\cdot])(l,\cdot)}_\rho \left( \|\ten\|_\rho + \sup_{|l|\leq 1} \bigg\{ \norm{w_2^i[l,\cdot]}_\rho \norm{\frac{\tes(\cdot,w_1^i[l,\cdot])}{|\tes(\cdot,w_1^i[l,\cdot])|^2}}_{\rho} \bigg\} \right), \  i \in \{1,2\}\,,
\end{align}
and thus, by Proposition \ref{geo est} and Corollary \ref{C.estimateUepsilon} we get that
\begin{equation} \label{E.F1bounded}
 \sup_{|l|\leq 1} \norm{F_1(w^i[\cdot,\cdot])(l,\cdot)}_\rho \lesssim 1\,, \quad i \in \{1,2\}\,.
\end{equation}
On the other hand, observe that
\begin{align*}
    & F_1(w^1) - F_1(w^2) = \big(U_\ep(w^1) - U_\ep(w^2)\big) \cdot \left( \ten(\cdot) - w_2^2\frac{\tes(\cdot,w_1^2)}{|\tes(\cdot,w_1^2)|^2}  \right)  \\ & \qquad  - (w_2^1-w_2^2) \, U_\ep(w^1) \cdot \frac{\tes(\cdot,w_1^1)}{|\tes(\cdot,w_1^1)|^2} - w_2^2U_\ep (w^1) \cdot \left( \frac{\tes(\cdot,w_1^1)}{|\tes(\cdot,w_1^1)|^2} - \frac{\tes(\cdot,w_1^2)}{|\tes(\cdot,w_1^2)|^2} \right).   
\end{align*}
Combining Proposition \ref{geo est} with Corollary \ref{C.estimateUepsilon}, and using again that $\cX_\rho$ is an algebra, we then get
\begin{align} \label{E.F1Lipschitz}
    \sup_{|l|\leq 1} \norm{F_1(w^1[\cdot,\cdot])(l,\cdot) - F_1(w^2[\cdot,\cdot])(l,\cdot)}_\rho \lesssim \|w^1 - w^2\|_{\cY^4_\rho}\,.
\end{align}

Now, to deal with $F_2$, we use that
$$
F_2(w^i[\cdot,\cdot])(s) = \pd_s \left(F_1(w^i[\cdot,\cdot])(s)\right)\,.
$$
Indeed, by Cauchy's integral formula, it follows that
$$
\sup_{|l| \leq 1} \norm{F_2(w^i[\cdot,\cdot])(l,\cdot)}_{\rho'} \lesssim \frac{1}{\rho-\rho'} \, \sup_{|l| \leq 1}\,\norm{F_1(w^i[\cdot,\cdot])(l,\cdot)}_\rho\,, \quad i \in \{1,2\}\,,
$$
uniformly in $0 \leq \rho' < \rho \leq \rho_0$. Hence, using \eqref{E.F1bounded} and \eqref{E.F1Lipschitz}, we get that 
\begin{equation} \label{E.F2bounded}
     \sup_{|l|\leq 1} \norm{F_2(w^i[\cdot,\cdot])(l,\cdot)}_{\rho'} \lesssim \frac{1}{\rho-\rho'}\,, \quad i \in \{1,2\}\,,
\end{equation}
and
\begin{align} \label{E.F2Lipschitz}
    \sup_{|l|\leq 1} \norm{F_2(w^1[\cdot,\cdot])(l,\cdot) - F_2(w^2[\cdot,\cdot])(l,\cdot)}_{\rho'} \lesssim \frac{1}{\rho-\rho'}\,\|w^1 - w^2\|_{\cY^4_\rho}\,.
\end{align}

 Next, observe that, for $i \in \{1,2\}$,
\begin{align}
    & \sup_{|l| \leq 1} \norm{F_3(w^i[\cdot,\cdot])(l,\cdot)}_{\rho'} \leq \sup_{|l| \leq 1} \norm{\pd_s \big( U_\ep^{\bs}(w^i[\cdot,\cdot])(l,\cdot)\, w_3^i[l,\cdot]\big)}_{\rho'} \\
    & \qquad + \sup_{|l| \leq 1} \Big\{ \norm{w_3^i[l,\cdot]}_{\rho'} \, \norm{\K(w_1^i[\cdot,\cdot])(l,\cdot)}_{\rho'} \norm{U_\ep(w^i[\cdot,\cdot])(l,\cdot)}_{\rho'} \Big\}\,.
\end{align}
The first term on the right-hand side can be estimated as we did with $F_2$. Indeed, combining the Cauchy integral formula with Proposition \ref{geo est} and Corollary \ref{C.estimateUepsilon}, we get that,
$$
\sup_{|l| \leq 1} \norm{\pd_s \big( U_\ep^{\bs}(w^i[\cdot,\cdot])(l,\cdot)\, w_3^i[l,\cdot]\big)}_{\rho'} \lesssim \frac{1}{\rho-\rho'}\,,
$$
uniformly in $0 \leq \rho' < \rho \leq \rho_0$. The second can be estimated directly by combining Proposition \ref{geo est} with Corollary \ref{C.estimateUepsilon}, as we did with $F_1$. In short, we get that
\begin{align} \label{E.F3bounded}
     \sup_{|l| \leq 1} \norm{F_3(w^i[\cdot,\cdot])(l,\cdot)}_{\rho'}  \lesssim \frac{1}{\rho-\rho'}\,.
\end{align}
Also, observe that
\begin{align}
    & F_3(w^1)-F_3(w^2) = (w_3^1-w_3^2)\,\K(w^1) \cdot U_\ep(w^1) \\
    & \quad  + w_3^2 \Big( \big(\K(w^1) - \K(w^2) \big) \cdot U_\ep(w^1)  +\K(w^2) \cdot \big( U_\ep(w^1) - U_\ep(w^2) \big) \Big) \\
    & \quad - \pd_s \bigg[ (w_3^1-w_3^2)\, \frac{e_\bs(\cdot,w_1^1)}{|e_\bs(\cdot,w_1^1)|^2} \cdot U_\ep(w^1) \\
    & \hspace{1.35cm} + w_3^2 \bigg( \bigg( \frac{e_\bs(\cdot,w_1^1)}{|e_\bs(\cdot,w_1^1)|^2} - \frac{e_\bs(\cdot,w_1^2)}{|e_\bs(\cdot,w_1^2)|^2} \bigg) \cdot U_\ep(w^1) + \frac{e_\bs(\cdot,w_1^{2})}{|e_\bs(\cdot,w_1^{2})|^2} \cdot \big(U_\ep(w^1) - U_\ep(w^2)\big) \bigg) \bigg]\,.
\end{align}
Having this expression  at hand, one can argue as we did to prove \eqref{E.F1Lipschitz} and \eqref{E.F2Lipschitz} and conclude: 
\begin{align} \label{E.F3Lipschitz}
    \sup_{|l|\leq 1} \norm{F_3(w^1[\cdot,\cdot])(l,\cdot) - F_3(w^2[\cdot,\cdot])(l,\cdot)}_{\rho'} \lesssim \frac{1}{\rho-\rho'}\,\|w^1 - w^2\|_{\cY^4_\rho}\,.
\end{align}

Finally, to estimate $F_4$, we once again use Cauchy's integral formula to infer that
$$
\sup_{|l|\leq 1} \norm{F_4(w^i[\cdot,\cdot])(l,\cdot)}_{\rho'} \lesssim \frac{1}{\rho-\rho'}\, \sup_{|l| \leq 1} \norm{U_\ep^\bs(w^i[\cdot,\cdot])(l,\cdot)\, w_4^i[l,\cdot]}_{\rho}\,, \quad i \in \{1,2\}\,,
$$
uniformly in $0 \leq \rho' < \rho \leq \rho_0$. Arguing exactly as for the previous terms, we conclude that
\begin{align} \label{E.F4bounded}
     \sup_{|l| \leq 1} \norm{F_4(w^i[\cdot,\cdot])(l,\cdot)}_{\rho'}  \lesssim \frac{1}{\rho-\rho'}\,, \quad i \in \{1,2\}\,,
\end{align}
and that
\begin{align} \label{E.F4Lipschitz}
    \sup_{|l|\leq 1} \norm{F_4(w^1[\cdot,\cdot])(l,\cdot) - F_4(w^2[\cdot,\cdot])(l,\cdot)}_{\rho'} \lesssim \frac{1}{\rho-\rho'}\,\|w^1 - w^2\|_{\cY^4_\rho}\,.
\end{align}

{

We can now apply Theorem~\ref{Nishida} to
\eqref{E.initialValueProblem}. Fix a sufficiently small constant $r>0$
such that
\[
4r<\min\{\mathcal C_0,\tfrac12\},
\qquad 2r<R_0,
\]
where $R_0$ is the width of $U_\Gamma$ as defined in the beginning of Sect. \ref{S.setup}, and choose $0<\ep_0\leq\ep_*$ so that
\[
\mathfrak g_{\rho_0}(w^0)
\leq
\ep\bigl(1+\|\eta_\ep^0\|_{\rho_0}
          +\|\partial_s\eta_\ep^0\|_{\rho_0}\bigr)
\leq r
\qquad\text{for }0<\ep\leq\ep_0.
\]
This choice is possible by the uniform bound on the initial data.
Moreover, $w^0\in\cZ_{\rho_0}$, since
$\nu_{\ep,l}^0=\ep(l+\eta_\ep^0)$ and $w_3^0=1$.

Set $v:=w-w^0$ and $G_\ep(v):=F(w^0+v)$. If
$\|v\|_{\cZ_\rho}<r$, then
\[
\mathfrak g_\rho(w^0+v)<2r,
\qquad
\sup_{|l|\leq1}\|(w^0+v)_4[l,\cdot]\|_\rho\leq C_0+r.
\]
Consequently, the preceding estimates apply to every pair $v,\tilde{v}$
in this ball and give
\[
\|G_\ep(v)-G_\ep(\widetilde v)\|_{\cZ_{\rho'}}
\lesssim  \frac{1}{\rho-\rho'}\,
       \|v-\widetilde v\|_{\cZ_\rho},
\qquad
\|G_\ep(0)\|_{\cZ_\rho}
\lesssim \frac{1}{\rho_0-\rho},
\]
with constants independent of $\ep\in(0,\ep_0]$. Thus,
Theorem~\ref{Nishida}, applied to
\[
\partial_t v=G_\ep(v),\qquad v(0)=0,
\]
provides a unique solution on a time interval independent of $\ep$.

Fix the radius $\rho$ in the conclusion of Theorem~\ref{main thm}
and choose $\rho_1\in(\rho,\rho_0)$. We obtain a time $T>0$,
independent of $\ep$, such that
\[
\sup_{0\leq t<T}\|w(t)-w^0\|_{\cZ_{\rho_1}}<r.
\]
In particular, the solution is uniformly bounded on a strip wider
than $\T_\rho$. Integrating the bounds for $F$ between these two
strips also gives
$\mathfrak g_\rho(w(t))\leq C(\ep+t)$, with $C$ independent of $\ep$.
Hence, the size $\mathfrak{g}_\rho(w(t))$ of the solution on $\T_\rho$ can be made
arbitrarily small by decreasing $\ep_0$ and $T$. 

We next verify that this solution has the Euler interpretation
required in Theorem~\ref{main thm}. Set
\[
\nu_{\ep,l}=w_1[l,\cdot],\qquad
\eta_{\ep,l}=\ep^{-1}\nu_{\ep,l}-l,\qquad
\varpi_{\ep,l}=w_4[l,\cdot].
\]
The compatibility conditions and the bounds above imply
\[
\partial_l\nu_{\ep,l}=\ep w_3,
\qquad
\tfrac12\leq w_3\leq\tfrac32
\]
on real arguments throughout $[0,T)$. Thus
$l\mapsto\nu_{\ep,l}(s,t)$ remains invertible. Since
$|\nu_{\ep,l}|<2r<R_0$, all the graphs remain in $U_\Gamma$, and
the map
\[
(s,l)\longmapsto\bx(s,\nu_{\ep,l}(s,t))
\]
has Jacobian
$\ep J(s,\nu_{\ep,l})w_3$, bounded below by a positive multiple
of $\ep$. Therefore \eqref{def Phi om ep} defines a compactly
supported vorticity in $L^1(\RR^2)\cap L^\infty(\RR^2)$ for each
fixed $\ep$. Its Biot--Savart velocity agrees with $U_\ep(w)$ on
the layer graphs by construction. Testing
\eqref{ev gamma} and \eqref{ev pi} against smooth functions,
as in the remark following Proposition~\ref{P.Evolution},
shows that this vorticity solves \eqref{eq 2dEuler}.
Yudovich uniqueness identifies it with the solution having
initial datum \eqref{def omega0}.

Next, we verify that the curves obtained above have the
parametrization \eqref{def gamma evolved}. For each $l\in[-1,1]$,
let $a_l(\cdot,t)$ be the flow on $\TT$ associated with the
tangential velocity $U_\ep^\bs(w)(l,\cdot,t)$, with
$a_l(\cdot,0)=\mathrm{Id}$. The obtained analytic bounds give a uniform
bound on the spatial derivative of this velocity, so
$a_l(\cdot,t)$ remains a diffeomorphism throughout $[0,T)$.
By \eqref{ev gamma} and uniqueness of particle trajectories, it then follows that
\[
\gamma_{\ep,l}(a_l(s,t),t)
=
X_{t}(\gamma_{\ep,l}^0(s)).
\]
Hence, \eqref{def gamma evolved} holds with
$\sigma_{t}=a_l(\cdot,t)^{-1}$. Likewise, the continuity equation \eqref{ev pi} gives
\[
\varpi_{\ep,l}(s,t)
=
\varpi_{\ep,l}^0(\sigma_{t}(s))
\,\partial_s\sigma_{t}(s).
\]
In particular, $\varpi_{\ep,l}$ remains zero for $|l|>1/2$.
Moreover, the bound on $\partial_l\nu_{\ep,l}$ implies
\[
|\nu_{\ep,l}(s,t)-\nu_{\ep,0}(s,t)|
\leq \tfrac32\ep|l|,
\]
which proves the support and thickness assertions in (iv).
The identity $\eta_{\ep,l}=\ep^{-1}\nu_{\ep,l}-l$ gives
$\eta_{\ep,l}\in\cX_\rho$ for each fixed $\ep$, while the
transport formula for the continuous initial vorticity yields
$\omega_\ep\in C([0,T),C_c^0(\RR^2))$.
This concludes the proof of (i).

To prove (ii), observe that the assumed $C^r$ regularity in $l$,
together with analyticity in $s$, gives
$\omega_\ep^0\in C_c^r(\RR^2)$.
The Euler regularity recalled in Section~\ref{S.setup} then yields
\[
\sup_{0\leq t\leq T'}\|\omega_\ep(\cdot,t)\|_{C^r}<\infty
\qquad\text{for every }T'<T.
\]
The regularity above gives as well
$u_\ep\in L^\infty([0,T'],C^{r+1}(\RR^2))$
for every $T'<T$. Since $r$ is noninteger, we can identify $C^r$ with $B_{\infty,\infty}^r$ and, using for instance \cite[Proposition~4]{DanchinFanelli},
 obtain
\[
\omega_\ep\in C_{w}([0,T),C^r(\RR^2))\,,
\]
which proves (ii). The constants in these regularity estimates
may depend on $\ep$, while the existence time $T$ is the one
already obtained.

Finally, under the hypotheses of (iii), the same argument gives
bounds at arbitrarily high spatial orders. Interpolation with
continuity in $C^0$ gives continuity in every finite $C^m$ norm,
and the Euler equation then gives smoothness in time.
This proves (iii).
} 
\subsection*{Step 3: Convergence to the Birkhoff--Rott equations} First, let us recall that the Birkhoff--Rott equations written as a graph over the curve $\Ga$ were introduced in \eqref{ev pi 2}. Moreover, taking into account \eqref{sheet velo}, for $\nu_1, \varpi \in \cX_\rho$, and under the assumption that $\nu_1$ is small enough, let us set
\begin{equation} \label{E.U0}
K_0[\nu_1,\varpi](z):= \frac{1}{2\pi} {\rm p.v.} \int_{\T} \frac{\Xi^0[\nu_1](z,\varsigma)^{\perp}}{[\Xi_1^0[\nu_1](z,\varsigma)]^2+[\Xi_2^0[\nu_1](z,\varsigma)]^2} \, \varpi(z+\varsigma) \dd \varsigma\,.
\end{equation}
Here, $\Xi^0[\nu_1] = (\Xi_1^0[\nu_1],\Xi_2^0[\nu_1])$ is defined as
$$
\Xi^0[\nu_1](z,\varsigma):= \Ga(z) - \Ga(z+\varsigma) + \nu_1(z) \dot{\Ga}(z)^{\perp} - \nu_1(z+\varsigma) \dot{\Ga}(z+\varsigma)^{\perp}\,, \quad \textup{for } z \in \T_\rho \textup{ and } \varsigma \in \T \,.
$$
For later purposes, we also write 
\begin{align}
&U_0^{\,\bn}((w_1,w_4)):=K_0[w_1,w_4]\cdot \widetilde{e}_{\bn}\,,\\
&U_0^{\,\bs}((w_1,w_4)):=K_0[w_1,w_4]\cdot \frac{\widetilde{e}_{\bs}(s,w_1(\cdot))}{|\widetilde{e}_{\bs}(s,w_1(\cdot))|^2}\,.
\end{align}

Now, we denote by
\begin{align}
w^\eps[l,s](t)=\big(\nu_{\eps,l}(s,t),\de_s\nu_{\eps,l}(s,t),1+\de_l\eta_{\eps,l}(s,t),\varpi_{\eps,l}(s,t)\big) \in \cZ_\rho\,,
\end{align}
the solution to \eqref{E.initialValueProblem} obtained in the previous subsection, and denote by
$$
v^0(s,t):= (\nu_0(s,t),\,\varpi_0(s,t))\,,
$$
the unique solution to \eqref{ev pi 2} with initial datum $(0,\varpi_0^0)$.
Likewise, we set 
\begin{equation}\bar\nu_\eps:=\frac{1}{2}\left(\nu_{\eps,-1}+\nu_{\eps,1}\right)\,,
\end{equation}
and 
\begin{align}
    \bar{w}^\eps(s,t)=\left(\bar{\nu}_\eps(s,t),\int_{-1}^1 \varpi_{\eps,l}(s,t)\dd l\right)\,.
\end{align}
We shall focus on the quantitative estimate here and explain the other points at the end.
{By \eqref{def nu eps l2}, it follows that}
\begin{equation} \label{E.nubarnul}
\|\nu_{\ep,l}(\cdot,t) - \bar{\nu}_\ep(\cdot,t)\|_{\rho} \leq 4 \ep \Big( 1+ \sup_{|l|\leq 1} \norm{\pd_l \eta_{\ep,l}(\cdot,t)}_{\rho} \Big) \lesssim \ep\,, \quad \textup{for all } t \in [0,T)\,.
\end{equation}
Hence, to control the difference $\nu_0(\cdot,t)-\nu_{\ep,l}(\cdot,t)$ for all $l \in [-1,1]$ and all $t \in [0,T)$, it is enough to control the difference $\nu_0(\cdot,t)-\bar{\nu}_\ep(\cdot,t)$. 

Now, taking into account {\eqref{ev gamma}} and \eqref{ev pi 2}, we get that
\begin{equation} \label{E.diff-nu0}
    \pd_t(\nu_0 - \bar{\nu}_\ep) + E_1^\ep(v^0, v^0-\bar{w}^\ep,t) + E_2^\ep(t) = 0\,,
\end{equation}
where we use the shorthand notations
\begin{align}
& \bullet\ E_1^\eps(v^0, v^0-\bar{w}^\ep,t) :=U_0^{\,\bs}(v^0)\de_s(\nu_0-\bar{\nu}_\eps)+(U_0^{\,\bs}(v^0)-U_0^{\,\bs}(\bar{w}^\eps))\de_s\bar{\nu}_\eps-\left(U_0^{\,\bn}(v^0)-U_0^{\,\bn}(\bar{w}^\eps)\right)\,,\\
& \bullet\ E_2^{\eps}(s,t):=\left(U_0^{\,\bs}(\bar{w}^\eps)-\frac{1}{2}\big(U_\eps^{\,\bs}(w^\eps)(1,s)+U_\eps^{\,\bs}(w^\eps)(-1,s)\big)\right)\de_s\bar{\nu}_\eps\\
&\qquad+\frac{1}{4}\left(U_\eps^{\,\bs}(w^\eps)(-1,s)-U_\eps^{\,\bs}(w^\eps)(1,s)\right)\de_s(\nu_{\eps,1}-\nu_{\eps,-1})\\
&\qquad-\left(U_0^{\,\bn}(\bar{w}^\eps)-\frac{1}{2}\left(U_\eps^{\,\bn}(w^\eps)(1,s)+U_\eps^{\,\bn}(w^\eps)(-1,s)\right)\right)\,.
\end{align}
Likewise, it follows that 
\begin{equation} \label{E.diff-pdsnu0}
    \pd_t \pd_s(\nu_0 - \bar{\nu}_\ep) + \pd_sE_1^\ep(v^0, v^0-\bar{w}^\ep,t) + \pd_sE_2^\ep(t) = 0\,.
\end{equation}
On the other hand, taking into account \eqref{ev pi} and \eqref{ev pi 2}, we get that
\begin{equation} \label{E.diff-varpi0}
\pd_t  \left(\varpi_0-\int_{-1}^1 \varpi_{\eps,l}\dd l\right)+E_3^\eps(v^0, v^0-\bar{w}^\ep,t)+E_4^\eps(t)=0\,,
\end{equation}
with
\begin{align}
& \bullet\ E_3^\eps(v^0, v^0-\bar{w}^\ep, t) :=\de_s\left(U_0^{\,\bs}(v^0)\left(\varpi_0-\int_{-1}^1 \varpi_{\eps,l}\dd l\right)\right)+\de_s\left(\left(U_0^{\,\bs}(v^0)-U_0^{\,\bs}(\bar{w}^\eps)\right)\int_{-1}^1 \varpi_{\eps,l}\dd l\right)\,, \\
& \bullet\ E_4^\eps(s,t):=\de_s\left(U_0^{\,\bs}(\bar{w}^\eps)\int_{-1}^1 \varpi_{\eps,l}\dd l-\int_{-1}^1 U_\ep^{\, \bs}(w^{\ep})\,\varpi_{\eps,l}\dd l\right).
\end{align}

Having \eqref{E.diff-nu0}, \eqref{E.diff-pdsnu0} and \eqref{E.diff-varpi0} at hand, we will use a Gronwall-type argument to prove Theorem \ref{main thm} (v). 

In the next two results, whose proofs are postponed to Section \ref{Sec 7}, we prove the estimates required to perform the Gronwall-type argument. First, we show that the $\eps$-regularized velocity $U_\eps$, defined in \eqref{E.Uepsilon}, converges to the velocity in \eqref{E.U0}, up to the expected jump discontinuity of the vortex sheet. For the identification of the sign of the jump, it is necessary to fix the orientation of $\Gamma$. We assume that $\Gamma$ is oriented so that $\dot{\Gamma}^\perp$ points inside the domain enclosed by $\Gamma$. The other case can be handled by arguing in a similar way. {We suppress the time variable in the operator estimates below and write} $\gamma_{\ep,0}(s)=\Gamma(s)+w_{1}[0,s]\dot\Gamma(s)^{\perp}${. The symbol} $A$ {denotes a positive decreasing function of the strip gap which may be enlarged at successive occurrences, independently of} $\ep${.}

\begin{proposition}\label{Prop K conv}
There exists a constant $\cC_4 > 0$, depending only on $\Ga$, such that, for every $0 \leq \rho < \rho_0$, if $w \in \cZ_\rho$ satisfies 
$$
\mathfrak g_\rho(w) < \cC_4 \quad \textup{and} \quad \cC_5:= \sup_{|l|\leq 1} \norm{w_4[l,\cdot]}_\rho < \infty\,,
$$
then, for all $0 \leq \rho' < \rho$, and all $l \in [-1,1]$, it follows that
\begin{align}
    \bigg\| &\,  U_\eps(w[\cdot,\cdot])(l,\cdot) -K_0\left[\frac{w_1[1,\cdot]+w_1[-1,\cdot]}{2},\int_{-1}^1w_4[\mu,\cdot]\dd\mu\right] \\
    & -\frac{\gamma_{\eps,0}'(\cdot)}{2|\gamma_{\eps,0}'(\cdot)|^2}\int_{-1}^{1}\left(\mathds{1}_{\mu\geq l}-\mathds{1}_{\mu\leq l}\right)w_4[\mu,\cdot]\dd\mu \,\bigg\|_{\rho'} \lesssim \ep A(\rho-\rho')\,,
\end{align}
and that
\begin{align}
    \bigg\| &\, \frac{1}{2}\Big(U_\eps(w[\cdot,\cdot])(1,\cdot)+U_\eps(w[\cdot,\cdot])(-1,\cdot)\Big) \\
    & -K_0\left[\frac{w_1[1,\cdot]+w_1[-1,\cdot]}{2},\int_{-1}^1w_4[\mu,\cdot]\dd\mu\right]\bigg\|_{\rho'}\lesssim \ep A(\rho-\rho')\,.
\end{align}
\end{proposition}

\begin{remark}
It is natural to lose regularity here.  Roughly speaking, quantitative estimates for the difference between the velocity at different curves correspond to taking a normal derivative of the velocity at the curves.
Furthermore, it is also natural not to have convergence of $U_\eps$ to $K_0$ without an additional correction term, because the limit of the velocity should contain the jump of the vortex sheet.
\end{remark}

The second result we require is the analogue of Proposition \ref{main est} and Corollary \ref{C.estimateUepsilon} for $K_0$. To simplify the notation, we set
$$
\cW_\rho:= \Big\{ f = (f_1,f_2) \in \cX_\rho^2\ \big|\ \pd_s f_1 \in \cX_\rho\Big\}\,.
$$
We regard $\cW_\rho$ as a Banach space, equipped with the norm
$$
\norm{f}_{\cW_\rho} := \|f_1\|_\rho + \|\pd_s f_1\|_\rho + \|f_2\|_\rho\,.
$$

\begin{lemma}\label{lem K0}
There exists a constant $\mathcal{C}_6 > 0$, depending only on $\Ga$, such that, for every $0 \leq \rho < \rho_0$, if $w^1, w^2 \in \cW_\rho$ satisfy
$$
\norm{w_1^1}_{\rho} + \norm{\pd_s w_1^1}_{\rho} + \norm{w_1^2}_{\rho} + \norm{\pd_s w_1^2}_{\rho} < \mathcal{C}_6 \quad \textup{and} \quad \mathcal{C}_7 := \norm{w_2^1}_\rho + \norm{w_2^2}_\rho < \infty\,,
$$
then
\begin{align}\label{bd K0}
  \left\|K_{0}[w^i] \right\|_\rho \lesssim 1 \,, \quad i \in \{1,2\}\,,
\end{align}
and
\begin{equation} \label{lipschitz K0}
\left\|K_{0}[w^1]-K_0[w^2] \right\|_\rho \lesssim \norm{w^1-w^2}_{\cW_{\rho}}\,.
\end{equation}
The analogous estimates also hold for $U_0^{\,\bn}$ and $U_0^{\,\bs}$.
\end{lemma}

Having these two results at hand, and also Propositions \ref{geo est} and \ref{main est}, it is straightforward to prove that for all $g$ we have
\begin{equation} \label{E.error1}
\norm{E_1^\ep(v^0,g)}_\rho \lesssim \norm{g}_{\cW_\rho}, 
\end{equation}
and that, for all $0 < \rho' < \rho$,
\begin{equation} \label{E.error2}
\|E_2^\ep\|_{\rho'} \lesssim \ep A(\rho-\rho')\,.
\end{equation}
Note that, here and during the rest of the proof, we omit the dependence on $t$ for the ease of writing. 

Likewise, combining \eqref{E.error1} and \eqref{E.error2} with Cauchy's integral theorem, we get that
$$
\norm{\pd_s E_1^\ep(v^0,g)}_{\rho'} \lesssim \frac{1}{\rho-\rho'} \norm{g}_{\cW_\rho},
$$
and that
$$
\|\pd_s E_2^\ep\|_{\rho'} \lesssim \ep A\Big( \frac{\rho-\rho'}{2} \Big)\,,
$$
for all $0 < \rho' < \rho$. Moreover, arguing similarly, we get that for all $g$ we have
\begin{equation} \label{E.error3}
\norm{E_3^\ep(v^0, g)}_{\rho'} \lesssim \frac{1}{\rho-\rho'} \norm{g}_{\cW_\rho}, \quad \textup{for all } 0 < \rho' < \rho\,.
\end{equation}

{

To estimate $E_4^\ep$, we first isolate the contribution of the
jump term. Set
\[
b_\ep:=
\frac{\widetilde e_\bs(\cdot,\bar\nu_\ep)}
     {|\widetilde e_\bs(\cdot,\bar\nu_\ep)|^2}
\cdot
\frac{\gamma_{\ep,0}'}{2|\gamma_{\ep,0}'|^2}.
\]
For $0<\rho'<\rho$, let $\sigma=(\rho+\rho')/2$.
Combining Proposition~\ref{Prop K conv} with
\eqref{E.nubarnul}, the Lipschitz estimate for the tangential
projection in Proposition~\ref{geo est}, and the bound
\eqref{bd K0}, we obtain
\[
U_\ep^\bs(w^\ep)(l)-U_0^\bs(\bar w^\ep)
=
b_\ep\int_{-1}^1
\operatorname{sgn}(\mu-l)\varpi_{\ep,\mu}\,d\mu
+r_{\ep,l}\,,
\qquad
\sup_{|l|\leq1}\|r_{\ep,l}\|_\sigma
\lesssim\ep A(\rho-\sigma)\,.
\]
for some error $r_{\eps,l}$. Here, the error introduced by replacing the tangential projection
at $\nu_{\ep,l}$ by that at $\bar\nu_\ep$ is included in
$r_{\ep,l}$. Since
\[
\int_{-1}^1\int_{-1}^1
\operatorname{sgn}(\mu-l)
\varpi_{\ep,\mu}\varpi_{\ep,l}\,\dd\mu\,\dd l=0\,,
\]
by antisymmetry, it follows that
$E_4^\ep=-\partial_s\int_{-1}^1r_{\ep,l}\varpi_{\ep,l}\,\dd l$.
Cauchy's integral formula therefore gives
\begin{equation}\label{E.error4}
\|E_4^\ep\|_{\rho'}
\lesssim\ep A\left(\frac{\rho-\rho'}2\right),
\end{equation}
after enlarging $A$ to absorb the derivative loss.

Once we have the estimates for $E_j^\ep$, namely \eqref{E.error1}, \eqref{E.error2}, \eqref{E.error3} and \eqref{E.error4}, the estimate \eqref{quant conv} in (v) follows from a Gronwall-type argument. See, for instance, \cite[Sect. 3]{BP} or \cite[Sect. 10]{Caflisch} for two slightly different approaches to the same conclusion. 

For the benefit of the reader, we provide a sketch of the proof. Fix
$\rho<R<S<\rho_0$. Step~2 and the same application of
Theorem~\ref{Nishida} to \eqref{ev pi 2}, using
Lemma~\ref{lem K0}, give uniform bounds for both solutions
on $\T_S$. Decreasing $\ep_0$ and $T$ if necessary, we may
assume that the geometric smallness conditions hold throughout
this interval. Write $d:=v^0-\bar w^\ep$ and set {\begin{align}
L(t)(g):=\left(-E_1^\eps(v^0,g,t),\, -E_3^\eps(v^0,g,t)\right)
\end{align}}
The difference equations \eqref{E.diff-nu0} and \eqref{E.diff-varpi0} can then be written as
\[
\partial_t d=L(t)d+f_\ep,\qquad
f_\ep:=(-E_2^\ep,-E_4^\ep).
\]

\noindent Furthermore \eqref{E.error1} and \eqref{E.error3} give
\[
\|L(t)h\|_{\cW_{\sigma'}}
\leq\frac{C}{\sigma-\sigma'}\|h\|_{\cW_\sigma},
\qquad \rho<\sigma'<\sigma\leq R.
\]
The preceding error estimates, including the bound for
$\partial_sE_2^\ep$, yield
\[
\sup_{0\leq t<T}\|f_\ep(t)\|_{\cW_R}\lesssim\ep,
\qquad
\|d(0)\|_{\cW_R}\lesssim D_\ep,
\quad
D_\ep:=
\ep+\left\|
\int_{-1}^1\varpi_{\ep,l}^0\,dl-\varpi_0^0
\right\|_{\rho_0},
\]
where the constants may depend on the fixed strip sizes $R,S,\rho_0$, but
are independent of $\ep$. Consequently, for $\rho\leq\sigma'<\sigma\leq R$
and $0\leq t\leq T$,
\[
\|d(t)\|_{\cW_{\sigma'}}
\leq C D_\ep+
\frac{C}{\sigma-\sigma'}
\int_0^t\|d(\tau)\|_{\cW_\sigma}\,d\tau.
\]

Now, choose $a>0$ with $a(R-\rho)<T$, and define
\[
M_a:=
\sup_{\substack{\rho\leq\sigma<R\\
                 0\leq t<a(R-\sigma)}}
\sqrt{R-\sigma-t/a}\,\|d(t)\|_{\cW_\sigma}.
\]
This quantity is finite, since $d$ is uniformly bounded
in $\cW_R$. For $\rho\leq\sigma<R$ and
$0\leq\tau\leq t<a(R-\sigma)$, set
\[
\sigma_\tau:=\frac{R+\sigma-\tau/a}{2},
\]
and observe that
\[
\sigma_\tau-\sigma
=R-\sigma_\tau-\tau/a
=\frac{R-\sigma-\tau/a}{2}.
\]
Using the integral equation for $d$ and estimating
$L(\tau)(d(\tau))$ between the radii $\sigma_\tau$ and
$\sigma$, we obtain
\[
\|d(t)\|_{\cW_\sigma}
\leq C D_\ep+
2\sqrt{2}\,C M_a
\int_0^t(R-\sigma-\tau/a)^{-3/2}\,d\tau.
\]
Since
\[
\int_0^t(R-\sigma-\tau/a)^{-3/2}\,d\tau
=
2a\left(
\frac{1}{\sqrt{R-\sigma-t/a}}
-\frac{1}{\sqrt{R-\sigma}}
\right),
\]
multiplying by $\sqrt{R-\sigma-t/a}$ and taking the
supremum gives
\[
M_a\leq C D_\ep\sqrt{R-\rho}
+4\sqrt{2}\,CaM_a.
\]
Taking $a$ sufficiently small, independently of $\ep$,
we absorb the last term. Hence, with $T=a(R-\rho)/2$,
\[
\sup_{0\leq t\leq T}\|d(t)\|_{\cW_\rho}
\lesssim D_\ep.
\]
Together with \eqref{E.nubarnul}, this proves
\eqref{quant conv}, uniformly in $l$ and $t$.
}

\begin{remark}
    If we do not assume \eqref{E.convergenceInitialdata}, arguing as in \cite[Sect. 2]{BP}, it is possible to prove a weaker {(not quantitative)} version of Theorem \ref{main thm} (v) in a much simpler way. Indeed, it is straightforward to prove that, up to a subsequence if necessary, for all $0 < \rho' < \rho$,
   $$
   \nu_{\ep,l}^0 \to 0 \quad \textup{and} \quad \int_{-1}^1 \varpi_{\ep,l}^0 \dd l \to \varpi_0^0 \quad \textup{in } \cX_{\rho'}, \quad \textup{as } \ep \to 0^{+}\,.
   $$
    However, even without assuming \eqref{E.convergenceInitialdata}, the preceding proof would give a quantitative bound up to a subsequence. Indeed, the {compactness from Lemma \ref{lem lsc} gives} a subsequence $\ep_{j}\to0$ {for which}
\[
\int_{-1}^{1}\varpi_{\ep_{j},l}^{0}\,dl\to\varpi_{0}^{0}
\quad\hbox{in }\cX_{\rho'}\,, \quad \textup{for every fixed } 0 < \rho' < \rho_0\,.
\]
{Also,} $\sup_{|l| \leq 1}\|\nu_{\ep,l}^0\|_{\rho_{0}}=O(\ep)$ {directly from its definition.}
One can also obtain from compactness, \eqref{E.nubarnul} and the bounds on the time derivative that the variables $\nu_{\eps,l}(\cdot,t)$ converge to some $\nu_0$ and $\int_{-1}^1\varpi_{\eps_j,l}$ converges to some $\varpi_0$ both in $L^\infty([0,T],\cX_{\rho'})$ up to potentially lowering $T$ and further restricting the subsequence. The fact that this limit is the solution to Birkhoff-Rott can be easily obtained again by compactness. 
\end{remark} 

\section{Vortex sheet-type equations} \label{S.vortexSheets}

This section is devoted to the proof {of} Proposition \ref{P.Evolution}. We keep here the notation introduced in Sections \ref{S.setup} and \ref{S.proofMain}.

\begin{proof}[Proof of Proposition \ref{P.Evolution}] We derive each equation separately.

\medbreak
\noindent \textbf{a)} Let $X_t$ be the particle-trajectory map at time $t$ associated to the velocity field $u_\ep$. For simplicity, here we assume that $u_\ep(\cdot,t)$ is $C^1(\R^2)$. {For bounded vorticity, the integrated identities below can be justified using smooth approximations with uniform} $L^{1}\cap L^{\infty}$ {bounds} for {fixed} $\ep${. Their velocities have} a {common} log-Lipschitz {modulus, and the associated area-preserving flows converge uniformly on compact time intervals. The transport identity therefore passes to bounded measurable test functions by approximation.}

Taking into account \eqref{E.sgraph} and  \eqref{def gamma evolved}, we get that
\begin{align*}
  & \frac{\pd}{\pd t}\, \big( \ga_{\ep,l}(s,t) \big) = \frac{\pd}{\pd t}\, \Big( X_t \big(\ga_{\ep,l}^0(\si_t(s))\big) \Big) \\
  & \qquad = u_\ep (\gamma_{\ep,l}(s,t),t) + \pd_t \si_t(s) \, \Big( {\rm D}X_t(\ga_{\ep,l}^0(\si_t(s)))\dot{\ga}_{\ep,l}^0(\si_t(s)) \Big) \\
  & \qquad = \tu(s,\nu_{\ep,l}(s,t),t) + \pd_t \si_t(s) \, \Big( {\rm D}X_t(\ga_{\ep,l}^0(\si_t(s))) \dot{\ga}_{\ep,l}^0(\si_t(s)) \Big)\,,
\end{align*}
and so that
\begin{equation} \label{E.pdtsigma}
\begin{aligned}
    0 & = \frac{\pd}{\pd t} \big(\bs(\ga_{\ep,l}(s,t))\big) = \nabla \bs (\ga_{\ep,l}(s,t)) \cdot  \frac{\pd}{\pd t}\, \big( \ga_{\ep,l}(s,t) \big)\\
    & = \tu^{\, \bs}(s,\nu_{\ep,l}(s,t),t) +  \pd_t \si_t(s) \, \Big( {\rm D}X_t(\ga_{\ep,l}^0(\si_t(s)))\dot{\ga}_{\ep,l}^0(\si_t(s)) \Big) \cdot \frac{\tes(s,\nu_{\ep,l}(s,t))}{|\tes(s,\nu_{\ep,l}(s,t))|^2}\,.
\end{aligned}
\end{equation}
Note that here we are using the identity $\bs(\ga_{\ep,l}(s,t)) = s$. Likewise, we have that 
\begin{align*}
     & \frac{\pd}{\pd s}\, \big( \ga_{\ep,l}(s,t) \big) = \frac{\pd}{\pd s}\, \Big( X_t \big(\ga_{\ep,l}^0(\si_t(s))\big) \Big) = \pd_s \si_t(s) \, \Big( {\rm D}X_t(\ga_{\ep,l}^0(\si_t(s))) \dot{\ga}_{\ep,l}^0(\si_t(s)) \Big)\,,
\end{align*}
and so the identity $\bs(\gamma_{\eps,l}(s,t))=s$ yields 
\begin{equation} \label{E.pdssigma}
\begin{aligned}
    1 & =  \frac{\pd}{\pd s} \big(\bs(\ga_{\ep,l}(s,t))\big) =   \frac{\pd}{\pd s}\, \big( \ga_{\ep,l}(s,t) \big) \cdot \nabla \bs (\ga_{\ep,l}(s,t)) \\
    & = \pd_s \si_t(s) \, \Big( {\rm D}X_t(\ga_{\ep,l}^0(\si_t(s)))\dot{\ga}_{\ep,l}^0(\si_t(s)) \Big) \cdot \frac{\tes(s,\nu_{\ep,l}(s,t))}{|\tes(s,\nu_{\ep,l}(s,t))|^2}\,.
\end{aligned}
\end{equation}
Rearranging \eqref{E.pdtsigma} and \eqref{E.pdssigma}, we see that \begin{equation}
\de_t\sigma_t(s)+\tu^{\,\bs}(s,\nu_{\eps,l}(s,t),t)\de_s\sigma_t(s)=0\,.\label{pde sigma}
\end{equation}

 Finally, using {\eqref{def nu eps l}}, we see that
\begin{equation}\begin{aligned}
    & \pd_t \nu_{\ep,l}(s,t) = \frac{\pd}{\pd t} \big( \bn (\ga_{\ep,l}(s,t)) \big) = \frac{\pd}{\pd t} \big( \gamma_{\ep,l}(s,t) \big) \cdot \nabla \bn (\ga_{\ep,l}(s,t))  \\
    & \qquad = \bigg( \tu(s,\nu_{\ep,l}(s,t),t) + \pd_t \si_t(s) \, \Big( {\rm D}X_t(\ga_{\ep,l}^0(\si_t(s))) \dot{\ga}_{\ep,l}^0(\si_t(s)) \Big) \bigg)\cdot \ten(s) \\
    & \qquad = \tu^{\, \bn}(s,\nu_{\ep,l}(s,t),t) + \pd_t \si_t(s) \Big( {\rm D}X_t(\ga_{\ep,l}^0(\si_t(s))) \dot{\ga}_{\ep,l}^0(\si_t(s)) \Big) \cdot \ten(s)\,,\label{detnu} 
\end{aligned}\end{equation}
and that
\begin{equation}\begin{aligned}
    & \pd_s \nu_{\ep,l}(s,t) = \frac{\pd}{\pd s} \big( \bn (\ga_{\ep,l}(s,t)) \big) = \frac{\pd}{\pd s} \big( \ga_{\ep,l}(s,t) \big) \cdot \nabla \bn(\ga_{\ep,l}(s,t))\\
    & \qquad = \pd_s \si_t(s) \, \Big( {\rm D}X_t(\ga_{\ep,l}^0(\si_t(s))) \dot{\ga}_{\ep,l}^0(\si_t(s)) \Big) \cdot \ten(s)\,.\label{desnu}
\end{aligned}\end{equation}
Now, from those two equations and \eqref{pde sigma} we obtain that \begin{align*}
&\de_t\nu_{\eps,l}(s,t)+\tu^{\,\bs}(s,\nu_{\eps,l}(s,t),t)\de_s\nu_{\eps,l}(s,t)\\
&\quad=\tu^{\,\bn}(s,\nu_{\eps,l}(s,t),t)+\bigl(\de_t\sigma_t(s)+\tu^{\,\bs}(s,\nu_{\eps,l}(s,t),t)\de_s\sigma_t(s)\bigr)\left(\left(\mathrm{D}X_t(\gamma_{\eps,l}^0(\sigma_t(s)))\dot{\gamma}_{\eps,l}^0(\sigma_t(s))\right)\cdot\tilde{e}_{\bn}(s)\right)\\
&\quad=\tu^{\,\bn}(s,\nu_{\eps,l}(s,t),t)\,,
\end{align*}
which is exactly \eqref{ev gamma}.

\medbreak
\noindent \textbf{b)} Let $l, l' \in [-1,1]$ be fixed but arbitrary. By \eqref{ev gamma} and direct computations, we get that
\begin{align*}
    & \pd_t \big( \eta_{\ep,l'}(s,t) - \eta_{\ep,l}(s,t)) = \frac{1}{\ep} \pd_t \Big( \nu_{\ep,l'}(s,t)- \nu_{\ep,l}(s,t) \Big)  = \frac{1}{\ep} \big( \tu^{\, \bn}(s,\nu_{\ep,l'}(s,t),t) - \tu^{\, \bn}(s,\nu_{\ep,l}(s,t),t) \big)  \\
    & \qquad- \frac{1}{\ep} \Big( \tu^{\, \bs}(s,\nu_{\ep,l'}(s,t),t) \pd_s \nu_{\eps,l'}(s,t) - \tu^{\, \bs}(s,\nu_{\ep,l}(s,t),t) \pd_s \nu_{\eps,l}(s,t) \Big)\,.
\end{align*}
Now, on one hand, we have that
\begin{align*}
    & \frac{\pd}{\pd s} \left( \int_{\nu_{\ep,l}(s,t)}^{\nu_{\ep,l'}(s,t)} \tu^{\, \bs} (s,\tau,t) \, \dd\tau \right)\\
    & \quad = \int_{\nu_{\ep,l}(s,t)}^{\nu_{\ep,l'}(s,t)} \frac{\pd}{\pd s} \big(  \tu^{\, \bs} (s,\tau,t)  \big) \dd \tau + \tu^{\, \bs}(s,\nu_{\ep,l'}(s,t),t) \pd_s \nu_{\eps,l'}(s,t)  -  \tu^{\, \bs}(s,\nu_{\ep,l}(s,t),t) \pd_s \nu_{\eps,l}(s,t) \,,
\end{align*}
and on the other hand, that
\begin{align*}
    \int_{\nu_{\ep,l}(s,t)}^{\nu_{\ep,l'}(s,t)} \frac{\pd}{\pd \tau} \big (\tu^{\, \bn}(s,\tau,t) \big) \dd \tau = \tu^{\, \bn}(s,\nu_{\ep,l'}(s,t),t) - \tu^{\,\bn}(s,\nu_{\ep,l}(s,t),t)\,.
\end{align*}
Hence, it follows that
\begin{align*}
    & \pd_t \big( \eta_{\ep,l'}(s,t) - \eta_{\ep,l}(s,t)) \\
    & \qquad = - \frac{1}{\ep} \frac{\pd}{\pd s} \left( \int_{\nu_{\ep,l}(s,t)}^{\nu_{\ep,l'}(s,t)} \tu^{\, \bs} (s,\tau,t) \, \dd\tau \right)  + \frac{1}{\ep} \int_{\nu_{\ep,l}(s,t)}^{\nu_{\ep,l'}(s,t)} \bigg[  \frac{\pd}{\pd s}  \big( \tu^{\, \bs} (s,\tau,t) \big) + \frac{\pd}{\pd \tau}  \big( \tu^{\, \bn} (s,\tau,t) \big) \bigg] \dd \tau\,.
\end{align*}
Also, using the incompressibility of the velocity field $u_\ep$, we infer that
\begin{align*}
   & \frac{\pd}{\pd s}  \big( \tu^{\, \bs} (s,\tau,t)  \big) + \frac{\pd}{\pd \tau}  \big( \tu^{\, \bn} (s,\tau,t)  \big) \\
   & \qquad = \pd_s \tu(s,\tau,t) \cdot  \frac{\tes(s,\tau)}{|\tes(s,\tau)|^2} + \pd_\tau \tu(s,\tau,t) \cdot \ten(s) + \ka(s,\tau) \cdot \tu(s,\tau,t) \\
   & \qquad = \div_x u_\ep (\bx(s,\tau),t) + \ka(s,\tau) \cdot \tu(s,\tau,t) = \ka(s,\tau) \cdot \tu(s,\tau,t)\,,
\end{align*}
and therefore
\begin{align*}
    & \pd_t \big( \eta_{\ep,l'}(s,t) - \eta_{\ep,l}(s,t)) \\
    & \quad = - \frac{1}{\ep} \frac{\pd}{\pd s} \left( \int_{\nu_{\ep,l}(s,t)}^{\nu_{\ep,l'}(s,t)} \tu^{\, \bs} (s,\tau,t) \, \dd\tau \right) + \frac{1}{\ep} \int_{\nu_{\ep,l}(s,t)}^{\nu_{\ep,l'}(s,t)} \ka(s,\tau) \cdot \tu(s,\tau,t)\, \dd \tau\,.
\end{align*}
Next, observe that $\pd_l \nu_{\ep,l}= \ep (1+\pd_l \eta_{\ep,l})$. In particular, the map $l \mapsto \nu_{\ep,l}(s,t)$ is invertible whenever $\|\pd_l \eta_{\ep,l}\|_{L^{\infty}(\TT \times [0,T))} < 1$. Thus, doing a change of variables in each integral, we get that
\begin{align*}
     &\pd_t \big( \eta_{\ep,l'}(s,t) - \eta_{\ep,l}(s,t))  = - \frac{\pd}{\pd s} \left( \int_l^{l'}  \tu^{\, \bs}(s,\nu_{\ep,\tau}(s,t),t) \big(1+ \pd_\tau \eta_{\ep,\tau}(s,t) \big)\, \dd \tau \right) \\
     & \qquad + \int_l^{l'} \big( \ka(s,\nu_{\ep,\tau}(s,t)) \cdot \tu(s,\nu_{\ep,\tau}(s,t),t)  \big)\big(1+ \pd_\tau \eta_{\ep,\tau}(s,t) \big)\, \dd \tau\,,
\end{align*}
for all $l,l' \in [-1,1]$. Having this identity at hand, \eqref{ev eta} follows {by} dividing both sides by $l'-l$ with $l \neq l'$, and then sending $l'-l$ to $0$ and using the Lebesgue differentiation theorem on each integral of the right hand side. 

\medbreak
\noindent \textbf{c)} Let $X_t$ be the particle-trajectory map at time $t$ associated with the velocity field $u_\ep$, and consider a test function $\Phi\in L^\infty (\R^2)$ of the form \begin{equation}
\Phi(x):=\mathds{1}_{A(a,b,l_1,l_2)}(x)\,,
\end{equation}
with
\begin{equation}
A(a,b,l_1,l_2):=\bigcup_{l\in [l_1,l_2]}\bigl\{\gamma_{\eps,l}^0(s)\, |\, s\in [a,b]\bigr\}\,.
\end{equation}
Here, $\mathds{1}_A$ denotes the indicator function of a subset $A \subset \R^2$. By the incompressibility of the flow, 
\begin{align}
\int_{\R^2} \omega_\ep(x,t)(\Phi\circ X_t^{-1})(x)\dx=\int_{\R^2} \omega_\ep^0(x)\Phi(x)\dx\,.
\end{align}
Although \eqref{def Phi om ep} was written for continuous test functions, both sides define finite Radon measures; equality for continuous tests implies equality for bounded Borel test functions. Thus, \eqref{def Phi om ep} can be used for this $\Phi$.

Also, it clearly holds that 
\begin{equation}
\Phi\circ X_t^{-1}=\mathds{1}_{X_t(A(a,b,l_1,l_2))}\,,
\end{equation}
and that
\begin{equation}
X_t(A(a,b,l_1,l_2))=\bigcup_{l\in [l_1,l_2]}X_t\left(\bigl\{\gamma_{\eps,l}^0(s)\, |\, s\in [a,b]\bigr\}\right)=\bigcup_{l\in [l_1,l_2]}\bigl\{X_t(\gamma_{\eps,l}^0(s))\, |\, s\in [a,b]\bigr\}\,.
\end{equation}

 Now, using that the curves $\gamma_{\eps,l}(s,t)$ are transported by the flow by definition, and using the reparametrization map $\sigma$, defined in \eqref{E.sgraph} and \eqref{def gamma evolved}, we can further rewrite this set as \begin{equation}
\bigcup_{l\in [l_1,l_2]}\bigl\{X_t(\gamma_{\eps,l}^0(s))\, |\, s\in [a,b]\bigr\}=\bigcup_{l\in [l_1,l_2]}\bigl\{\gamma_{\eps,l}(\sigma_t^{-1}(s),t)\, |\, s\in [a,b]\bigr\}\,,
\end{equation}
where the inverse of $\sigma_t$ is taken with respect to the variable $s$, and $\sigma_t$ is invertible by the assumption that the curves are graphs over $\Gamma$, as explained in \eqref{E.sgraph}--\eqref{def gamma evolved} above. The function $\sigma_t^{-1}$ is monotone (otherwise it cannot be invertible) and hence maps $[a,b]$ to $[\sigma_t^{-1}(a),\sigma_t^{-1}(b)]$ bijectively. Thus, we have that 
\begin{equation}
\bigcup_{l\in [l_1,l_2]}\bigl\{\gamma_{\eps,l}(\sigma_t^{-1}(s),t)\, |\, s\in [a,b]\bigr\}=\bigcup_{l\in [l_1,l_2]}\bigl\{\gamma_{\eps,l}(\tilde{s},t)\, |\, \tilde{s}\in [\sigma_t^{-1}(a),\sigma_t^{-1}(b)]\bigr\}\,.
\end{equation}
In sum, we have obtained that \begin{equation}
\int_{\bigcup_{l\in [l_1,l_2]}\bigl\{\gamma_{\eps,l}(s,t)\, |\, s\in [\sigma_t^{-1}(a),\sigma_t^{-1}(b)]\bigr\}}\omega_\ep(x,t)\dx=\int_{\bigcup_{l\in [l_1,l_2]}\bigl\{\gamma_{\eps,l}^0(s)\, |\, s\in [a,b]\bigr\}}\omega_\ep^0(x)\dx\,.
\end{equation}
Rewriting this equation in terms of the densities $\varpi_{\ep,l}$ and using the definitions \eqref{E.ansatzomega} and \eqref{E.ansatzomegabis}, this equation reads as 
{
\begin{equation}
\int_{l_1}^{l_2}\int_{\sigma_t^{-1}(a)}^{\sigma_t^{-1}(b)}\varpi_{\eps,l}(s,t)\dd s\dd l= \int_{l_1}^{l_2} \int_{a}^{b}\varpi_{\eps,l}^0(s)\dd s\dd l\,.
\end{equation}
Dividing by $l_2- l_1$, sending $l_2-l_1\rightarrow 0^{+}$  and using the Lebesgue differentiation theorem, we get 
 that\begin{equation}
\int_{\sigma_t^{-1}(a)}^{\sigma_t^{-1}(b)} \varpi_{\eps,l}(s,t)\dd s=\int_{a}^{b} \varpi_{\eps,l}^0(s)\dd s\,,
\end{equation}
where $l$ denotes the common limit of $l_1$ and $l_2$.} In other words, the integral on the left hand side is a conserved quantity for every $l\in [-1,1]$ and for all $a,b\in \T$. Using the Reynolds transport theorem and the fundamental theorem of calculus, we may rewrite this fact as \begin{align}
0&=\int_{\sigma_t^{-1}(a)}^{\sigma_t^{-1}(b)}\de_t\varpi_{\eps,l}(s,t)\dd s+\varpi_{\eps,l}(\sigma_t^{-1}(b),t)\de_t\sigma_t^{-1}(b)-\varpi_{\eps,l}(\sigma_t^{-1}(a),t)\de_t\sigma_t^{-1}(a)\\
&=\int_{\sigma_t^{-1}(a)}^{\sigma_t^{-1}(b)} \left(\de_t\varpi_{\eps,l}(s,t)+\de_t\sigma_t^{-1}(b)\de_s\varpi_{\eps,l}(s,t) \right) \dd s+\varpi_{\eps,l}(\sigma_t^{-1}(a),t)\de_t(\sigma_t^{-1}(b)-\sigma_t^{-1}(a))\,.
\end{align} 

 Finally, to derive \eqref{ev pi} from this we need to compute $\de_t\sigma_t^{-1}$. By definition (see \eqref{E.sgraph}), it holds that $\bs(X_t(\gamma_{\eps,l}^0(s)))=\sigma_t^{-1}(s)$, differentiating this identity in $t$ and using \eqref{E.usun} yields \begin{equation}
\de_t\sigma_t^{-1}(s)=\nabla \bs(X_t(\gamma_{\eps,l}^0(s)))\cdot u_\ep(X_t(\gamma_{\eps,l}^0(s)),t)=\tu^{\, \bs} (\sigma_t^{-1}(s),\nu_{\eps,l}(\sigma_t^{-1}(s),t),t)\,.
\end{equation}

 We can then let $a-b\rightarrow 0$ and see that
\begin{align}
0&=\lim_{a-b\rightarrow 0}\frac{1}{\sigma_t^{-1}(b)-\sigma_t^{-1}(a)}\,\biggl(\int_{\sigma_t^{-1}(a)}^{\sigma_t^{-1}(b)} \left(\de_t\varpi_{\eps,l}(s,t)+\de_t\sigma_t^{-1}(b)\de_s\varpi_{\eps,l}(s,t) \right)\dd s\\
&\quad+\varpi_{\eps,l}(\sigma_t^{-1}(a),t)\de_t(\sigma_t^{-1}(b)-\sigma_t^{-1}(a))\biggr)\\
&=\de_t\varpi_{\eps,l}(\tilde{s},t)+\tu^{\,\bs}(\tilde{s},\nu_{\eps,l}(\tilde{s},t),t)\de_s\varpi_{\eps,l}(\tilde{s},t)\\
&\quad+\varpi_{\eps,l}(\tilde{s},t)\lim_{a-b\rightarrow 0}\frac{\tu^{\, \bs} (\sigma_t^{-1}(b),\nu_{\eps,l}(\sigma_t^{-1}(b),t),t)-\tu^{\, \bs} (\sigma_t^{-1}(a),\nu_{\eps,l}(\sigma_t^{-1}(a),t),t)}{\sigma_t^{-1}(b)-\sigma_t^{-1}(a)}\\
&=\de_t\varpi_{\eps,l}(\tilde{s},t)+\de_s(\tu^{\,\bs}(\tilde{s},\nu_{\eps,l}(\tilde{s},t),t)\varpi_{\eps,l}(\tilde{s},t)),
\end{align}
where $\tilde{s}$ denotes the common limit of $\sigma_t^{-1}(b)$ and $\sigma_t^{-1}(a)$.
\end{proof}

\section{Geometric estimates} \label{S.Geo}

This section is devoted to proving Proposition \ref{geo est}. We begin by introducing some notation for complexifications which is used in this and the next section.

For $z=(z_1,z_2)\in\CC^2$, we denote its modulus by
\[
|z|:=\sqrt{|z_1|^2+|z_2|^2}\in[0,\infty)\,.
\]
We set \begin{align*}
    \mathbb{H}:=\{z\in\CC^2:\Re(z_{1}^{2}+z_{2}^{2})>0\},
\end{align*}
and, when $z\in \mathbb{H}$, we use the main branch of the square root on $\C\backslash (-\infty,0)$ to define
\begin{align} \label{E.complexificationModulus}
&|z|_{\C}:=\sqrt{z_1^2+z_2^2}\in \C\,,
\end{align}
and denote the corresponding quadratic form by \begin{align}
\scalar{z}{z'}_\C:=z_1z_1'+z_2z_2'\in \C\,.
\end{align}
Of course, \eqref{E.complexificationModulus} is not a norm. A direct calculation shows that \begin{align} \label{E.complex-modulus}
||z|_{\C}|\leq |z|\,.
\end{align}

We require a preliminary technical lemma whose proof is a straightforward consequence of the fundamental theorem of calculus. 

\begin{lemma} \label{L.doubleDifference}
    Let $B\subset \C^2$ be convex, and let $g: B \subset \C^2 \to \C$ be a $C^2$-function. Then, for all $a,b,c,d \in B$, it follows that
    $$
    |g(a)-g(b)-g(c)+g(d)| \leq \|{\rm D}g \|_{L^{\infty}(B)} |a-b+d-c|+ \|{\rm D}^2 g\|_{L^{\infty}(B)} \big(|a-b|+|c-d| \big)\big(|a-c| + |b-d| \big)\,.
    $$
\end{lemma}

We now derive the key result of this {section}, from which Proposition \ref{geo est} will follow.

\begin{lemma} \label{L.geoEst}
    Let $g: \T_\rho \times B \to \C$ be a holomorphic function, where $B\subset\C$ is a neighborhood of $[-R_1,R_1]$, with $\rho,\, R_1 > 0$ fixed. Assume that $g(\T\times[-R_1,R_1])\subset\R$ and that $\norm{g}_{C^2(\T_\rho\times B)}<\infty$. Then, there exists $C > 0$, depending only on $g$ and~$B$, such that, if $w^j\in \cX_\rho$ are functions bounded as
$$
  \norm{w^j}_{\rho}   < C\,,
$$
with $j=1,2$, {then} $g(\cdot,w^j(\cdot))\in\cX_\rho$. Moreover, they can be estimated as
\begin{align} \label{E.gBounded}
 \norm{g(\cdot,w^j(\cdot))}_\rho &\lesssim 1\,, \\
\label{E.gLipschitz}
\norm{g(\cdot,w^1(\cdot))-g(\cdot,w^2(\cdot))}_\rho &\lesssim \,  \norm{w^1 - w^2}_\rho\,,
\end{align}
with implicit constants that depend only on~$C$, $g$ and~$B$.
\end{lemma}

\begin{remark}
The assumptions in the previous result can be slightly relaxed. We need that
$$
\sup_{j\in\{1,2\}}\sup_{|\beta|\leq\rho}\norm{w^j(\cdot+i\beta)}_{L^{\infty}(\TT)}<C_1\,,
$$
for some sufficiently small $C_1>0$, and that the $C^{\frac12}$-seminorm is finite, i.e.\
$$
C_2:=\sup_{j\in\{1,2\}}\sup_{|\beta|\leq\rho}[w^j(\cdot+i\beta)]_{C^{1/2}(\TT)}<\infty\,.
$$
In this case, the implicit constants would depend on both $C_1$ and on (an upper bound of) $C_2$. 
\end{remark}

\begin{proof}
{Choose} $d>0$ {such} that {the closed complex} $d${-neighborhood of} $[-R_1,R_1]$ {is contained in} $B$, {and take} $C<\min(d,R_{1})/2${. Then} $w^j(z)$ {lies in a fixed compact convex subset} $B_{*}\Subset B${. The compositions} are holomorphic {in the open strip, continuous on its closure and real on} $\TT$. {We work on} the {universal cover of the strip}, {choosing the shortest representatives of horizontal differences. All derivatives of} $g$ {below are bounded on} $\TT_{\rho}\times B_{*}${.} In the rest of the proof, $M>0$ denotes a constant depending only on $\|g\|_{C^2(\TT_\rho\times B_*)}$, which may change from one estimate to another.

{The mean value theorem gives}
\[
|g(z_{1},w^{j}(z_{1}))-g(z_{2},w^{j}(z_{2}))|
\leq M\big(|z_{1}-z_{2}|+|w^{j}(z_{1})-w^{j}(z_{2})|\big),
\]
{where} $z_{i}=\alpha_{i}+i\beta${. This proves}~{\eqref{E.gBounded}}. {For} the {difference} of {two compositions}, { we first note that}
\begin{equation}\label{E.gLipschitz1}
\sup_{|\beta|\leq\rho}\|g(\cdot+i\beta,w^{1})-g(\cdot+i\beta,w^{2})\|_{L^{\infty}}
\leq M\|w^1-w^2\|_{\rho}.
\end{equation}
{Apply} Lemma~{\ref{L.doubleDifference} to} $(z_{1},w^1(z_{1}))${,} $(z_{1},w^2(z_{1}))${,} $(z_{2},w^1(z_{2}))$ {and} $(z_{2},w^2(z_{2}))${. With} $h=w^{1}-w^2${, the absolute value of the double difference is at most}
\[
\begin{split}
M|h(z_{1})-h(z_2)|+M\big(|h(z_{1})|+|h(z_{2})|\big)
\big(2|\alpha_{1}-\alpha_2|+\sum_{j=1}^{2}|w^j(z_{1})-w^j(z_{2})|\big).
\end{split}
\]
{Dividing by} $|\alpha_{1}-\alpha_{2}|^{1/2}$ {and taking the suprema gives}
\begin{equation}\label{E.gLipschitz2}
\sup_{|\beta|\leq\rho}
[g(\cdot+i\beta,w^{1})-g(\cdot+i\beta,w^{2})]_{C^{1/2}(\TT)}
\leq M\big(1+\|w^1\|_{\rho}+\|w^2\|_\rho\big)\|h\|_{\rho}.
\end{equation}
{Together these prove}~\eqref{E.gLipschitz} and {the asserted membership}.

\end{proof}

The second ingredient that we need to prove Proposition \ref{geo est} is the following lemma. Here, we use a fixed reference curve $\Ga \in \cX_{\rho_0}^2$ for some $\rho_0 > 0$ as in Theorem \ref{main thm} with $\Ga^{(5)} \in \cX_{\rho_0}^2$.
 
\begin{lemma}\label{L.real analytic}
There exists $c_0 > 0$ such that $\tes$ and $\kappa$ are real analytic in the set $\T \times (-c_0,c_0)$. Furthermore, $|\tes| \geq \frac12$ in $\T \times (-c_0,c_0)$ and, for every $n \in (-c_0,c_0)$ we have $\tes(\cdot,n),\, \kappa(\cdot,n) \in \cX^2_{\rho_0}$. 
\end{lemma}

\begin{proof}
{Put} $k(z)=\scalar{\ddot{\Gamma}(z)^\perp}{\dot{\Gamma}(z)}${. The arclength identity extends as} $\scalar{\dot{\Gamma}(z)}{\dot{\Gamma}(z)}_{\C}=1$, {and differentiation gives} $\scalar{\dot{\Gamma}}{{\ddot\Gamma}}_{\C}{=}0${. Hence, using the complex extensions introduced in Section \ref{S.setup}, we see that as in \eqref{expr tang} we have}
\[
\widetilde e_{\bs}(z,n)=(1+nk(z))\dot\Gamma(z),\quad
\scalar{\widetilde{e}_s}{\widetilde{e}_s}_{\C}=(1+nk(z))^{2}.
\]
Recall that $\kappa$ was defined in \eqref{E.kappa} and that  its complexification is given by $\de_z(\widetilde{e}_s/\scalar{\widetilde{e}_s}{\widetilde{e}_s}_{\C})$.
{Choose} $c_{0}>0$ {such} that $2c_{0}\|k\|_{L^{\infty}(\TT_{\rho_{0}})}<1/2${. On the complex domain} $z\in\TT_{\rho_{0}}$, $|n|<2c_{0}${, the denominator is nonzero and}
\[
|1+nk(z)|\geq\tfrac12,\qquad
\frac{\widetilde e_\bs}{ \scalar{\widetilde e_\bs}{\widetilde e_\bs}}_{\C}
=\frac{\dot\Gamma(z)}{1+nk(z)}.
\]
{Both this function and its} $z$-{derivative} $\kappa$ {are holomorphic and have the required uniform} $C^{2}$ {bounds on smaller closed domains, since} $\Gamma^{(5)}\in\cX_{\rho_0}^2$. {Restricting to real} $n$ {proves} the {lemma.} 
\end{proof}

\begin{remark}
    This proof also shows that $\frac{\tes(\cdot,n)}{|\tes(\cdot,n)|^2} \in \cX_{\rho_0}^2$ for all fixed $n \in (-c_0,c_0)$, and that $\frac{\tes}{|\tes|^2}$ is real analytic in $\T \times (-c_0,c_0)$. Moreover, since $\ten(s) = \dot{\Ga}(s)^{\perp}$, we also have $\ten \in \cX_{\rho_0}^2$.
\end{remark}

We are now ready to prove Proposition \ref{geo est}.

\begin{proof}[Proof of Proposition \ref{geo est}]
We {Apply} Lemma~\ref{L.geoEst} {componentwise to} $g=\kappa$ and $g=\dot\Gamma/(1+nk)$ {on} $\TT_{\rho_{0}}\times\{|n|<2c_{0}\}${. The preceding proof gives a uniform nonvanishing denominator on that complex domain. Choosing} $\cC_{1}$ {so that the graphs take values in} $|n|<c_{0}$ {proves all the stated bounds uniformly for} $\rho\leq\rho_{0}$.
\end{proof}

\section{Kernel estimates} \label{S.Kernel}

In this section we prove Proposition \ref{main est}. Since the proof is rather long, we split it into several parts. In the next subsection, keeping the notation introduced in Section \ref{S.Geo}, we prove several technical lemmas which will be key to proving Proposition \ref{main est}.

\subsection{Preliminary lemmas} \label{S.preliminaryLemmas} We start with a technical result whose proof is just direct calculus. We provide some details for the benefit of the reader. 

\begin{lemma}\label{Lem double diff}
Let $p,q,r,t\in \mathbb{H}$, then \begin{equation}\begin{aligned}
&\left|\frac{p}{|p|_{\C}^2}-\frac{q}{|q|_{\C}^2}-\frac{r}{|r|_{\C}^2}+\frac{t}{|t|_{\C}^2}-\frac{p-q-r+t}{|p|_{\C}^2}-\frac{2p\scalar{p}{q-p+r-t}_{\C}}{|p|_{\C}^4}\right|\\
& \qquad \lesssim \frac{\max\{|p|,|q|,|r|,|t|\}^3}{\min\{||p|_{\C}|,||q|_{\C}|,||r|_{\C}|,||t|_{\C}|\}^6}\big(|p-q|+|r-t|\big)\big(|p-r|+|q-t|\big)\,.\label{dd est}
\end{aligned}\end{equation}
\end{lemma}

\begin{proof}
First, note that
\begin{equation} \label{2 dd est}
\begin{aligned}
&\frac{p}{|p|_{\C}^2}-\frac{q}{|q|_{\C}^2}-\frac{r}{|r|_{\C}^2}+\frac{t}{|t|_{\C}^2}=\frac{p-q}{|p|_{\C}^2}-\frac{r-t}{|r|_{\C}^2}+q\frac{|q|_{\C}^2-|p|_{\C}^2}{|p|_{\C}^2|q|_{\C}^2}-t\frac{|t|_{\C}^2-|r|_{\C}^2}{|r|_{\C}^2|t|_{\C}^2}\\
&=\frac{p-q-r+t}{|p|_{\C}^2}-(r-t)\frac{\scalar{p-r}{p+r}_{\C}}{|p|_{\C}^2|r|_{\C}^2}+q\frac{\scalar{q-p}{q+p}_{\C}}{|p|_\C^2|q|_{\C}^2}-t\frac{\scalar{t-r}{t+r}_{\C}}{|t|_\C^2|r|_{\C}^2}\,.
\end{aligned}
\end{equation}
 Now, observe that
\begin{align*}
    & t\frac{\scalar{t-r}{t+r}_{\C}}{|t|_\C^2|r|_{\C}^2} = q\frac{\scalar{t-r}{q+p}_{\C}}{|p|_\C^2|q|_{\C}^2} + q \scalar{t-r}{q+p}_{\C} \left( \frac{1}{|r|^2_\C |t|^2_\C}  - \frac{1}{|p|^2_\C |q|^2_\C} \right) \\
    & \qquad + q \frac{\scalar{t-r}{t-q+r-p}_\C }{|t|^2_\C |r|^2_\C} + (t-q) \frac{\scalar{t-r}{t+r}_\C}{|t|^2_\C |r|^2_\C}\,,
\end{align*}
and so that
\begin{align}
    & q\frac{\scalar{q-p}{q+p}_{\C}}{|p|_\C^2|q|_{\C}^2}-t\frac{\scalar{t-r}{t+r}_{\C}}{|t|_\C^2|r|_{\C}^2} =  q\frac{\scalar{q-p-t+r}{q+p}_{\C}} {|p|_\C^2|q|_{\C}^2} \\
    & \qquad -q \scalar{t-r}{q+p}_{\C} \left( \frac{1}{|r|^2_\C |t|^2_\C}  - \frac{1}{|p|^2_\C |q|^2_\C} \right) - q \frac{\scalar{t-r}{t-q+r-p}_\C }{|t|^2_\C |r|^2_\C} + (q-t) \frac{\scalar{t-r}{t+r}_\C}{|t|^2_\C |r|^2_\C}\,.
\end{align}
On the other hand, it follows that
\begin{align}
     & q\frac{\scalar{q-p-t+r}{q+p}_{\C}} {|p|_\C^2|q|_{\C}^2} = 2q \frac{\scalar{q-t-p+r}{p}_\C}{|p|^2_\C|q|^2_\C } + q \frac{\scalar{q-t-p+r}{q-p}_\C}{|p|^2_\C|q|^2_\C } \\
     & \quad = 2p \frac{\scalar{q-p-t+r}{p}_\C}{|p|_\C^4} + 2p \frac{\scalar{q-p-t+r}{p}_\C \scalar{p-q}{p+q}_\C}{|p|_\C^4 |q|_\C^2} \\
     & \qquad+ q \frac{\scalar{q-t-p+r}{q-p}_\C}{|p|^2_\C|q|^2_\C } + 2(q-p) \frac{\scalar{q-p-t+r}{p}_\C}{|p|^2_\C |q|^2_\C}\,.
\end{align}
Hence, we actually have that
\begin{align}
    & q\frac{\scalar{q-p}{q+p}_{\C}}{|p|_\C^2|q|_{\C}^2}-t\frac{\scalar{t-r}{t+r}_{\C}}{|t|_\C^2|r|_{\C}^2} = 2p \frac{\scalar{q-p-t+r}{p}_\C}{|p|_\C^4} + 2p \frac{\scalar{q-p-t+r}{p}_\C \scalar{p-q}{p+q}_\C}{|p|_\C^4 |q|_\C^2} \\
     & \qquad+ q \frac{\scalar{q-t-p+r}{q-p}_\C}{|p|^2_\C|q|^2_\C } + 2(q-p) \frac{\scalar{q-p-t+r}{p}_\C}{|p|^2_\C |q|^2_\C} \\
    & \qquad -q \scalar{t-r}{q+p}_{\C} \left( \frac{1}{|r|^2_\C |t|^2_\C}  - \frac{1}{|p|^2_\C |q|^2_\C} \right) - q \frac{\scalar{t-r}{t-q+r-p}_\C }{|t|^2_\C |r|^2_\C} + (q-t) \frac{\scalar{t-r}{t+r}_\C}{|t|^2_\C |r|^2_\C}\,.
\end{align}
Finally, we substitute everything into \eqref{2 dd est} and get that
\begin{align}
    & \frac{p}{|p|_{\C}^2}-\frac{q}{|q|_{\C}^2}-\frac{r}{|r|_{\C}^2}+\frac{t}{|t|_{\C}^2}-\frac{p-q-r+t}{|p|_{\C}^2}-\frac{2p\scalar{p}{q-p+r-t}_{\C}}{|p|_{\C}^4} \\  \label{E.pqrtComplete}
    & \quad= (t-r) \frac{\scalar{p-r}{p+r}_{\C}}{|p|_{\C}^2|r|_{\C}^2} + 2p \frac{\scalar{q-p-t+r}{p}_\C \scalar{p-q}{p+q}_\C}{|p|_\C^4 |q|_\C^2} \\
     & \qquad+ q \frac{\scalar{q-t-p+r}{q-p}_\C}{|p|^2_\C|q|^2_\C } + 2(q-p) \frac{\scalar{q-p-t+r}{p}_\C}{|p|^2_\C |q|^2_\C} \\
    & \qquad -q \scalar{t-r}{q+p}_{\C} \left( \frac{1}{|r|^2_\C |t|^2_\C}  - \frac{1}{|p|^2_\C |q|^2_\C} \right) - q \frac{\scalar{t-r}{t-q+r-p}_\C }{|t|^2_\C |r|^2_\C} + (q-t) \frac{\scalar{t-r}{t+r}_\C}{|t|^2_\C |r|^2_\C}\,.
\end{align}
Then, using \eqref{E.complex-modulus}, one can estimate each term on the right hand side and conclude the proof.
\end{proof}

In the next result, using classical harmonic analysis tools, we estimate some singular integrals in $C^\frac12(\T, \C)$. We do it for a general Calder\'on--Zygmund kernel depending on two complex parameters $a,b \in \C^2$. Here, $\phi$ is the periodic extension to $\T$ of a smooth even cutoff function such that
$$
\phi(s) = 1 \quad \textup{if } |s| \leq \frac18\,, \quad\textup{ and } \quad \phi(s) = 0 \quad \textup{if } |s| \geq \frac14\,.
$$

\begin{lemma}\label{SingInt}
Let $a,b\in \mathbb{H}$ be fixed and such that, for some $A_0\geq 1$,
\begin{equation}\begin{aligned}
||a|_{\C}-1|\leq \frac{1}{10}\quad &\mathrm{and}\quad ||b|_{\C}-||b|_{\C}||\leq \frac{1}{4}||b|_{\C}|\,,\\
|a|\leq A_0\quad &\mathrm{and}\quad ||b|_\C|\geq \frac{1}{10 A_0}|b|\,, \label{arg comp modulus}
\end{aligned}\end{equation}
 and \begin{align}
|\scalar{a}{b}_{\C}|\leq  \frac{1}{100 A_0 }|b|\,.\label{kinda ort}
\end{align}
Consider convolution kernels $f: \T \to \C$ defined for $|x| \leq \frac12$ by
 \begin{align}
f(x):=\frac{x^m\eps^r}{|xa+\eps b|_{\C}^{r+m+1}}\phi(x)\, \mathds{1}_{c\leq |x|\leq \frac12}\,,
\end{align}
with $\ep \in (0,1)$, $c\in [0,1/2]$, $r,m\in \N \cup \{0\}$, and $r\neq 0$ or $m$ odd. Then,
\begin{equation} \label{E.convolutionHolder}
\norm{f *g}_{C^{\frac12}(\T)} \lesssim |b|^{-r}\, \norm{g}_{C^{\frac12}(\T)} \,, \quad \textup{for all } g \in C^{\frac12}(\T)\,.
\end{equation}
Moreover, for $k_1, k_2 \in \N \cup \{0\}$, the derivatives of f with respect to the parameters a and b satisfy
\begin{equation} \label{E.convolutionDerivativeHolder}
\norm{({\rm D}^{k_1}_a {\rm D}_b^{k_2} f) * g}_{C^{\frac12}(\T)}  \lesssim |b|^{-r-k_2}\, \norm{g}_{C^{\frac12}(\T)} \,, \quad \textup{for all } g \in C^{\frac12}(\T)\,.
\end{equation}
The implicit constants here depend only on $k_1, k_2, r,m$ and $A_0$.
\end{lemma}
 
\begin{proof}
First of all, combining \eqref{arg comp modulus} and \eqref{kinda ort} with Young's inequality, we get that
\begin{equation} \label{lower bd}
\begin{aligned}
\Re|ax+\eps b|_{\C}^2 & = \Re\left\{x^2 |a|_{\C}^2+\eps^2|b|_{\C}^2+2 \ep x\scalar{a}{b}_{\C}\right\}\\
& \geq x^2 + \ep^2 ||b|_\C|^2 - x^2 (||a|_\C-1|^2 + 2||a|_\C-1| ) \\
& \quad - \ep^2 (||b|_\C-||b|_\C||^2 + 2||b|_\C|||b|_\C-||b|_\C|| ) - 2\ep |x| |\scalar{a}{b}_\C|  \gtrsim x^2 + \ep^2 |b|^2\,.
\end{aligned}
\end{equation}
Having this lower bound at hand, we analyze separately the cases $r > 0$ and $r =0$.
 
\noindent \textit{Case 1: $r > 0$}. Note that \eqref{lower bd} implies
\begin{equation} \label{E.pointwiseF}
|f(x)| \lesssim \frac{\ep^r |x|^m}{|x|^{r+m+1} + (\ep|b|)^{r+m+1}}\,, \quad \textup{for all } |x| \leq \frac12 \textup{ and all } r \geq 0\,.
\end{equation}
Hence, since we are dealing with the case where $r > 0$, we get that
\begin{align}
    & \int_{-\frac12}^{\frac12} |f(x)| \dd x\lesssim \ep^r  \int_{-\frac12}^{\frac12} \frac{|x|^m}{|x|^{r+m+1} + (\ep|b|)^{r+m+1}} \dd x \\
    & \qquad = \frac{2}{\ep |b|^{r+1}} \int_0^{\frac12}  \frac{\big(\frac{x}{\ep|b|}\big)^m}{1+\big(\frac{x}{\ep|b|}\big)^{r+m+1}} \dd x  \leq \frac{2}{|b|^r} \int_0^{\infty} \frac{\rho^m}{1+\rho^{r+m+1}} \dd \rho \lesssim|b|^{-r}\,.
\end{align}
In particular, the convolution $f *g$ is always well-defined for $g \in L^{\infty}(\T,\C)$. Furthermore, using Young's convolution inequality, we get that
\begin{align}
&\norm{g*f}_{L^\infty(\T) }\lesssim |b|^{-r}\norm{g}_{L^\infty(\T)}\,,\\
&\norm{g*f(\cdot)-g*f(\cdot+y)}_{L^\infty(\T)} \lesssim |b|^{-r} \norm{g(\cdot)-g(\cdot+y)}_{L^\infty(\T)}\lesssim |b|^{-r}|y|^\frac{1}{2}\norm{g}_{C^\frac{1}{2}(\T)},
\end{align}
which shows \eqref{E.convolutionHolder} in the case where $r > 0$.

\medbreak
\noindent \textit{Case 2: $r=0$.} In this case, the convolution needs to be estimated as a singular integral.
By standard Fourier multiplier theorems (see for instance \cite[Corollary \ 6.7.2 and Remark \ 6.5.2]{Grafakos2} and \cite[Theorem 4.4.1]{Grafakos1}) it suffices to show that \begin{align}
\sup_{0<R\leq \frac12}\frac{1}{R} \int_{-R}^R|f(x)||x|\dx &\lesssim 1 \,,\label{sing 1}\\
\sup_{y\neq 0}\int_{|x|\geq 2|y|}|f(x-y)-f(x)|\dx&\lesssim 1 \,, \label{sing 2}\\
\sup_{0<R_1 < R_2\leq \frac12}\left|\int_{R_1 \leq |x| \leq R_2}f(x)\dx\right|&\lesssim 1\,.\label{sing 3}
\end{align}
The first condition \eqref{sing 1} immediately follows from \eqref{E.pointwiseF} arguing as in the case where $r > 0$. Hence, we focus on the other two conditions. 

We first deal with \eqref{sing 2}. Here, we distinguish the cases $|y|\geq c$ and $|y|\leq c$. Note that in the first case
\begin{align}
    \int_{|x| \geq 2|y|} |f(x-y) - f(x)| \dd x = \int_{|x| \geq 2|y|} \left| \int_{-1}^0 f'(x+\tau y) y \dd \tau \right| \dd x\,.
\end{align}
Thus, using that
$$
|x+\tau y| \geq |x|-|y| \geq |y|\,, \quad \textup{for all } x,y \in [-1/2,1/2] \textup{ with }|x| \geq 2|y|\,,
$$
we get that
\begin{align}
      \int_{|x| \geq 2|y|} |f(x-y) - f(x)| \dd x \leq \int_{|z| \geq |y|} |y| |f'(z)| \dd z\,.
\end{align}
On the other hand, combining again \eqref{lower bd} with \eqref{arg comp modulus}, we get that
$$
|f'(z)| \lesssim \frac{|z|^{m-1}}{|z|^{m+1}+(\ep|b|)^{m+1}} + \frac{|z|^m}{|z|^{m+1}+(\ep|b|)^{m+1}} + \frac{|z|^m}{|z|^{m+2}+(\ep|b|)^{m+2}}\,,
$$
and so that
\begin{align}
         \int_{|x| \geq 2|y|} |f(x-y) - f(x)| \dd x \lesssim |y| \int_{|y|}^{\frac12} \frac{1}{\rho^2} \dd \rho \lesssim 1\,.
\end{align}
This immediately gives \eqref{sing 2} in the case where $|y|\geq c$.

In the other case, i.e. when $|y| \leq c$, we treat separately the cases where $|x|  \geq 2c$ and $|x| \leq 2c$. Note that here we are implicitly assuming that $c > 0$. If $|x| \leq 2c$, then $|x-y| \leq 3c$ and hence, using \eqref{E.pointwiseF}, we can estimate \begin{align}
    \int_{2|y|\leq|x|\leq2c}|f(x-y)-f(x)|\dd x\leq 2\int_{|x|\in [c,3c]}|f(x)|\dx\lesssim \int_{c}^{3c}\frac{1}{x}\dd x= \log(3c)-\log(c)\lesssim 1\,,
\end{align}
showing \eqref{sing 2}. On the other hand, if $|x| \geq 2c$, we have that
$$
|x + \tau y| \geq |x|-|y| \geq c\,,
$$
and we can argue exactly as we did when $|y| \geq c$ and obtain again \eqref{sing 2}.

We finally deal with \eqref{sing 3}. Let $0 < R_1 < R_2 \leq 1/2$. We need to consider three different cases separately.

\noindent \textit{Case 2.1: $R_2 \leq 10\ep|b|$.} In this case, using once again \eqref{E.pointwiseF}, we immediately get that
\begin{align}
    \left|\int_{R_1 \leq |x| \leq R_2}f(x)\dx\right| \leq \int_{-10\ep|b|}^{10\ep|b|} |f(x)| dx \lesssim \int_{-10}^{10} \frac{|\rho|^m}{1+|\rho|^{m+1}} \dd \rho \lesssim 1\,.
\end{align}

\noindent \textit{Case 2.2: $R_1 \geq 10\ep |b|$.} Since $m$ is odd, it follows that
$$
\left|\int_{R_1 \leq |x| \leq R_2}f(x)\dx\right| = \left| \int_{R_1}^{R_2} (f(x)+f(-x)) \dd x \right| \leq \int_{R_1}^{R_2} x^m \left| \frac{1}{|xa+\ep b|_\C^{m+1}} - \frac{1}{|-xa+\ep b|_\C^{m+1}} \right| \dd x \,.
$$
Moreover, using \eqref{kinda ort}, \eqref{lower bd}, and the polarization identity  
$$
||xa+\ep b|_\C^2 - |-xa+\ep b|_\C^2| = 4 | \scalar{ax}{\ep b}_\C|\,,
$$
we get that
$$
\left| \frac{1}{|xa+\ep b|_\C^{m+1}} - \frac{1}{|-xa+\ep b|_\C^{m+1}} \right| \lesssim \frac{| \scalar{ax}{\ep b}_\C|}{x^{m+3}+(\ep|b|)^{m+3}} \lesssim \frac{\eps|x||b|}{x^{m+3}+(\ep|b|)^{m+3}}\,, \quad \textup{for } |x| \in [R_1,R_2]\,.
$$
Thus, we infer that
$$
   \left|\int_{R_1 \leq |x| \leq R_2}f(x)\dx\right| \lesssim \frac{1}{\ep|b|} \int_{10\ep |b|}^{\frac12}  \frac{\big(\frac{x}{\ep|b|}\big)^{m+1}}{1+\big(\frac{x}{\ep|b|}\big)^{m+3}}\,\dd x \lesssim 1\,.
$$

\noindent \textit{Case 2.3: $R_1 < 10\ep|b| < R_2$.} Note that, in this case
$$
   \left|\int_{R_1 \leq |x| \leq R_2}f(x)\dx\right| \leq    \left|\int_{R_1 \leq |x| \leq 10\ep|b|}f(x)\dx\right| +    \left|\int_{10\ep|b| \leq |x| \leq R_2}f(x)\dx\right|\,.
$$
Hence, combining the previous cases, we get that
$$
 \left|\int_{R_1 \leq |x| \leq R_2}f(x)\dx\right| \lesssim 1\,.
$$

We have now proved \eqref{sing 1}--\eqref{sing 3}. Hence, by \cite[Corollary \ 6.7.2 and Remark \ 6.5.2]{Grafakos2} and \cite[Theorem 4.4.1]{Grafakos1}, we conclude that \eqref{E.convolutionHolder} holds. 

Concerning \eqref{E.convolutionDerivativeHolder}, we simply note that the derivatives of $f$ with respect to $a_j$ and $b_j$, for $j \in \{ 1,2\}$, are a linear combination of terms with the same structure as $f$, the only difference being that each derivative with respect to~$b_j$ introduces a factor~$\ep$. Each additional~$\ep$ can be absorbed in the estimates by picking an additional factor~$|b|$, with the same argument as above. This is the content of the estimate~\eqref{E.convolutionDerivativeHolder}.
\end{proof}

\subsection{First steps towards the proof of Proposition \ref{main est}} 
We consider the operator $K_\ep$ introduced in \eqref{E.defKep} and point out that one can equivalently write this operator as
$$
K_\ep[\nu_1, \nu_2, g, \varpi](s+i\be) = \frac{1}{2\pi} \int_\T  \frac{d^{\eps}(s+i\beta,\varsigma+i\beta)^\perp}{\big|d^{\eps}(s+i\beta,\varsigma+i\beta)\big|_{\C}^2}\,\varpi(\varsigma+i\beta)\dd \varsigma\,,
$$
where
\begin{align}
d^\ep(s+i\be, \varsigma+i\be) & := \tilde{\digamma}(s+i\be) - \tilde{\digamma}(\varsigma+i\be) + \ep \tilde{\ze}(s+i\be)\,,
\end{align}
with
\begin{align}
\tilde{\digamma}(a+ib) & := \Ga(a+ib) + \nu_1(a+ib)\, \dot{\Ga}(a+ib)^{\perp}\,, \\
\tilde{\ze}(a+ib) &:= - g(a+ib) \dot{\Ga}(a+ib)^{\perp}\,.
\end{align}
Moreover, for later purposes, and with some abuse of notation, we set
\begin{align}
& d_{l,\ell}^{\eps}(s+i\beta,\varsigma+i\beta)=\Ga(s+i\beta) - \Ga(\varsigma+i\beta)  +  w_1[\ell,s+i\beta]\dot{\Ga}(s+i\beta)^{\perp}-w_1[\ell,\varsigma+i\beta]\dot{\Ga}(\varsigma+i\beta)^{\perp} \\
&\quad- \ep\left( \int_l^{\ell} w_3[\mu,s+i\beta]  \dd \mu \right) \dot{\Ga}(s+i\beta)^{\perp}= \digamma(s+i\beta)-\digamma(\varsigma+i\beta)+\eps\zeta(s+i\beta)\,,
\end{align}
with
\begin{equation} \label{E.digamma}
\digamma(a+ib) := \Ga(a+ib) + w_1[\ell,a+ib]\, \dot{\Ga}(a+ib)^{\perp}\,,
\end{equation}
and
\begin{equation} \label{E.zeta}
 \zeta(a+ib) := - \bigg( \int_l^{\ell} w_3[\mu,a+ib]  \dd \mu \bigg) \dot{\Ga}(a+ib)^{\perp}\,.
\end{equation}
Then, we have that
\begin{equation} \label{E.holomorphicKepsilon}
\begin{aligned}
    &K_\ep\left[w_1[\ell,\cdot],\, w_2[\ell,\cdot],\, \int_l^\ell w_3[\mu,\cdot] \dd \mu,\, w_4[\ell,\cdot]\right](s+i\be) \\
    & \qquad =\frac{1}{2\pi} \int_\T   \frac{d_{l,\ell}^{\eps}(s+i\beta,\varsigma+i\beta)^\perp}{\big|d_{l,\ell}^{\eps}(s+i\beta,\varsigma+i\beta)\big|_{\C}^2}\,w_4[\ell,\varsigma+i\beta]\dd \varsigma\,,
    \end{aligned}
\end{equation}

\begin{remark}
    Both $\digamma$ and $\zeta$ depend on $w$, $l$ and $\ell$, but we do not reflect this in the notation for convenience. Moreover, taking into account the notation introduced in Section \ref{S.proofMain}, one can notice that $\digamma = \ga_{\ep,\ell}$ when restricted to $\T$.
\end{remark}

{Throughout this section we assume} $0<\ep\leq\ep_{*}${, where $\eps_*$ is sufficiently small.}
Having this notation at {hand}, we consider $0 \leq \rho < \rho_0$ and $w \in \cY^4_\rho$ satisfying 
\begin{align} \label{E.Zrho1}
    & \pd_s w_1[l,\cdot] = w_2[l,\cdot]\,, \quad \textup{for all } |l| \leq 1\,, \\
    & w_1[l,\cdot] - w_1[\ell,\cdot] = \ep \int_\ell^l w_3[\mu,\cdot]  \dd \mu \,, \quad \textup{for all } l,\ell \in [-1,1]\,. \label{E.Zrho2}
\end{align}
Moreover, we assume in all the proofs that $l \neq \ell$ (which is not restrictive as $\{l-\ell=0\}$ is a zero set). Taking into account \eqref{E.Zrho1} and \eqref{E.Zrho2}, it is straightforward to check that
 \begin{equation}\label{direct est}
 \begin{aligned}
& \sup_{|\be| \leq \rho} \Big\{\norm{\digamma(\cdot +i\beta)}_{C^\frac{3}{2}(\T)}+ \frac{1}{|l-\ell|}\norm{\zeta(\cdot+i\beta)}_{C^\frac{1}{2}(\T)}+\eps\norm{\zeta(\cdot+i\beta)}_{C^\frac{3}{2}(\T)} \Big\} \\
& \qquad \lesssim 1+\mathfrak g_\rho(w) , 
\end{aligned}
\end{equation}
where we recall that $\digamma$ and $\zeta$ were given in \eqref{E.digamma} and \eqref{E.zeta} respectively. Also, we set
$$
\tilde{d}(v,s+i\beta) :=v\digamma'(s+i\beta)+\eps\zeta(s+i\beta)\,, 
$$
where {the} derivative refers to the derivative in $s$, and split the kernel in \eqref{E.holomorphicKepsilon} as \begin{equation}\begin{aligned}
& \frac{\big(d_{l,\ell}^{\eps}(s+i\beta,\varsigma+i\beta)\big)^\perp}{\big|d_{l,\ell}^{\eps}(s+i\beta,\varsigma+i\beta)\big|_{\C}^2}= J_1(s-\varsigma,s+i\beta)+J_2(s+i\beta,\varsigma+i\beta)+J_3(s+i\beta,\varsigma+i\beta)\,,\label{splitting J}
\end{aligned}\end{equation}
with
\begin{align}
    & J_1(s-\varsigma, s+i\be) := \frac{\big(\tilde{d}(s-\varsigma,s+i\beta)\big)^\perp}{\big|\tilde{d}(s-\varsigma,s+i\beta)\big|_{\C}^2}\phi(s-\varsigma)\,,\\
    & J_2(s+i\beta,\varsigma+i\beta) := \left(\frac{\big(d_{l,\ell}^{\eps}(s+i\beta,\varsigma+i\beta)\big)^\perp}{\big|d_{l,\ell}^{\eps}(s+i\beta,\varsigma+i\beta)\big|_{\C}^2}-\frac{\big(\tilde{d}(s-\varsigma,s+i\beta)\big)^\perp}{\big|\tilde{d}(s-\varsigma,s+i\beta)\big|_{\C}^2}\right)\phi(s-\varsigma)\,,\\
    & J_3(s+i\beta,\varsigma+i\beta):= \frac{\big(d_{l,\ell}^{\eps}(s+i\beta,\varsigma+i\beta)\big)^\perp}{\big|d_{l,\ell}^{\eps}(s+i\beta,\varsigma+i\beta)\big|_{\C}^2}(1-\phi(s-\varsigma))\,.
\end{align}
Here, $\phi$ is the periodic extension to $\T$ of a smooth even cutoff function such that
$$
\phi(s) = 1 \quad \textup{if } |s| \leq \frac18\,, \quad\textup{ and } \quad \phi(s) = 0 \quad \textup{if } |s| \geq \frac14\,.
$$

\subsection{Proof of (\ref{bd K})} This subsection is devoted to proving \eqref{bd K}. Having  \eqref{splitting J} at hand,  we analyze the part corresponding to each kernel $J_j$ separately. We start with the one corresponding to $J_1$. We treat it as a convolution operator on $\T$ with an extra dependence on $s + i\be \in \T_\rho\,.$ 

\begin{lemma} \label{L.ab}
    Let $a= a(s+i\be) := \digamma'(s+i\be) \in \C^2$ and $b = b(s+i\be):= \ze(s+i\be) \in \C^2$ for $\beta\in [-\rho,\rho]$. There exists a constant $C_1 > 0$ depending only on $\Ga$ such that, if $w \in \cY^4_\rho$ satisfies \eqref{ass small} with $\mathcal{C}_2 \leq  C_1$ as well as \eqref{E.Zrho1} and \eqref{E.Zrho2}, then $a$ and $b$ satisfy \eqref{arg comp modulus} and \eqref{kinda ort}  for some $A_0 \geq 1$ depending only on $\Ga$.
\end{lemma}
 
\begin{proof}
First of all, it directly follows from \eqref{direct est} that $||a|_\C| \leq |a|\lesssim 1$. Moreover, as $\Ga$ is parametrized by {arclength}, we have that
\begin{equation} \label{E.arclengthComplex}
|\dot{\Ga}(s+i\be)|_\C = 1\,,
\end{equation}
and so that
\begin{align}
    ||a|_{\C}-1| = \Big|\big|\dot{\Ga}(s+i\be) + w_2[\ell,s+i\be] \dot{\Ga}(s+i\be)^{\perp} + w_1[\ell,s+i\be] \ddot{\Ga}(s+i\be)^{\perp}\big|_\C - \big|\dot{\Ga}(s+i\be)\big|_\C\Big|\,.
\end{align}

Now, observe that
\begin{align}
    & \Re\big|\dot{\Ga}(s+i\be) + w_2[\ell,s+i\be] \dot{\Ga}(s+i\be)^{\perp} + w_1[\ell,s+i\be] \ddot{\Ga}(s+i\be)^{\perp}\big|_\C \\
    & \qquad \geq 1 - C \big(\|w_1[\ell,\cdot]\|_\rho + \|w_2[\ell,\cdot]\|_\rho \big)\,,
\end{align}
for some $C > 0$ depending only on $\Ga$. Hence, there exists $C_1 > 0$ sufficiently small such that, if \eqref{ass small} holds with $\mathcal{C}_2 \leq C_1$, then
$$
\Re\big|\dot{\Ga}(s+i\be) + w_2[\ell,s+i\be] \dot{\Ga}(s+i\be)^{\perp} + w_1[\ell,s+i\be] \ddot{\Ga}(s+i\be)^{\perp}\big|_\C  \geq \frac12\,.
$$
On the other hand, observe that $|\cdot|_\C$ is locally Lipschitz away from the set of points $z \in \C^2$ where $|z|_\C = 0$. Hence, we have that
\begin{equation}
    ||a|_\C -1| \lesssim \left|w_2[\ell,s+i\be] \dot{\Ga}(s+i\be)^{\perp} + w_1[\ell,s+i\be] \ddot{\Ga}(s+i\be)^{\perp}\right| \lesssim \|w_1[\ell,\cdot]\|_\rho + \|w_2[\ell,\cdot]\|_\rho\,.
\end{equation}
Note that all the implicit constants here depend only on $\Ga$. Taking $C_1 > 0$ smaller if necessary we conclude that
$$
||a|_\C-1| \leq \frac{1}{10}\,.
$$

Next, observe that
\begin{align} 
    \bigl||b|_\C - ||b|_\C|\bigr| \leq \frac{1}{4} ||b|_\C| \ \Longleftrightarrow\ ||b|_\C| - \Re |b|_\C \leq \frac{1}{32}||b|_\C|  \Longleftrightarrow\  \Re |b|_\C \geq \frac{31}{32}||b|_\C|\,.
\end{align}
Moreover, by \eqref{E.arclengthComplex}, it follows that
$$
|b|_\C^2 = (l-\ell)^2 \left( \fint_l^\ell w_3[\mu,s+i\be] \dd \mu \right)^2\,,
$$
and so that
\begin{equation} \label{E.lowerboundReb}
\Re |b|_\C \geq |l-\ell| \Big(1-\sup_{|l|\leq 1} \norm{w_3[l,s+i\be]-1}_{L^{\infty}(\T)} \Big)\,,
\end{equation}
and that
$$
||b|_\C| \leq |l-\ell| \Big( 1+\sup_{|l|\leq 1} \norm{w_3[l,s+i\be]-1}_{L^{\infty}(\T)} \Big)\,.
$$

Hence, to have that
\begin{equation} \label{E.bminusbComplex}
||b|_\C - ||b|_\C|| \leq \frac14 ||b|_\C|\,,
\end{equation}
it is enough to impose that 
$$
\sup_{|l|\leq 1} \norm{w_3[l,\cdot +i\be]-1}_{L^{\infty}(\T)} \leq \frac{1}{64}\,.
$$
Taking $C_1 > 0$ smaller if necessary, we conclude that \eqref{E.bminusbComplex} holds. On the other hand, 
$$
|b| \leq |l-\ell| \bigl(1+\sup_{|l| \leq 1} \|w_3[l,\cdot + i\be]-1\|_{L^{\infty}(\T)}\bigr) \|\dot{\Ga}(\cdot + i \be)\|_{L^{\infty}(\T)} \lesssim |l-\ell|  \bigl(1+\sup_{|l| \leq 1} \|w_3[l,\cdot + i\be]-1\|_{L^{\infty}(\T)}\bigr) \,,
$$
for some implicit constant depending only on $\Ga$. Using \eqref{E.lowerboundReb}, it is then straightforward to obtain the existence of $A_0 \geq 1$ depending only on $C_1 > 0$ (and thus on $\Ga$) such that
$$
||b|_\C|\geq \frac{1}{10 A_0}|b|\,.
$$
For later purposes, let us point out that taking $C_1 > 0$ smaller would not affect the existence of $A_0 \geq 1$. This constant actually depends on an upper bound of $C_1 > 0$. This concludes the proof of \eqref{arg comp modulus}. 

We now prove \eqref{kinda ort}.  To that end, let us first stress that
$$
\scalar{\dot{\Ga}(s+i\be)^{\perp}}{\dot{\Ga}(s+i\be)}_\C = 0\,.
$$
Hence, we have that
\begin{align}
    |\scalar{a}{b}_\C| & = \bigg| \bigg\langle w_2[\ell,s+i\be] \dot{\Ga}(s+i\be)^{\perp} \\
    & \qquad + w_1[\ell,s+i\be] \ddot{\Ga}(s+i\be)^{\perp}, \bigg( \int_l^{\ell} w_3[\mu,s+i\beta]  \dd \mu \bigg) \dot{\Ga}(s+i\beta)^{\perp} \bigg\rangle_\C \bigg| \\
    & \leq \bigl|w_2[\ell,s+i\be] \dot{\Ga}(s+i\be)^{\perp} + w_1[\ell,s+i\be] \ddot{\Ga}(s+i\be)^{\perp}\bigr| \, |b| \\
    & \lesssim \big(\|w_1[\ell, \cdot + i \be]\|_{L^{\infty}(\T)} + \|w_2[\ell, \cdot + i \be]\|_{L^{\infty}(\T)} \big) |b|\,,
\end{align}
for some implicit constant depending only on $\Ga$. Taking $C_1 > 0$ smaller if necessary we get \eqref{kinda ort} and conclude the proof. 
\end{proof}

\begin{lemma} \label{L.J1}
Let $C_1 > 0$ be as in Lemma \ref{L.ab}. If $w \in \cY^4_\rho$ satisfies \eqref{ass small} with $\mathcal{C}_2 \leq  C_1$, \eqref{E.Zrho1} and \eqref{E.Zrho2}, then
$$
\sup_{l,\ell \in [-1,1]} \left\{ \sup_{|\be| \leq \rho} \norm{ \int_\T J_1(\cdot-\varsigma,\cdot+i\be) w_4[\ell,\varsigma+i\be] \dd \varsigma}_{C^{\frac12}(\T)} \right\} \lesssim\, \sup_{|l|\leq 1} \norm{w_4[l,\cdot]}_\rho \,.
$$
\end{lemma}
 
\begin{proof}
For $a$ and $b$ as in Lemma \ref{L.ab}, observe that
$$
J_1(x,s+i\be) = \left(a^{\perp} \frac{x}{|xa+\ep b|_\C^2} + b^{\perp} \frac{\ep}{|xa+\ep b|_\C^2}\right)\phi(x)\,.
$$
By Lemma \ref{SingInt}, applied first with $m = 1$ and $r = 0$ and then with $m = 0$ and $r =1$, we get that
\begin{align} \label{E.J1bounded}
\norm{\int_{\T} J_1(\cdot-\varsigma,\cdot+i\beta)w_4[\ell,\varsigma+i\beta]\dd\varsigma}_{L^{\infty}(\T)} \lesssim \norm{w_4[\ell,\cdot+i\beta]}_{C^\frac{1}{2}(\T)}\,, \quad \textup{for } |\be| \leq \rho\,.
\end{align}

Now, observe that
\begin{align}
& \left|\int_{\T} \big( J_1(s_1-\varsigma,s_1+i\beta)w_4[\ell,\varsigma+i\beta]-J_1(s_2-\varsigma,s_2+i\beta)w_4[\ell,\varsigma+i\beta] \big) \dd\varsigma\right|\\
& \quad \leq \left| \int_{\T} J_1(s_1-\varsigma,s_1+i\beta)w_4[\ell,\varsigma+i\beta] \dd \varsigma - \int_{\T} J_1(s_2-\varsigma,s_1+i\beta)w_4[\ell,\varsigma+i\beta] \dd \varsigma \right| \\
&\qquad+\left|\int_{\T} \left(
J_1(s_2-\varsigma,s_1+i\beta)-J_1(s_2-\varsigma,s_2+i\beta)\right)w_4[\ell,\varsigma+i\beta]\dd\varsigma\right| =: {\rm I}_1 + {\rm I}_2\,.
\end{align}
Arguing as in the proof of \eqref{E.J1bounded}, using Lemma \ref{SingInt} twice, we get that
\begin{equation} \label{E.conclussionI1J1}
     {\rm I}_1 \lesssim |s_1-s_2|^{\frac12} \norm{w_4[\ell,\cdot+i\beta]}_{C^\frac{1}{2}(\T)}\,, \quad \textup{for } |\be| \leq \rho\,.
\end{equation}
To analyze ${\rm I}_2$, {write} $a_{j}=\digamma'(s_j+i\beta)$ and $b_{j}=\zeta(s_j+i\beta)${. The direct estimate \eqref{direct est} gives}
\begin{equation}\label{E.a1a2diff}
|a_1-a_2|\lesssim |s_1-s_2|^{1/2},\qquad
|b_1-b_2|\lesssim |l-\ell||s_1-s_2|^{1/2}.
\end{equation}
{Also}
\begin{equation}\label{E.lowerbll}
|b_j|\geq||b_j|_{\CC}|\geq\tfrac12|l-\ell|.
\end{equation}
{For} $|s_{1}-s_{2}|$ {larger than a fixed small number,} the $L^{\infty}$ {bound already suffices. For smaller differences we interpolate linearly between} $(a_{1},b_{1})$ {and} $(a_{2},b_{2})${.} Let \begin{align}
    a_\theta=(2-\theta)a_1+(\te-1) a_2, \quad b_\theta=(2-\theta)b_1+(\theta-1) b_2 
\end{align}
for $\theta\in [1,2]$. It is straightforward to check that the conditions \eqref{arg comp modulus} and \eqref{kinda ort} hold for $(a_\theta,b_\theta)$ for each $\theta$ with a fixed $A_0$.

Therefore, if we {let} $T(a_\theta,b_\theta)$ {denote convolution with} $\phi(x)(xa_\theta+\ep b_\theta)^{\perp}/|xa_\theta+\ep b_\theta|_{\CC}^{2}${} then Lemma~\ref{SingInt}, {including its parameter derivatives, gives}
\[
\|D_{a}T(a_{\theta},b_{\theta})\|_{C^{1/2}\to C^{1/2}}\lesssim 1,
\qquad
\|D_{b}T(a_{\theta},b_{\theta})\|_{C^{1/2}\to C^{1/2}}\lesssim |l-\ell|^{-1},
\]
which shows that \begin{align}
    \norm{T(a_1,b_1)-T(a_2,b_2)}_{C^\frac{1}{2}\rightarrow C^{\frac{1}{2}}}\lesssim |a_1-a_2|+\frac{|b_1-b_2|}{|l-\ell|}\lesssim |s_1-s_2|^\frac{1}{2}.
\end{align}
Since this precisely the difference in $I_2$, we see that 
\begin{equation}\label{E.conclussionI2J1}
\begin{aligned}
{\rm I}_2 & \lesssim \big(|a_{1}-a_{2}|+|l-\ell|^{-1}|b_{1}-b_{2}|\big)
\|w_{4}[\ell,\cdot+i\beta]\|_{C^{1/2}(\TT)} \\
& \lesssim |s_1-s_2|^{1/2}\|w_4[\ell,\cdot+i\beta]\|_{C^{1/2}(\TT)}\,, \quad \textup{for } |\be| \leq \rho\,.
\end{aligned}
\end{equation}
The result follows by combining \eqref{E.J1bounded}, \eqref{E.conclussionI1J1} and \eqref{E.conclussionI2J1}.
\end{proof}
 
We now move on to the analysis of the part corresponding to $J_2$. A key role in this analysis is played by Lemma \ref{Lem double diff}.  

\begin{lemma}\label{bds d}
Let $C_1> 0$ be as in Lemma \ref{L.ab}. There exists a constant $C_2 \in (0,C_1]$ depending only on $\Ga$ such that, if $w \in \cY^4_\rho$ satisfies \eqref{ass small} with $\mathcal{C}_2 \leq C_2$, \eqref{E.Zrho1} and \eqref{E.Zrho2}, then, for $s + i \be \in \T_\rho$ and $\varsigma \in \T$:
 \begin{align}
& \textup{\rm a) }\ |d_{l,\ell}^\eps(s+i\beta,\varsigma+i\beta)|+|\tilde{d}(s-\varsigma,s+i\beta)|\lesssim |s-\varsigma|+\eps|l-\ell|\,, \label{bd d up} \\
& \textup{\rm b) }\  |d_{l,\ell}^\eps(s+i\beta,\varsigma+i\beta)-\tilde{d}(s-\varsigma,s+i\beta)|\lesssim |s-\varsigma|^\frac{3}{2}\,, \label{bd diff} \\
&  \textup{\rm c) }\  \min \Big\{\Re |d_{l,\ell}^\eps(s+i\beta,\varsigma+i\beta)|_\C^2\,, \, \Re |\tilde{d}(s-\varsigma,s+i\beta)|_\C^2 \Big\} \gtrsim |s-\varsigma|^2+\eps^2|l-\ell|^2\,.\label{bd d low} 
\end{align}
\end{lemma}

\begin{proof}
First, let us recall that
$$
d_{l,\ell}^\ep (s+i\be, \varsigma+i\be) = \digamma(s+i\be) - \digamma(\varsigma+i\be) + \ep \ze(s+i\be)\,, 
$$
and that
$$
\tilde{d}(s-\varsigma,s+i\be) = (s-\varsigma)\digamma'(s+i\be) + \ep\ze(s+i\be)\,.
$$
We prove each estimate separately. Since the upper bounds in \eqref{bd d up} are a straightforward consequence of \eqref{direct est}, we focus on b) and c). Using again \eqref{direct est} and the mean value theorem, it is also immediate that
\begin{align}
& |d_{l,\ell}^\eps(s+i\beta,\varsigma+i\beta)-\tilde{d}(s-\varsigma,s+i\beta)| \\
& = \left|\digamma(s+i\beta)-\digamma(\varsigma+i\beta)-(s-\varsigma)\digamma'(s+i\beta)\right| \lesssim |s-\varsigma|^\frac{3}{2}\norm{\digamma(\cdot+i\beta)}_{C^\frac{3}{2}(\T)}\lesssim |s-\varsigma|^\frac{3}{2}\,.
\end{align}

Finally, we prove the lower bounds in c). For $a$ and $b$ as in Lemma \ref{L.ab}, it follows that
$$
|\tilde{d}(s-\varsigma, s+i\be)|_\C^2 = |(s-\varsigma) a + \ep b|_\C^2\,.
$$
Hence, if \eqref{ass small} holds with $\cC_2 \leq C_1$ for $C_1> 0$ as in Lemma \ref{L.ab}, combining \eqref{lower bd} and \eqref{E.lowerbll}, we get that
$$
\Re |\tilde{d}(s-\varsigma, s+i\be)|_\C^2 \gtrsim  |s-\varsigma|^2 + \ep^2 |l-\ell|^2 \,.
$$
\allowdisplaybreaks
Next, using \eqref{E.complex-modulus}, \eqref{bd d up} and \eqref{bd diff}, we get that, whenever $|s-\varsigma| \leq 1$,
\begin{align}
    & ||\tilde{d}(s-\varsigma, s+i\be)|_\C^2-|d_{l,\ell}^\eps(s+i\beta,\varsigma+i\beta)|_\C^2| \\
    & \quad = |\scalar{\tilde{d}(s-\varsigma, s+i\be)-d_{l,\ell}^\eps(s+i\beta,\varsigma+i\beta)}{\tilde{d}(s-\varsigma, s+i\be)+d_{l,\ell}^\eps(s+i\beta,\varsigma+i\beta)}_\C| \\
    & \quad \leq |\tilde{d}(s-\varsigma, s+i\be)-d_{l,\ell}^\eps(s+i\beta,\varsigma+i\beta)|^2 \\
    & \qquad + 2 |\scalar{\tilde{d}(s-\varsigma, s+i\be)-d_{l,\ell}^\eps(s+i\beta,\varsigma+i\beta)}{d_{l,\ell}^\eps(s+i\beta,\varsigma+i\beta)}_\C| \\
     & \quad \leq |\tilde{d}(s-\varsigma, s+i\be)-d_{l,\ell}^\eps(s+i\beta,\varsigma+i\beta)|^2 \\
    & \qquad + 2 |d_{l,\ell}^\eps(s+i\beta,\varsigma+i\beta)||\tilde{d}(s-\varsigma, s+i\be)-d_{l,\ell}^\eps(s+i\beta,\varsigma+i\beta)| \\
    & \quad \lesssim |s-\varsigma|^{\frac12} \big( |s-\varsigma|^2 + \ep|l-\ell||s-\varsigma| \big)  \\
    & \quad \lesssim |s-\varsigma|^{\frac12} \big( |s-\varsigma|^2 + \ep^2 |l-\ell|^2 \big)\,. 
\end{align}
Also, observe that
\begin{align}
    & \Re |d_{l,\ell}^\eps(s+i\beta,\varsigma+i\beta)|_\C^2 = \Re \big\{ |\tilde{d}(s-\varsigma, s+i\be)|_\C^2 + |d_{l,\ell}^\eps(s+i\beta,\varsigma+i\beta)|_\C^2 - |\tilde{d}(s-\varsigma, s+i\be)|_\C^2 \big\} \\
    & \quad \geq \Re |\tilde{d}(s-\varsigma, s+i\be)|_\C^2 - ||\tilde{d}(s-\varsigma, s+i\be)|_\C^2-|d_{l,\ell}^\eps(s+i\beta,\varsigma+i\beta)|_\C^2|\,.
\end{align}
Hence, there exists $c \in (0,1)$ depending only on $\Ga$ such that, if $|s-\varsigma| \leq c$, then 
$$
\Re |d_{l,\ell}^\eps(s+i\beta,\varsigma+i\beta)|_\C^2 \gtrsim |s-\varsigma|^2 + \ep^2 |l-\ell|^2\,.
$$
We now deal with the case where $|s-\varsigma| \geq c$\,. In this case, by the Assumption \eqref{tec no self int}, it holds that 
\begin{align}
\Re|\Gamma(s+i\beta)-\Gamma(\varsigma+i\beta)|_{\C}^2\geq \sigma,
\end{align}
for some $\sigma\in (0,1)$, independent of $\eps$. Hence, if we choose $C_2$ and $\eps_*$ sufficiently small, and assume that $\cC_2 \leq C_2$, we see from the triangle inequality and the definition of $d_{l,\ell}^\ep$ that 
\begin{equation}
\Re |d_{l,\ell}^\eps(s+i\beta,\varsigma+i\beta)|_\C^2\geq \frac{1}{2}\sigma\gtrsim |s-\varsigma|^2+\eps^2|l-\ell|^2.
\end{equation}
This concludes the proof. 
\end{proof}

Having  this lemma at hand, and before going any further, let us point out that the fact that $K_\ep$ is holomorphic for $l \neq \ell$ is a straightforward consequence of the lower bound in c). More precisely, we have the following:

\begin{lemma} \label{L.holomorphicExtension}
Let $C_2 > 0$ be as in Lemma \ref{bds d}. For all $\ep \in (0,\ep_{*}]$, all $l \neq \ell$,   and every $w \in \cY_\rho^4$ satisfying \eqref{ass small} with $\cC_2 \leq C_2$, \eqref{E.Zrho1} and \eqref{E.Zrho2}, it follows that
$$
K_\ep\bigg[w_1[\ell,\cdot], w_2[\ell,\cdot], \int_l^\ell w_3[\mu,\cdot] \dd \mu, w_4[\ell,\cdot]\bigg] \in \cX_\rho^2\,.
$$
\end{lemma}

\begin{proof} First of all, let us recall that
\begin{align}
    K_\ep & \bigg[w_1[\ell,\cdot], w_2[\ell,\cdot], \int_l^\ell w_3[\mu,\cdot] \dd \mu, w_4[\ell,\cdot]\bigg] (s+i\be) \\
    & = \frac{1}{2\pi} \int_\TT \frac{(\digamma(s+i\beta)-\digamma(s+i\beta+ \varsigma)+\eps\zeta(s+i\beta))^{\perp}}{|\digamma(s+i\beta)-\digamma(s+i\beta+\varsigma)+\eps\zeta(s+i\beta)|_\C^2} w_4[\ell, s+i\be + \varsigma] \dd \varsigma 
\end{align}
which is holomorphic by definition.

Moreover, by Lemma \ref{bds d} c) we know that
$$
\Re |\digamma(s+i\beta)-\digamma(s+i\beta+ \varsigma)+\eps\zeta(s+i\beta)|_\C^2 \gtrsim \varsigma^2 + \ep^2 |l-\ell|^2\,.
$$
Hence, since $l \neq \ell$, the result immediately follows.
\end{proof}

\begin{lemma} \label{L.J2}
Let $C_2 > 0$ be as in Lemma \ref{bds d}. If $w \in \cY^4_\rho$ satisfies \eqref{ass small} with $\cC_2 \leq C_2$, \eqref{E.Zrho1} and \eqref{E.Zrho2}, then
$$
\sup_{l,\ell \in [-1,1]} \left\{ \sup_{|\be| \leq \rho} \norm{ \int_\T J_2(\cdot+i\be,\varsigma+i\be) w_4[\ell,\varsigma+i\be] \dd \varsigma}_{C^{\frac12}(\T)} \right\} \lesssim\, \sup_{|l|\leq 1} \norm{w_4[l,\cdot]}_\rho\,.
$$
\end{lemma}

\begin{proof}
First of all, observe that
\begin{align}
|J_2(s+i\beta,\varsigma+i\beta)|&= \left|\frac{(d_{l,\ell}^{\eps}(s+i\beta,\varsigma+i\beta))^\perp}{|d_{l,\ell}^{\eps}(s+i\beta,\varsigma+i\beta)|_{\C}^2}-\frac{(\tilde{d}(s-\varsigma,s+i\beta))^\perp}{|\tilde{d}(s-\varsigma,s+i\beta)|_{\C}^2}\right|\phi(s-\varsigma)\\
&\leq \left|\frac{\big(d_{l,\ell}^{\eps}(s+i\beta,\varsigma+i\beta)-\tilde{d}(s-\varsigma,s+i\beta)\big)|\tilde{d}(s-\varsigma,s+i\beta)|_{\C}^2}{|d_{l,\ell}^{\eps}(s+i\beta,\varsigma+i\beta)|_{\C}^2|\tilde{d}(s-\varsigma,s+i\beta)|_{\C}^2}\right|\phi(s-\varsigma)\\
&\quad+\left|\frac{\tilde{d}(s-\varsigma,s+i\beta)\big(|\tilde{d}(s-\varsigma,s+i\beta)|_{\C}^2-|d_{l,\ell}^{\eps}(s+i\beta,\varsigma+i\beta)|_{\C}^2\big)}{|d_{l,\ell}^{\eps}(s+i\beta,\varsigma+i\beta)|_{\C}^2|\tilde{d}(s-\varsigma,s+i\beta)|_{\C}^2}\right|\phi(s-\varsigma)\,.
\end{align}
\noindent From Lemma \ref{bds d} and the elementary formula $||p|_\C^2-|q|_{\C}^2|\leq |p-q||p+q|$ we hence see that \begin{align}
\left|J_2(s+i\beta,\varsigma+i\beta)\right|\lesssim  |s-\varsigma|^{-\frac{1}{2}}\,\mathds{1}_{|s-\varsigma|\leq\frac{1}{4}}\,, \quad \textup{for }|\be| \leq \rho\,, \label{bd J2}
\end{align}
and so that 
\begin{equation} \label{E.J2bounded}
\begin{aligned}
& \norm{\int_{\T} J_2(\cdot+i\be,\varsigma+i\beta)w_4[\ell,\varsigma+i\beta]\dd\varsigma}_{L^{\infty}(\T)} \\ & \qquad  \lesssim \norm{w_4[\ell,\cdot+i\beta]}_{L^\infty(\T)} \int_{|s-\varsigma|\leq\frac{1}{2}} \frac{\dd\varsigma}{|s-\varsigma|^{\frac{1}{2}}}\lesssim \norm{w_4[\ell,\cdot+i\beta]}_{L^\infty(\T)}\,, \quad \textup{for } |\be| \leq \rho\,.
\end{aligned}
\end{equation}

Now, to bound the corresponding H\"older seminorm, we use that
\begin{align}
&\left|\int_{\T} \big(J_2(s_1+i\beta,\varsigma+i\beta)-J_2(s_2+i\beta,\varsigma+i\beta)\big)w_4[\ell,\varsigma+i\beta]\dd\varsigma \right|\\
& \quad \leq\int_{|\varsigma-s_1|\leq 3|s_1-s_2|}\big(|J_2(s_1+i\beta,\varsigma+i\beta)|+|J_2(s_2+i\beta,\varsigma+i\beta)|\big)\left|w_4[\ell,\varsigma+i\beta] \right| \dd\varsigma\\
&\qquad+\left|\int_{|\varsigma-s_1|\geq 3|s_1-s_2|}\left(J_2(s_1+i\beta,\varsigma+i\beta)-J_2(s_2+i\beta,\varsigma+i\beta)\right)w_4[\ell,\varsigma+i\beta]\dd\varsigma\right| =: {\rm I}_1 + {\rm I}_2\,.
\end{align}
On one hand, using \eqref{bd J2} and arguing as in \eqref{E.J2bounded}, we get that 
\begin{equation} \label{E.conclussionI1J2}
     {\rm I}_1 \lesssim |s_1-s_2|^{\frac12} \norm{w_4[\ell,\cdot+i\beta]}_{L^{\infty}(\T)}\,, \quad \textup{for } |\be| \leq \rho\,.
\end{equation}
On the other hand, to estimate ${\rm I}_2$, we consider separately the cases where $|s_1 - s_2| \geq \frac{1}{32}$ and where $|s_1 - s_2| \leq \frac{1}{32}$.

\noindent \textit{Case 1:} $|s_1-s_2| \geq \frac1{32}$. In this case, using again \eqref{bd J2}, it is immediate to get that
\begin{equation} \label{E.conclussionI2J2}
     {\rm I}_2 \lesssim |s_1-s_2|^{\frac12} \norm{w_4[\ell,\cdot+i\beta]}_{L^{\infty}(\T)}\,, \quad \textup{for } |\be| \leq \rho\,.
\end{equation}

\noindent \textit{Case 2:} $|s_1-s_2| \leq \frac{1}{32}$. First of all, we set
\begin{align}
& p = p(\varsigma,s_1+i\be):= \tilde{d}(s_1-\varsigma,s_1+i\beta) \,,\\ 
& q = q(\varsigma, s_1+i\be) := d_{l,\ell}^\eps(s_1+i\beta,\varsigma+i\beta)\,,     \\
& r = r(\varsigma,s_2+i\be):= \tilde{d}(s_2-\varsigma,s_2+i\beta)\,,   \\
& t = t(\varsigma,s_2+i\be) := d_{l,\ell}^\eps(s_2+i\beta,\varsigma+i\beta)\,,
\end{align}
and split the kernel in ${\rm I}_2$ as
\begin{align}
& J_2(s_2+i\beta,\varsigma+i\beta)-J_2(s_1+i\beta,\varsigma+i\beta)=\left(\frac{p^\perp}{|p|_{\C}^2}-\frac{q^\perp}{|q|_{\C}^2}-\frac{r^\perp}{|r|_{\C}^2}+\frac{t^\perp}{|t|_{\C}^2}\right)\phi(s_1-\varsigma)\\
&\quad+\left(\frac{r^\perp}{|r|_{\C}^2}-\frac{t^\perp}{|t|_{\C}^2}\right)\big(\phi(s_1-\varsigma)-\phi(s_2-\varsigma)\big)\,.
\end{align}
Then, using the mean value theorem, we get that 
\begin{align}
    \left|\frac{r^\perp}{|r|_{\C}^2}-\frac{t^\perp}{|t|_{\C}^2}\right|\big|\phi(s_1-\varsigma)-\phi(s_2-\varsigma)\big| = \left|\frac{r^\perp}{|r|_{\C}^2}-\frac{t^\perp}{|t|_{\C}^2}\right| \left| \int_0^1 \phi'(\la(s_1-s_2) + s_2-\varsigma) \dd \la \right||s_1-s_2|\,. 
\end{align}
Moreover, observe that
\begin{equation} \label{E.supportPhiprime}
\phi'(\la(s_1-s_2) + s_2-\varsigma) \neq 0 \ \Longrightarrow \ |s_1-\varsigma| \geq \frac{1}{16} \ \Longrightarrow\ |s_2-\varsigma| \geq \frac{1}{32}\,.
\end{equation}
Hence, by \eqref{bd d low}, we get that 
$$
   \left|\frac{r^\perp}{|r|_{\C}^2}-\frac{t^\perp}{|t|_{\C}^2}\right|\big|\phi(s_1-\varsigma)-\phi(s_2-\varsigma)\big|  \lesssim |s_1-s_2|\,,  \quad \textup{for }|\be| \leq \rho\,,
$$
and so that
\begin{equation}
\begin{aligned}
    & \left| \int_{|\varsigma-s_1|\geq 3|s_1-s_2|}\left(\frac{r^\perp}{|r|_{\C}^2}-\frac{t^\perp}{|t|_{\C}^2}\right)\big(\phi(s_1-\varsigma)-\phi(s_2-\varsigma)\big) w_4[\ell,\varsigma+i\beta]\dd\varsigma \right| \\[0.4cm]
    & \qquad  \lesssim |s_1-s_2| \norm{w_4[\ell,\cdot+i\beta]}_{L^{\infty}(\T)} \lesssim |s_1-s_2|^{\frac12} \norm{w_4[\ell,\cdot+i\beta]}_{L^{\infty}(\T)}  \,, \quad \textup{for } |\be| \leq \rho\,.
    \end{aligned}
\end{equation}

To estimate the remaining part, we are going to use Lemma \ref{Lem double diff}. First of all, observe that
\begin{align}
    &\frac{p-q-r+t}{|p|_{\C}^2} = - \frac{\digamma(s_1+i\beta)-\digamma(s_2+i\beta)-(s_1-s_2)\digamma'(s_2+i\beta)}{|p|_\C^2} \\
    & \qquad + \frac{(s_1-\varsigma)}{|p|_\C^2} \big( \digamma'(s_1+i\be) - \digamma'(s_2+i\be) \big)\,.
\end{align}
On one hand, combining \eqref{direct est} with \eqref{bd d low} (see also the proof of \eqref{bd diff}), we get that
\begin{equation} \label{E.Digamma1}
\left|\frac{\digamma(s_1+i\beta)-\digamma(s_2+i\beta)-(s_1-s_2)\digamma'(s_2+i\beta)}{|p|_\C^2}\right| \lesssim \frac{|s_1-s_2|^{\frac32}}{|s_1-\varsigma|^{2}}\,, \quad \textup{for } |\be| \leq \rho\,,
\end{equation}
and so that
\begin{equation} \label{E.Digamma2}
\begin{aligned}
& \left| \int_{|\varsigma-s_1|\geq 3|s_1-s_2|} \frac{\digamma(s_1+i\beta)-\digamma(s_2+i\beta)-(s_1-s_2)\digamma'(s_2+i\beta)}{|p|_\C^2}  w_4[\ell, \varsigma+i\be] \dd \varsigma \right| \\
& \qquad \lesssim \norm{w_4[\ell,\cdot+i\beta]}_{L^{\infty}(\T)} |s_1-s_2|^{\frac32} \int_{|\varsigma-s_1|\geq 3|s_1-s_2|}  \frac{1}{|s_1-\varsigma|^{2}} \dd\varsigma \\
& \qquad \lesssim |s_1-s_2|^{\frac12} \norm{w_4[\ell,\cdot+i\beta]}_{L^{\infty}(\T)}\,, \quad \textup{for } |\be|\leq \rho\,.
\end{aligned}
\end{equation}
On the other hand, using Lemma \ref{SingInt} with $c = 3 |s_1-s_2|$, $r = 0$ and $m = 1$, Lemma \ref{L.ab}, and \eqref{E.a1a2diff}, we get that
\begin{equation} \label{E.Digamma3}
\begin{aligned}
    & \left| \int_{|\varsigma-s_1|\geq 3|s_1-s_2|}  \frac{(s_1-\varsigma)}{|p|_\C^2} \phi(s_1-\varsigma) \big( \digamma'(s_1+i\be) - \digamma'(s_2+i\be) \big) w_4[\ell, \varsigma+i\be] \dd \varsigma \right| \\
    & \quad = \left| \int_{\T}  \frac{(s_1-\varsigma)}{|p|_\C^2} \phi(s_1-\varsigma) \mathds{1}_{|s_1-\varsigma| \geq 3|s_1-s_2|}\, \big( \digamma'(s_1+i\be) - \digamma'(s_2+i\be) \big) w_4[\ell, \varsigma+i\be] \dd \varsigma \right|  \\[0.3cm]
    & \quad \lesssim |s_1-s_2|^{\frac12} \norm{w_4[\ell,\cdot+i\beta]}_{C^\frac{1}{2}(\T)}\,, \quad \textup{for }|\be| \leq \rho\,.
\end{aligned}
\end{equation}
Hence, we have that
\begin{equation} \label{1st sing}
\begin{aligned}
 & \left| \int_{|\varsigma-s_1|\geq 3|s_1-s_2|} \frac{p^{\perp}-q^{\perp}-r^{\perp}+t^{\perp}}{|p|_{\C}^2} \,  \phi(s_1-\varsigma) w_4[\ell, \varsigma+i\be] \dd \varsigma \right| \\[0.3cm]
 & \qquad \lesssim |s_1-s_2|^{\frac12} \norm{w_4[\ell,\cdot+i\beta]}_{C^\frac{1}{2}(\T)}\,, \quad \textup{for }|\be| \leq \rho\,.
\end{aligned}
\end{equation}

Next, observe that
\begin{align}
    & \frac{2p\scalar{p}{q-p+r-t}_\C}{|p|_\C^4} =  \frac{2p\scalar{p}{\digamma(s_1+i\beta)-\digamma(s_2+i\beta)-(s_1-s_2)\digamma'(s_2+i\beta)}_\C}{|p|_\C^4} \\
    & \quad - \frac{2p\scalar{p}{(s_1-\varsigma) ( \digamma'(s_1+i\be) - \digamma'(s_2+i\be) )}_\C}{|p|_\C^4}\,.
\end{align}
Moreover, using \eqref{E.complex-modulus}, it is immediate to check that
\begin{align}
    &\left|\frac{2p\scalar{p}{\digamma(s_1+i\beta)-\digamma(s_2+i\beta)-(s_1-s_2)\digamma'(s_2+i\beta)}_\C}{|p|_\C^4}\right| \\
    & \qquad \lesssim \left|\frac{\digamma(s_1+i\beta)-\digamma(s_2+i\beta)-(s_1-s_2)\digamma'(s_2+i\beta)}{|p|_\C^2}\right|\, \quad \textup{for } |\be| \leq \rho\,.
\end{align}
Thus, by \eqref{E.Digamma1} and \eqref{E.Digamma2}, it follows that
\begin{equation} 
\begin{aligned}
& \left| \int_{|\varsigma-s_1|\geq 3|s_1-s_2|} \frac{2p\scalar{p}{\digamma(s_1+i\beta)-\digamma(s_2+i\beta)-(s_1-s_2)\digamma'(s_2+i\beta)}_\C}{|p|_\C^4} \,  w_4[\ell, \varsigma+i\be] \dd \varsigma \right| \\[0.3cm]
& \qquad \lesssim |s_1-s_2|^{\frac12} \norm{w_4[\ell,\cdot+i\beta]}_{L^{\infty}(\T)}\,, \quad \textup{for } |\be|\leq \rho\,.
\end{aligned}
\end{equation}

On the other hand, letting $p_1,p_2$ and $\digamma_1,\digamma_2$ denote the components of $p$ and $\digamma$ respectively and expanding the scalar product, we observe that
\begin{align}
\mel \frac{p \scalar{p}{(s_1-\varsigma) ( \digamma'(s_1+i\be) - \digamma'(s_2+i\be) )}_\C}{|p|_\C^4}   \\
&=  \frac{(s_1-\varsigma) p_1p_2 \big( (\digamma'_2,\digamma'_1)(s_1+i\be) -  (\digamma'_2,\digamma'_1)(s_2+i\be) \big)}{|p|_\C^4} \\
& \quad + \frac{(s_1-\varsigma) \left( p_1^2( \digamma_1'(s_1+i\be) - \digamma_1'(s_2+i\be)),\ p_2^2( \digamma_2'(s_1+i\be) - \digamma_2'(s_2+i\be))\right)}{|p|_\C^4} \,.
\end{align}
We then observe that each of the corresponding parts of the kernel has the structure in Lemma \ref{SingInt} by Lemma \ref{L.ab}. Indeed it holds that \begin{align}
&(s_1-\varsigma)\frac{p_1^2}{|p|_{\C}^4}=\frac{(s_1-\varsigma)^3a_1^2+2\eps(s_1-\varsigma)^2a_1b_1+\eps^2(s_1-\varsigma)b_1^2}{|(s_1-\varsigma)a+\eps b|_{\C}^4}\\
&(s_1-\varsigma)\frac{p_1p_2}{|p|_{\C}^4}=\frac{(s_1-\varsigma)^3a_1a_2+\eps(s_1-\varsigma)^2(a_2b_1+a_1b_2)+\eps^2(s_1-\varsigma)b_1b_2}{|(s_1-\varsigma)a+\eps b|_{\C}^4}\\
&(s_1-\varsigma)\frac{p_2^2}{|p|_{\C}^4}=\frac{(s_1-\varsigma)^3a_2^2+2\eps(s_1-\varsigma)^2a_2b_2+\eps^2(s_1-\varsigma)b_2^2}{|(s_1-\varsigma)a+\eps b|_{\C}^4}\,,
\end{align}
where $a_1,a_2,b_1,b_2$ are the components of $a$ and $b$ in Lemma \ref{L.ab} with $s = s_1$. Taking into account  \eqref{direct est} and \eqref{E.lowerbll}, we then conclude from Lemma \ref{SingInt} that  
\begin{equation} 
\begin{aligned}
    & \left| \int_{|\varsigma-s_1|\geq 3|s_1-s_2|}  \frac{2p \scalar{p}{(s_1-\varsigma) ( \digamma'(s_1+i\be) - \digamma'(s_2+i\be) )}_\C}{|p|_\C^4}  \phi(s_1-\varsigma) w_4[\ell, \varsigma+i\be] \dd \varsigma \right| \\[0.3cm]
    & \quad \lesssim |s_1-s_2|^{\frac12} \norm{w_4[\ell,\cdot+i\beta]}_{C^\frac{1}{2}(\T)}\,, \quad \textup{for }|\be| \leq \rho\,,
\end{aligned}
\end{equation}
and so that
\begin{equation} \label{2nd sing}
\begin{aligned}
 & \left| \int_{|\varsigma-s_1|\geq 3|s_1-s_2|} \frac{2p^{\perp}\scalar{p^{\perp}}{q^{\perp}-p^{\perp}+r^{\perp}-t^{\perp}}_\C}{|p|_\C^4} \,  \phi(s_1-\varsigma) w_4[\ell, \varsigma+i\be] \dd \varsigma \right| \\[0.3cm]
 & \qquad \lesssim |s_1-s_2|^{\frac12} \norm{w_4[\ell,\cdot+i\beta]}_{C^\frac{1}{2}(\T)}\,, \quad \textup{for }|\be| \leq \rho\,.
\end{aligned}
\end{equation}

Finally, we estimate the right-hand side in Lemma \ref{Lem double diff}. First, observe that, by \eqref{bd diff}, whenever $|s_1-\varsigma| \geq 3 |s_1-s_2|$, it follows that
\begin{align}
|r-t|+|p-q|\lesssim |s_1-\varsigma|^\frac{3}{2}\,, \quad \textup{for }|\be| \leq \rho\,.\label{rt bd}
\end{align}
Likewise, using \eqref{direct est}, we get that
\begin{equation}\begin{aligned}
 |p-r| &\lesssim \left|(s_1-\varsigma)\digamma'(s_1+i\beta)-(s_2-\varsigma)\digamma'(s_2+i\beta)\right|+\eps\left|\zeta(s_1+i\beta)-\zeta(s_2+i\beta)\right| \\
&\lesssim |s_1-\varsigma||s_1-s_2|^\frac{1}{2}+|s_1-s_2|\,, \quad \textup{for }|\be| \leq \rho\,, \label{pr bd}
\end{aligned}\end{equation}
and that, for $|\be| \leq \rho$,
 \begin{equation}\begin{aligned}
|q-t|&\lesssim \left|\digamma(s_1+i\beta)-\digamma(s_2+i\beta)\right|+\eps\left|\zeta(s_1+i\beta)-\zeta(s_2+i\beta)\right| \lesssim |s_1-s_2|\,. \label{qt bd}
\end{aligned}\end{equation}
Then, using the shortened notation
\begin{equation} \label{E.Kernelpqrt}
{\rm K}(p,q,r,t) :=  \frac{\max\{|p|,|q|,|r|,|t|\}^3}{\min\{||p|_{\C}|,||q|_{\C}|,||r|_{\C}|,||t|_{\C}|\}^6}\big(|p-q|+|r-t|\big)\big(|p-r|+|q-t|\big)\,,
\end{equation}
we infer from \eqref{bd d low}, \eqref{bd d up}, \eqref{rt bd}, \eqref{pr bd} and \eqref{qt bd} that \begin{equation}
\begin{aligned} & \left|
\int_{|s_1-\varsigma|\geq 3|s_1-s_2|} {\rm K}(p,q,r,t) \,\phi(s_1-\varsigma) w_4[\ell, \varsigma+i\be] \dd \varsigma \right|\\
& \qquad \lesssim \norm{w_4[\ell,\cdot + i \be]}_{L^{\infty}(\T)}\, \int_{3|s_1-s_2|}^1x^{-\frac{3}{2}}(x|s_1-s_2|^\frac{1}{2}+|s_1-s_2|)\dx \\
& \qquad\lesssim |s_1-s_2|^\frac{1}{2}\norm{w_4[\ell,\cdot + i \be]}_{L^{\infty}(\T)} \,, \quad \textup{for } |\be| \leq \rho\,. \label{rhs bd}
\end{aligned}
\end{equation}
Combining \eqref{1st sing}, \eqref{2nd sing}, \eqref{rhs bd} and Lemma \ref{Lem double diff}, we get that 
\begin{equation} \label{E.conclussionI2J22}
     {\rm I}_2 \lesssim |s_1-s_2|^{\frac12} \norm{w_4[\ell,\cdot+i\beta]}_{C^{\frac12}(\T)}\,, \quad \textup{for } |\be| \leq \rho\,.
\end{equation}
The result follows from \eqref{E.J2bounded}, \eqref{E.conclussionI1J2}, \eqref{E.conclussionI2J2} and \eqref{E.conclussionI2J22}.
\end{proof}

Finally, concerning $J_3$, we have the following:

\begin{lemma} \label{L.J3}
Let $C_2 > 0$ be as in Lemma \ref{bds d}. If $w \in \cY^4_\rho$ satisfies \eqref{ass small} with $\cC_2 \leq C_2$, \eqref{E.Zrho1} and \eqref{E.Zrho2}, then
$$
\sup_{l,\ell \in [-1,1]} \left\{ \sup_{|\be| \leq \rho} \norm{ \int_\T J_3(\cdot+i\be,\varsigma+i\be) w_4[\ell,\varsigma+i\be] \dd \varsigma}_{C^{\frac12}(\T)} \right\} \lesssim\, \sup_{|l|\leq 1} \norm{w_4[l,\cdot]}_\rho \,.
$$
\end{lemma}

\begin{proof}
Note that $J_3$ is nonzero only whenever $|s-\varsigma| \geq \frac18$. Thus, by \eqref{bd d low}, we have that
$$
{\rm Re }\, \big|d_{l,\ell}^{\eps}(s+i\beta,\varsigma+i\beta)\big|_{\C}^2 \gtrsim 1\,, \quad \textup{for } |\be| \leq \rho\,,
$$
in its support. The result then follows from the definition of $d_{l,\ell}^{\eps}(s+i\beta,\varsigma+i\beta)$ and \eqref{direct est}.
\end{proof}

Having  the previous lemmas at hand, we can now conclude.

\begin{proof}[Proof of \eqref{bd K}]
Taking into account Lemma \ref{L.holomorphicExtension} and \eqref{splitting J}, \eqref{bd K} immediately follows from Lemmas \ref{L.J1}, \ref{L.J2} and \ref{L.J3}. 
\end{proof}

\subsection{Proof of (\ref{E.KLipschitz})} In this subsection we prove \eqref{E.KLipschitz}. Since the proof is similar to the one of \eqref{bd K}, we will skip some details. First of all, taking into account \eqref{E.holomorphicKepsilon}, we see that
\begin{equation} \label{E.decompKLipschitz}
\begin{aligned}
    & K_{\ep}\left[w_1^1[\ell,\cdot],\, w_2^1[\ell,\cdot],\, \int_l^\ell w_3^1[\mu,\cdot] \dd \mu,\,w_4^1[\ell,\cdot]\right](s+i\beta) \\
        & \qquad- K_{\ep}\left[w_1^2[\ell,\cdot],\,w_2^2[\ell,\cdot],\, \int_l^\ell w_3^2[\mu,\cdot] \dd \mu,\,w_4^2[\ell,\cdot]\right](s+i\beta) \\
    & \quad = K_{\ep}\left[w_1^1[\ell,\cdot],\, w_2^1[\ell,\cdot],\, \int_l^\ell w_3^1[\mu,\cdot] \dd \mu,\,w_4^1[\ell,\cdot] - w_4^2[\ell,\cdot]\right](s+i\beta) \\
    & \qquad + K_{\ep}\left[w_1^1[\ell,\cdot],\, w_2^1[\ell,\cdot],\, \int_l^\ell w_3^1[\mu,\cdot] \dd \mu,\,w_4^2[\ell,\cdot]\right](s+i\beta) \\
    & \qquad- K_{\ep}\left[w_1^2[\ell,\cdot],\, w_2^2[\ell,\cdot],\, \int_l^\ell w_3^2[\mu,\cdot] \dd \mu,\,w_4^2[\ell,\cdot]\right](s+i\beta)\,.
\end{aligned}
\end{equation}
The first term on the right-hand side can be directly estimated using \eqref{bd K}. Indeed, we have:

\begin{lemma} \label{L.LipschitzTrivial}
    Let $C_2 > 0$ be as in Lemma \ref{bds d}. If $w^1,w^{2}\in\cY^4_\rho$ {both satisfy} \eqref{ass small} with $\cC_2 \leq C_2$, \eqref{E.Zrho1} and \eqref{E.Zrho2}, then
$$
\begin{aligned}
\sup_{l,\ell \in [-1,1]} & \norm{K_{\ep}\left[w_1^1[\ell,\cdot],\,w_2^1[\ell,\cdot],\, \int_l^\ell w_3^1[\mu,\cdot] \dd \mu,\,w_4^1[\ell,\cdot] - w_4^2[\ell,\cdot]\right](\cdot)}_\rho \\[0.3cm]
& \lesssim \sup_{|l| \leq 1} \norm{w_4^1[l,\cdot]-w_4^2[l,\cdot]}_\rho.
\end{aligned}
$$
\end{lemma}
We now focus on the other two terms and analyze them together. First, we prove the following:

\begin{lemma} \label{L.LipschitzLinfty}
    Let $C_2 > 0$ be as in Lemma \ref{bds d}. If $w^{1},w^{2}\in\cY^4_\rho$ {both satisfy} \eqref{ass small} with $\cC_2 \leq C_2$, \eqref{E.Zrho1} and \eqref{E.Zrho2}, then
$$
\begin{aligned}
\sup_{l,\ell \in [-1,1]} \bigg\{ \sup_{|\be| \leq \rho} \bigg\|  & K_{\ep}\left[w_1^1[\ell,\cdot],\,w_2^1[\ell,\cdot],\, \int_l^\ell w_3^1[\mu,\cdot] \dd \mu,\,w_4^2[\ell,\cdot]\right] \\
    & - K_{\ep}\left[w_1^2[\ell,\cdot],\,w_2^2[\ell,\cdot],\, \int_l^\ell w_3^2[\mu,\cdot] \dd \mu,\,w_4^2[\ell,\cdot]\right] \bigg\|_{L^{\infty}(\T)}
 \bigg\} \lesssim\, \|w^1-w^2\|_{\cY^4_\rho}\,.
\end{aligned}
$$
\end{lemma}

\begin{proof}
Taking into account \eqref{splitting J}, we analyze the part corresponding to each kernel $J_j$ separately. With a slight abuse of notation, we write
\begin{equation} \label{E.JjwDependece}
\begin{aligned}
    & J_1(s-\varsigma, s+i\be;\, w) := \frac{\big(\tilde{d}(s-\varsigma,s+i\beta;\, w)\big)^\perp}{\big|\tilde{d}(s-\varsigma,s+i\beta;\, w)\big|_{\C}^2}\phi(s-\varsigma)\,,\\
    & J_2(s+i\beta,\varsigma+i\beta;\, w) := \left(\frac{\big(d_{l,\ell}^{\eps}(s+i\beta,\varsigma+i\beta;\, w)\big)^\perp}{\big|d_{l,\ell}^{\eps}(s+i\beta,\varsigma+i\beta;\, w)\big|_{\C}^2}-\frac{\big(\tilde{d}(s-\varsigma,s+i\beta;\, w)\big)^\perp}{\big|\tilde{d}(s-\varsigma,s+i\beta;\, w)\big|_{\C}^2}\right)\phi(s-\varsigma)\,,\\
    & J_3(s+i\beta,\varsigma+i\beta;\, w):= \frac{\big(d_{l,\ell}^{\eps}(s+i\beta,\varsigma+i\beta;\, w)\big)^\perp}{\big|d_{l,\ell}^{\eps}(s+i\beta,\varsigma+i\beta;\, w)\big|_{\C}^2}(1-\phi(s-\varsigma))\,,
\end{aligned}
\end{equation}
to emphasize the dependence on $w^1$ and $w^2$. Also, throughout the proof we assume that $|\be| \leq \rho \leq \rho_0$.

We first analyze the part corresponding to $J_1$. To that end, we set
$$
a_j = a_j(s+i\be) := \digamma'(s+i\be;\, w^j)\,, \quad \textup{and} \quad b_j = b_j(s+i\be) = \ze(s+i\be;\, w^j)\,,
$$
and point out that{, with} $V_{j}=xa_{j}+\ep b_{j}${,}
\[
J_1(x,s+i\beta;w^1)-J_1(x,s+i\beta;w^2)
=\left(\frac{V_{1}^{\perp}}{|V_1|_{\CC}^2}-\frac{V_{2}^{\perp}}{|V_2|_{\CC}^2}\right)\phi(x).
\]
{The parameter-derivative argument in the proof of Lemma}~{\ref{L.J1} applies to this difference.}

Moreover, it follows that
\begin{equation} \label{E.a1minusa2}
|a_1-a_2| \lesssim \|(w_1^1-w_1^2)[\ell,\cdot+i\be]\|_{L^{\infty}(\T)} +  \|(w_2^1-w_2^2)[\ell,\cdot+i\be]\|_{L^{\infty}(\T)} \lesssim \|w^1-w^2\|_{\cY^4_\rho},
\end{equation}
and that
\begin{equation}\label{E.b1,inusb2}
|b_1-b_2| \lesssim |l-\ell| \sup_{|l| \leq 1} \|(w_3^1-w_3^2)[l,\cdot+i\be]\|_{L^{\infty}(\T)} \lesssim |l-\ell|\, \|w^1-w^2\|_{\cY^4_\rho}.
\end{equation}
Then, arguing exactly as we did in \eqref{E.a1a2diff}-\eqref{E.conclussionI2J1}, we conclude that
\begin{equation} \label{E.J1LinftyLipschitz}
\begin{aligned}
    \sup_{l,\ell \in [-1,1]} & \left\{ \sup_{|\be| \leq \rho} \norm{ \int_\T \big( J_1(\cdot-\varsigma,\cdot+i\be;\, w^1) - J_1(\cdot-\varsigma,\cdot+i\be;\, w^2) \big) w_4^2[\ell,\varsigma+i\be] \dd \varsigma}_{L^{\infty}(\T)} \right\} \\[0.3cm]
    & \lesssim \|w^1-w^2\|_{\cY^4_\rho}  \, \sup_{|l|\leq 1} \|w_4^2[l,\cdot]\|_\rho  \lesssim \|w^1-w^2\|_{\cY^4_\rho}.
\end{aligned}
\end{equation}

Next, to analyze the part corresponding to $J_2$, we rely on Lemma \ref{Lem double diff}. First of all, we set
\begin{align}
    & p = p(\varsigma,s+i\be;\, w^1) := \tilde{d}(s-\varsigma, s+i\be;\, w^1)\,,\\
    & q = q(\varsigma, s+i\be;\, w^1) := d_{l,\ell}^\ep(s+i\be, \varsigma+i\be;\, w^1)\,,\\
    & r = r(\varsigma,s+i\be;\, w^2) :=  \tilde{d}(s-\varsigma, s+i\be;\, w^2)\,, \\
    & t = t(\varsigma,s+i\be;\, w^2) := d_{l,\ell}^\ep(s+i\be, \varsigma+i\be;\, w^2)\,,
\end{align}
and observe that
\begin{align}
    J_2(s+i\be,\varsigma+i\be;\, w^1)- J_2(s+i\be,\varsigma+i\be;\, w^2) = - \left( \frac{p^{\perp}}{|p|_\C^2} - \frac{q^{\perp}}{|q|_\C^2} - \frac{r^{\perp}}{|r|_\C^2} + \frac{t^{\perp}}{|t|_\C^2} \right) \phi(s-\varsigma)\,.
\end{align}
Having this expression at hand, we argue as in the proof of Lemma \ref{L.J2}. First, observe that
$$
|p-q-r+t| \lesssim |s-\varsigma|^{\frac32} \|w^1-w^2\|_{\cY^4_\rho}\,.
$$
Then, by \eqref{bd d low} (see also the proof of \eqref{E.Digamma2}), we get that
\begin{align} \label{E.J2Lipschitz1}
    \left| \int_{\T} \frac{p^{\perp}-q^{\perp}-r^{\perp}+t^{\perp}}{|p|_\C^2} \, \phi(s-\varsigma) w_4^2[\ell,\varsigma+i\be] \dd \varsigma \right|  \lesssim \|w^1-w^2\|_{\cY^4_\rho}\,.
\end{align}
Likewise, we have that
\begin{align}
   \left| \frac{2p\scalar{p}{q-p+r-t}_\C}{|p|_\C^4} \right| \lesssim \frac{\|w^1-w^2\|_{\cY^4_\rho}}{|s-\varsigma|^{\frac12}}\,,
\end{align}
and so that 
\begin{align} \label{E.J2Lipschitz2}
    \left| \int_{\T}  \frac{2p\scalar{p}{q-p+r-t}_\C}{|p|_\C^4} \, \phi(s-\varsigma) w_4^2[\ell,\varsigma+i\be] \dd \varsigma \right|  \lesssim \|w^1-w^2\|_{\cY^4_\rho}\,.
\end{align}

On the other hand, we have that
$$
|p-q| + |r-t| \lesssim |s-\varsigma|^{\frac32}\,,
$$
and that
$$
|q-t|+|p-r| \lesssim \big(|s-\varsigma|+\ep |l-\ell| \big) \|w^1-w^2\|_{\cY^4_\rho}\,.
$$
Using then the shortened notation introduced in \eqref{E.Kernelpqrt}, we conclude that
\begin{align} \label{E.J2Lipschitz3}
    \left| \int_{\T} {\rm K}(p,q,r,t) \,\phi(s-\varsigma) w_4^{2}[\ell, \varsigma+i\be] \dd \varsigma \right| \lesssim \|w^1-w^2\|_{\cY^4_\rho}.
\end{align}
Combining \eqref{E.J2Lipschitz1}, \eqref{E.J2Lipschitz2} and \eqref{E.J2Lipschitz3} with Lemma \ref{Lem double diff}, we then obtain that
\begin{equation} \label{E.J2LinftyLipschitz}
\begin{aligned}
    \sup_{l,\ell \in [-1,1]} & \left\{ \sup_{|\be| \leq \rho} \norm{ \int_\T \big( J_2(\cdot+i\be,\varsigma+i\be;\, w^1) - J_2(\cdot+i\be,\varsigma+i\be;\, w^2) \big) w_4^2[\ell,\varsigma+i\be] \dd \varsigma}_{L^{\infty}(\T)} \right\} \\[0.3cm]
    &   \lesssim \|w^1-w^2\|_{\cY^4_\rho}.
\end{aligned}
\end{equation}
Finally, arguing as in the proof of Lemma \ref{L.J3}, it {is} straightforward to check that
\begin{equation} \label{E.J3LinftyLipschitz}
\begin{aligned}
    \sup_{l,\ell \in [-1,1]} & \left\{ \sup_{|\be| \leq \rho} \norm{ \int_\T \big( J_3(\cdot+i\be,\varsigma+i\be;\, w^1) - J_3(\cdot+i\be,\varsigma+i\be;\, w^2) \big) w_4^2[\ell,\varsigma+i\be] \dd \varsigma}_{L^{\infty}(\T)} \right\} \\[0.3cm]
    &   \lesssim \|w^1-w^2\|_{\cY^4_\rho}.
\end{aligned}
\end{equation}
The result follows combining \eqref{E.J1LinftyLipschitz}, \eqref{E.J2LinftyLipschitz} and \eqref{E.J3LinftyLipschitz}.
\end{proof}

We are missing the bound of the corresponding H\"older seminorm. Using the notation introduced in \eqref{E.JjwDependece}, we first prove the following:

\begin{lemma} \label{L.J1LipschitzHolder}
    Let $C_2 > 0$ be as in Lemma \ref{bds d}. If $w^{1},w^{2}\in\cY^4_\rho$ {both satisfy} \eqref{ass small} with $\cC_2 \leq C_2$, \eqref{E.Zrho1} and \eqref{E.Zrho2}, then
\begin{equation} \label{E.J1HolderLipschitz}
\begin{aligned}
    \sup_{l,\ell \in [-1,1]} & \left\{ \sup_{|\be| \leq \rho} \bigg[ \int_\T \big( J_1(\cdot-\varsigma,\cdot+i\be;\, w^1) - J_1(\cdot-\varsigma,\cdot+i\be;\, w^2) \big) w_4^2[\ell,\varsigma+i\be] \dd \varsigma\bigg]_{C^{\frac12}(\T)} \right\} \\[0.3cm]
    & \lesssim \|w^1-w^2\|_{\cY^4_\rho}.
\end{aligned}
\end{equation}
\end{lemma}

\begin{proof}
We start pointing out that there is no loss of generality in considering $s_1,s_2 \in \T$ with $|s_1-s_2| \leq \frac1{10}$. The case where $|s_1-s_2|\geq \frac1{10}$ immediately follows from Lemma \ref{L.LipschitzLinfty}. Also, we introduce the shortened notation
$$
a_j^h := a(s_j+i\be;\, w^h) = \digamma'(s_j+i\be;\, w^h)\,, \quad b_j^h := b(s_j + i \be;\, w^h) = \ze(s_j+i\be;\, w^h)\,, \quad j,\,h \in \{1,2\}\,,
$$
and observe that
\begin{equation} \label{E.splitJ1doubledifferences}
\begin{aligned}
    & J_1(s_1-\varsigma, s_1+i\be;\, w^1) - J_1(s_1-\varsigma, s_1 + i\be;\, w^2) \\
    & \qquad- J_1(s_2-\varsigma, s_2+i\be;\, w^1) + J_1(s_2-\varsigma, s_2 + i\be;\, w^2)\\
    & \quad = \bigg(\frac{((s_1-\varsigma)a_1^1+\ep b_1^1)^{\perp}}{|(s_1-\varsigma)a_1^1+\ep b_1^1|_\C^2} - \frac{((s_1-\varsigma)a_1^2+\ep b_1^2)^{\perp}}{|(s_1-\varsigma)a_1^2+\ep b_1^2|_\C^2} \bigg) \phi(s_1-\varsigma) \\
    & \qquad- \bigg( \frac{((s_2-\varsigma)a_2^1+\ep b_2^1)^{\perp}}{|(s_2-\varsigma)a_2^1+\ep b_2^1|_\C^2} - \frac{((s_2-\varsigma)a_2^2+\ep b_2^2)^{\perp}}{|(s_2-\varsigma)a_2^2+\ep b_2^2|_\C^2} \bigg) \phi(s_2-\varsigma) \\
    & \quad = \Big\{ \big( g_a(s_1-\varsigma, a_1^1) - g_a(s_1-\varsigma, a_1^2) - g_a(s_2-\varsigma, a_2^1) + g_a (s_2-\varsigma, a_2^2) \big) \\
    & \qquad + \big( g_b(s_1-\varsigma, b_1^1)  - g_b(s_1-\varsigma, b_1^2) - g_b(s_2-\varsigma, b_2^1) + g_b(s_2-\varsigma, b_2^2) \big) \Big\} \phi(s_1-\varsigma) \\
    & \qquad + \bigg\{ \frac{((s_2-\varsigma)a_2^1+\ep b_2^1)^{\perp}}{|(s_2-\varsigma)a_2^1+\ep b_2^1|_\C^2} - \frac{((s_2-\varsigma)a_2^2+\ep b_2^2)^{\perp}}{|(s_2-\varsigma)a_2^2+\ep b_2^2|_\C^2} \bigg\} \big( \phi(s_1-\varsigma) - \phi(s_2-\varsigma) \big)\,,
\end{aligned}
\end{equation}
with
$$
g_a(s_j-\varsigma, a_j^h) := \frac{((s_j-\varsigma)a_j^h+\ep b_j^1)^{\perp}}{|(s_j-\varsigma)a_j^h+\ep b_j^1|_\C^2} \quad \textup{and} \quad g_b(s_j-\varsigma, b_j^h) := \frac{((s_j-\varsigma)a_j^2+\ep b_j^h)^{\perp}}{|(s_j-\varsigma)a_j^2+\ep b_j^h|_\C^2}\,, \quad j,h \in \{1,2\}\,.
$$

Having  this decomposition at hand, we can estimate the convolution with the double differences in the first bracket using Lemmas \ref{SingInt} and \ref{L.ab}, and arguing as we did in the proof of Lemma \ref{L.geoEst}.  Through the rest of the proof we assume that $|\be| \leq \rho$. First, note that
\begin{align}
& |a_1^1-a_1^2-a_2^1+a_2^2|  = |\digamma'(s_1+i\be;\, w^1) - \digamma'(s_1+i\be;\, w^2) - \digamma'(s_2+i\be;\, w^1) + \digamma'(s_2+i\be; w^2)| \\
& \quad \lesssim |s_1-s_2|^{\frac12} \big( \big[(w_1^1-w_1^2)[\ell,\cdot + i\be]\big]_{C^{\frac12}(\T)} + \big[(w_2^1-w_2^2)[\ell,\cdot + i\be]\big]_{C^{\frac12}(\T)} \big) \\
& \qquad + |s_1-s_2|  \big( \norm{(w_1^1-w_1^2)[\ell,\cdot + i\be]}_{L^{\infty}(\T)} + \norm{(w_2^1-w_2^2)[\ell,\cdot + i\be]}_{L^{\infty}(\T)} \big) \\
& \quad \lesssim |s_1-s_2|^{\frac12} \|w^1-w^2\|_{\cY^4_\rho}\,.
\end{align}
Likewise, we have that (see \eqref{E.a1minusa2})
\begin{align}
    |a_1^1 - a_1^2| + |a_2^1-a_2^2| \lesssim \|w^1-w^2\|_{\cY^4_\rho},
\end{align}
and that
$$
|a_1^1 - a_2^1| + |a_1^2 - a_2^2| \lesssim  |s_1-s_2|^{\frac12}\,.
$$
Combining \eqref{E.convolutionDerivativeHolder} with Lemma \ref{L.ab} and these estimates, we conclude that 
\begin{align}
    & \left| \int_{\T}  \Big( g_a(s_1-\varsigma, a_1^1) - g_a(s_1-\varsigma, a_1^2) - g_a(s_2-\varsigma, a_2^1) + g_a (s_2-\varsigma, a_2^2) \Big)\phi(s_1-\varsigma)  w_4^2[\ell, \varsigma+i\be] \dd \varsigma \right| \\[0.3cm]
    & \qquad \lesssim |s_1-s_2|^{\frac12}\, \|w^1-w^2\|_{\cY^4_\rho}\,, \quad \textup{ for } |\be| \leq \rho\,.  \label{E.conclussionga}
    \end{align}

 On the other hand, we have that
\begin{align}
    & |b_1^1 - b_1^2 - b_2^1 + b_2^2| = |\zeta(s_1+i\be;\, w^1) - \zeta(s_1+i\be;\, w^2) - \zeta(s_2+i\be;\, w^1) + \ze(s_2+i\be;\, w^2)| \\
    & \quad = \left| \left(\int_l^{\ell} (w_3^1-w_3^2)[\mu,s_1+i\be] \dd \mu \right) \dot{\Ga}(s_1+i\be)^{\perp}  - \left( \int_l^\ell (w_3^1-w_3^2)[\mu,s_2+i\be] \dd \mu \right) \dot{\Ga}(s_2+i\be)^{\perp}  \right| \\
    & \quad \lesssim |s_1-s_2|^{\frac12} |l-\ell| \, \bigg( \sup_{|l| \leq 1} \big[(w_3^1-w_3^2)[l,\cdot+i\be]\big]_{C^{\frac12}(\T)} \\
    & \qquad +  |s_1-s_2|^{\frac12} \, \sup_{|l| \leq 1} \norm{(w_3^1-w_3^2)[l,\cdot+i\be]}_{L^{\infty}(\T)} \bigg) \\
    & \quad \lesssim |s_1-s_2|^{\frac12} |l-\ell| \|w^1-w^2\|_{\cY^4_\rho}.
\end{align}
Also, it follows that (see \eqref{E.b1,inusb2})
\begin{align}
    |b_1^1-b_1^2|+|b_2^1-b_2^2| \lesssim |l-\ell| \norm{w^1-w^2}_{\cY^4_\rho},
\end{align}
and that
$$
|b_1^1-b_2^1|+|b_1^2-b_2^2| \lesssim |l-\ell| |s_1-s_2|^{\frac12}.
$$
Combining \eqref{E.convolutionDerivativeHolder}, Lemma \ref{L.ab}, and the lower bound \eqref{E.lowerbll} with these estimates, we conclude that
\begin{align}
    & \left| \int_{\T}  \Big( g_b(s_1-\varsigma, b_1^1) - g_b(s_1-\varsigma, b_1^2) - g_b(s_2-\varsigma, b_2^1) + g_b (s_2-\varsigma, b_2^2) \Big) \phi(s_1-\varsigma) w_4^2[\ell, \varsigma+i\be] \dd \varsigma \right| \\[0.3cm]
    & \qquad \lesssim |s_1-s_2|^{\frac12}\, \|w^1-w^2\|_{\cY^4_\rho}\,, \quad \textup{ for } |\be| \leq \rho\,.  \label{E.conclussiongb}
    \end{align}

Finally, to estimate the convolution with the remaining term in \eqref{E.splitJ1doubledifferences}, we point out that
\begin{align}
     & \left| \frac{((s_2-\varsigma)a_2^1+\ep b_2^1)^{\perp}}{|(s_2-\varsigma)a_2^1+\ep b_2^1|_\C^2} - \frac{((s_2-\varsigma)a_2^2+\ep b_2^2)^{\perp}}{|(s_2-\varsigma)a_2^2+\ep b_2^2|_\C^2} \right| \big| \phi(s_1-\varsigma) - \phi(s_2-\varsigma) \big| \\
     & \quad = \left| \frac{((s_2-\varsigma)a_2^1+\ep b_2^1)^{\perp}}{|(s_2-\varsigma)a_2^1+\ep b_2^1|_\C^2} - \frac{((s_2-\varsigma)a_2^2+\ep b_2^2)^{\perp}}{|(s_2-\varsigma)a_2^2+\ep b_2^2|_\C^2} \right|\left| \int_0^1 \phi'(\la(s_1-s_2) + s_2-\varsigma) \dd \la \right||s_1-s_2|\,.
\end{align}
Then, taking into account \eqref{E.supportPhiprime} and arguing as in the proof of \eqref{E.J1LinftyLipschitz}, one can see that

\begin{align}
    & \left| \int_{\T}  \bigg\{ \frac{((s_2-\varsigma)a_2^1+\ep b_2^1)^{\perp}}{|(s_2-\varsigma)a_2^1+\ep b_2^1|_\C^2} - \frac{((s_2-\varsigma)a_2^2+\ep b_2^2)^{\perp}}{|(s_2-\varsigma)a_2^2+\ep b_2^2|_\C^2} \bigg\} \big( \phi(s_1-\varsigma) - \phi(s_2-\varsigma) \big) w_4^2[\ell, \varsigma+i\be] \dd \varsigma \right| \\[0.3cm]
    & \qquad \lesssim |s_1-s_2|^{\frac12}\, \|w^1-w^2\|_{\cY^4_\rho}\,, \quad \textup{ for } |\be| \leq \rho\,. \label{E.conclussionLastBracket}
    \end{align}
 The result follows from \eqref{E.conclussionga}, \eqref{E.conclussiongb} and \eqref{E.conclussionLastBracket}.
\end{proof}

\begin{lemma} \label{L.J2LipschitzHolder}
    Let $C_2 > 0$ be as in Lemma \ref{bds d}. If $w^{1},w^{2}\in\cY^4_\rho$ {both satisfy} \eqref{ass small} with $\cC_2 \leq C_2$, \eqref{E.Zrho1} and \eqref{E.Zrho2}, then
\begin{equation} \label{E.J2HolderLipschitz}
\begin{aligned}
    \sup_{l,\ell \in [-1,1]} & \left\{ \sup_{|\be| \leq \rho} \bigg[ \int_\T \big( J_2(\cdot+i\be,\varsigma+i\be;\, w^1) - J_2(\cdot+i\be,\varsigma+i\be;\, w^2) \big) w_4^2[\ell,\varsigma+i\be] \dd \varsigma\bigg]_{C^{\frac12}(\T)} \right\} \\[0.3cm]
    & \lesssim \|w^1-w^2\|_{\cY^4_\rho}.
\end{aligned}
\end{equation}
\end{lemma}

\begin{proof}
We start pointing out that there is no loss of generality in considering $s_1, s_2 \in \T$ with $|s_1-s_2| \leq \frac1{32}$. The case where $|s_1-s_2| \geq \frac{1}{32}$ follows from Lemma \ref{L.LipschitzLinfty}. Now, observe that
\begin{equation} \label{E.J2LipschitzBeginning}
\begin{aligned}
    & \bigg| \int_\T \Big( J_2(s_1+i\be,\varsigma+i\be;\, w^1) - J_2(s_1+i\be,\varsigma+i\be;\, w^2) \\
    & \qquad- J_2(s_2+i\be,\varsigma+i\be;\, w^1) + J_2(s_2+i\be,\varsigma+i\be;\, w^2)  \Big)  w_4^2[\ell,\varsigma+i\be] \dd \varsigma \bigg| \\
    & \leq\int_{|\varsigma-s_1| \leq 3|s_1-s_2|} \big|\big( J_2(s_1+i\be,\varsigma+i\be;\, w^1) - J_2(s_1+i\be,\varsigma+i\be;\, w^2)\big) w_4^2[\ell,\varsigma+i\be] \big| \dd \varsigma \\
    & \quad + \int_{|\varsigma-s_2| \leq 4|s_1-s_2|} \big| \big( J_2(s_2+i\be,\varsigma+i\be;\, w^1) - J_2(s_2+i\be,\varsigma+i\be;\, w^2)\big) w_4^2[\ell,\varsigma+i\be] \big| \dd \varsigma  \\
    & \quad + \bigg| \int_{|\varsigma-s_1| \geq 3|s_1-s_2|} \Big( J_2(s_1+i\be,\varsigma+i\be;\, w^1) - J_2(s_1+i\be,\varsigma+i\be;\, w^2) \\
    & \qquad- J_2(s_2+i\be,\varsigma+i\be;\, w^1) + J_2(s_2+i\be,\varsigma+i\be;\, w^2)  \Big)  w_4^2[\ell,\varsigma+i\be] \dd \varsigma \bigg|.
\end{aligned}
\end{equation}
The first two terms can be handled as we did in Lemma \ref{L.LipschitzLinfty} to prove \eqref{E.J2LinftyLipschitz}, the main difference being the domain where we are integrating. This difference is precisely what makes appear the extra factor $|s_1-s_2|^{\frac12}$, and allows us to prove that
\begin{equation} \label{E.J2LipschitzLipschitz-1}
\begin{aligned}
     & \int_{|\varsigma-s_1| \leq 3|s_1-s_2|} \big| \big( J_2(s_1+i\be,\varsigma+i\be;\, w^1) - J_2(s_1+i\be,\varsigma+i\be;\, w^2)\big) w_4^2[\ell,\varsigma+i\be] \big| \dd \varsigma  \\
    & \quad + \int_{|\varsigma-s_2| \leq 4|s_1-s_2|}\big| \big( J_2(s_2+i\be,\varsigma+i\be;\, w^1) - J_2(s_2+i\be,\varsigma+i\be;\, w^2)\big) w_4^2[\ell,\varsigma+i\be] \big| \dd \varsigma \\[0.3cm]
    & \lesssim |s_1-s_2|^{\frac12} \|w^1-w^2\|_{\cY^4_\rho}\,, \quad \textup{for }|\be| \leq \rho\,.
\end{aligned}
\end{equation}

Having  this estimate at hand, we are just missing to deal with the last term on the right-hand side of \eqref{E.J2LipschitzBeginning}. To that end, we first split the kernel there as \allowdisplaybreaks
\begin{align}
    & J_2(s_1+i\be,\varsigma+i\be;\, w^1) - J_2(s_1+i\be,\varsigma+i\be;\, w^2) \\
    &\quad - J_2(s_2+i\be,\varsigma+i\be;\, w^1) + J_2(s_2+i\be,\varsigma+i\be;\, w^2) \\
    & = \Bigg( \frac{\big(d_{l,\ell}^{\eps}(s_1+i\beta,\varsigma+i\beta;\, w^1)\big)^\perp}{\big|d_{l,\ell}^{\eps}(s_1+i\beta,\varsigma+i\beta;\, w^1)\big|_{\C}^2} - \frac{\big(\tilde{d}(s_1-\varsigma,s_1+i\beta;\, w^1)\big)^\perp}{\big|\tilde{d}(s_1-\varsigma,s_1+i\beta;\, w^1)\big|_{\C}^2} - \frac{\big(d_{l,\ell}^{\eps}(s_1+i\beta,\varsigma+i\beta;\, w^2)\big)^\perp}{\big|d_{l,\ell}^{\eps}(s_1+i\beta,\varsigma+i\beta;\, w^2)\big|_{\C}^2} \\
    & \quad + \frac{\big(\tilde{d}(s_1-\varsigma,s_1+i\beta;\, w^2)\big)^\perp}{\big|\tilde{d}(s_1-\varsigma,s_1+i\beta;\, w^2)\big|_{\C}^2} - \frac{\big(d_{l,\ell}^{\eps}(s_2+i\beta,\varsigma+i\beta;\, w^1)\big)^\perp}{\big|d_{l,\ell}^{\eps}(s_2+i\beta,\varsigma+i\beta;\, w^1)\big|_{\C}^2} + \frac{\big(\tilde{d}(s_2-\varsigma,s_2+i\beta;\, w^1)\big)^\perp}{\big|\tilde{d}(s_2-\varsigma,s_2+i\beta;\, w^1)\big|_{\C}^2} \\
    & \quad + \frac{\big(d_{l,\ell}^{\eps}(s_2+i\beta,\varsigma+i\beta;\, w^2)\big)^\perp}{\big|d_{l,\ell}^{\eps}(s_2+i\beta,\varsigma+i\beta;\, w^2)\big|_{\C}^2} - \frac{\big(\tilde{d}(s_2-\varsigma,s_2+i\beta;\, w^2)\big)^\perp}{\big|\tilde{d}(s_2-\varsigma,s_2+i\beta;\, w^2)\big|_{\C}^2} \Bigg) \phi(s_1- \varsigma) \\
    & \quad + \Bigg(\frac{\big(d_{l,\ell}^{\eps}(s_2+i\beta,\varsigma+i\beta;\, w^1)\big)^\perp}{\big|d_{l,\ell}^{\eps}(s_2+i\beta,\varsigma+i\beta;\, w^1)\big|_{\C}^2} - \frac{\big(\tilde{d}(s_2-\varsigma,s_2+i\beta;\, w^1)\big)^\perp}{\big|\tilde{d}(s_2-\varsigma,s_2+i\beta;\, w^1)\big|_{\C}^2} \\
    & \quad - \frac{\big(d_{l,\ell}^{\eps}(s_2+i\beta,\varsigma+i\beta;\, w^2)\big)^\perp}{\big|d_{l,\ell}^{\eps}(s_2+i\beta,\varsigma+i\beta;\, w^2)\big|_{\C}^2} + \frac{\big(\tilde{d}(s_2-\varsigma,s_2+i\beta;\, w^2)\big)^\perp}{\big|\tilde{d}(s_2-\varsigma,s_2+i\beta;\, w^2)\big|_{\C}^2} \Bigg) \big( \phi(s_1- \varsigma) - \phi(s_2-\varsigma) \big)\,.
\end{align}
The part corresponding to the second parenthesis can be analyzed using Lemma \ref{Lem double diff}. We set
\begin{equation} \label{E.p2q2r2t2}
\begin{aligned}
    & p = p(\varsigma,s_2+i\be;\, w^1) := \tilde{d}(s_2-\varsigma, s_2+i\be;\, w^1)\,,\\
    & q = q(\varsigma, s_2+i\be;\, w^1) := d_{l,\ell}^\ep(s_2+i\be, \varsigma+i\be;\, w^1)\,,\\
    & r = r(\varsigma,s_2+i\be;\, w^2) :=  \tilde{d}(s_2-\varsigma, s_2+i\be;\, w^2)\,, \\
    & t = t(\varsigma,s_2+i\be;\, w^2) := d_{l,\ell}^\ep(s_2+i\be, \varsigma+i\be;\, w^2)\,,
\end{aligned}
\end{equation}
and point out that
\begin{align}
    &\Bigg|\frac{\big(d_{l,\ell}^{\eps}(s_2+i\beta,\varsigma+i\beta;\, w^1)\big)^\perp}{\big|d_{l,\ell}^{\eps}(s_2+i\beta,\varsigma+i\beta;\, w^1)\big|_{\C}^2} - \frac{\big(\tilde{d}(s_2-\varsigma,s_2+i\beta;\, w^1)\big)^\perp}{\big|\tilde{d}(s_2-\varsigma,s_2+i\beta;\, w^1)\big|_{\C}^2} \\
    & \quad - \frac{\big(d_{l,\ell}^{\eps}(s_2+i\beta,\varsigma+i\beta;\, w^2)\big)^\perp}{\big|d_{l,\ell}^{\eps}(s_2+i\beta,\varsigma+i\beta;\, w^2)\big|_{\C}^2} + \frac{\big(\tilde{d}(s_2-\varsigma,s_2+i\beta;\, w^2)\big)^\perp}{\big|\tilde{d}(s_2-\varsigma,s_2+i\beta;\, w^2)\big|_{\C}^2} \Bigg| \big| \phi(s_1- \varsigma) - \phi(s_2-\varsigma) \big| \\
    & = \left| \frac{p^{\perp}}{|p|_\C^2} - \frac{q^{\perp}}{|q|_\C^2} - \frac{r^{\perp}}{|r|_\C^2} + \frac{t^{\perp}}{|t|_\C^2} \right| \big| \phi(s_1- \varsigma) - \phi(s_2-\varsigma) \big| \\
    & = \left| \frac{p^{\perp}}{|p|_\C^2} - \frac{q^{\perp}}{|q|_\C^2} - \frac{r^{\perp}}{|r|_\C^2} + \frac{t^{\perp}}{|t|_\C^2} \right| \left| \int_0^1 \phi'(\la(s_1-s_2) + s_2-\varsigma) \dd \la \right||s_1-s_2|
\end{align}
Taking into account \eqref{E.supportPhiprime}, one can argue as in the proof of \eqref{E.J2LinftyLipschitz} and conclude that
\begin{equation} \label{E.J2LipschitzLipschitz-2}
    \begin{aligned}
        & \left| \int_{|\varsigma-s_1|\geq 3|s_1-s_2|} \left(\frac{p^{\perp}}{|p|_\C^2} - \frac{q^{\perp}}{|q|_\C^2} - \frac{r^{\perp}}{|r|_\C^2} + \frac{t^{\perp}}{|t|_\C^2} \right) \big( \phi(s_1- \varsigma) - \phi(s_2-\varsigma) \big) w_4[\ell,\varsigma+i\be] \dd \varsigma \right| \\[0.3cm]
        & \qquad \lesssim |s_1-s_2|^{\frac12} \norm{w^1-w^2}_{\cY^4_\rho}, \quad \textup{for }|\be| \leq \rho\,.
    \end{aligned}
\end{equation}

To analyze the remaining term, we set
\begin{equation} \label{E.p1q1r1t1}
\begin{aligned}
    & p_1 := \tilde{d}(s_1-\varsigma, s_1+i\be;\, w^1)\,, \qquad && p_2 :=  \tilde{d}(s_1-\varsigma, s_1+i\be;\, w^2)\,,\\ 
    & q_1 := d_{l,\ell}^\ep(s_1+i\be, \varsigma+i\be;\, w^1)\,, \qquad && q_2 := d_{l,\ell}^\ep(s_1+i\be, \varsigma+i\be;\, w^2)\,,  \\
    & r_1 :=  \tilde{d}(s_2-\varsigma, s_2+i\be;\, w^1)\,, && r_2 := \tilde{d}(s_2-\varsigma, s_2+i\be;\, w^2)\,, \\
    & t_1 := d_{l,\ell}^\ep(s_2+i\be, \varsigma+i\be;\, w^1)\,, && t_2 := d_{l,\ell}^\ep(s_2+i\be, \varsigma+i\be;\, w^2)\,.
\end{aligned}
\end{equation}
The corresponding kernel can then be written as
\begin{equation}
    -\bigg( \frac{p_1^{\perp}}{|p_1|_\C^2} - \frac{q_1^{\perp}}{|q_1|_\C^2} - \frac{r_1^{\perp}}{|r_1|_\C^2} + \frac{t_1^{\perp}}{|t_1|_\C^2} - \frac{p_2^{\perp}}{|p_2|_\C^2} + \frac{q_2^{\perp}}{|q_2|_\C^2} + \frac{r_2^{\perp}}{|r_2|_\C^2} - \frac{t_2^{\perp}}{|t_2|_\C^2} \bigg) \phi(s_1-\varsigma)\,,
\end{equation}
and then, using \eqref{E.pqrtComplete} twice, we can rewrite this kernel in a more suitable way for our purposes. Indeed, note that
\allowdisplaybreaks
\begin{align}
& \frac{p_1}{|p_1|_{\C}^2}-\frac{q_1}{|q_1|_{\C}^2}-\frac{r_1}{|r_1|_{\C}^2}+\frac{t_1}{|t_1|_{\C}^2}-\frac{p_2}{|p_2|_{\C}^2}+\frac{q_2}{|q_2|_{\C}^2}+\frac{r_2}{|r_2|_{\C}^2}-\frac{t_2}{|t_2|_{\C}^2}\\
& \quad = \biggl(\frac{p_1-q_1-r_1+t_1}{|p_1|_{\C}^2}-\frac{p_2-q_2-r_2+t_2}{|p_2|_{\C}^2}\biggr)\\
&\qquad+ \biggl(\frac{2p_1\scalar{p_1}{q_1-p_1+r_1-t_1}_{\C}}{|p_1|_{\C}^4}-\frac{2p_2\scalar{p_2}{q_2-p_2+r_2-t_2}_{\C}}{|p_2|_{\C}^4}\biggr) \\  
    & \qquad + \left( (t_1-r_1) \frac{\scalar{p_1-r_1}{p_1+r_1}_{\C}}{|p_1|_{\C}^2|r_1|_{\C}^2}-(t_2-r_2) \frac{\scalar{p_2-r_2}{p_2+r_2}_{\C}}{|p_2|_{\C}^2|r_2|_{\C}^2}\right) \\
    &\qquad+ \biggl(2p_1 \frac{\scalar{q_1-p_1-t_1+r_1}{p_1}_\C \scalar{p_1-q_1}{p_1+q_1}_\C}{|p_1|_\C^4 |q_1|_\C^2}\\
    &\qquad\quad-2p_2 \frac{\scalar{q_2-p_2-t_2+r_2}{p_2}_\C \scalar{p_2-q_2}{p_2+q_2}_\C}{|p_2|_\C^4 |q_2|_\C^2}\biggr) \\
     & \qquad+ \left(q_1 \frac{\scalar{q_1-t_1-p_1+r_1}{q_1-p_1}_\C}{|p_1|^2_\C|q_1|^2_\C }-q_2 \frac{\scalar{q_2-t_2-p_2+r_2}{q_2-p_2}_\C}{|p_2|^2_\C|q_2|^2_\C}\right) \\
     &\qquad+ \left(2(q_1-p_1) \frac{\scalar{q_1-p_1-t_1+r_1}{p_1}_\C}{|p_1|^2_\C |q_1|^2_\C}- 2(q_2-p_2) \frac{\scalar{q_2-p_2-t_2+r_2}{p_2}_\C}{|p_2|^2_\C |q_2|^2_\C}\right)\\
    & \qquad -\biggl(q_1 \scalar{t_1-r_1}{q_1+p_1}_{\C} \left( \frac{1}{|r_1|^2_\C |t_1|^2_\C}  - \frac{1}{|p_1|^2_\C |q_1|^2_\C} \right)\\
    &\qquad\quad-q_2 \scalar{t_2-r_2}{q_2+p_2}_{\C} \left( \frac{1}{|r_2|^2_\C |t_2|^2_\C}  - \frac{1}{|p_2|^2_\C |q_2|^2_\C} \right) \biggr)\\
    &\qquad- \left(q_1 \frac{\scalar{t_1-r_1}{t_1-q_1+r_1-p_1}_\C }{|t_1|^2_\C |r_1|^2_\C}-q_2 \frac{\scalar{t_2-r_2}{t_2-q_2+r_2-p_2}_\C }{|t_2|^2_\C |r_2|^2_\C} \right)\\
    &\qquad+ \left((q_1-t_1) \frac{\scalar{t_1-r_1}{t_1+r_1}_\C}{|t_1|^2_\C |r_1|^2_\C}-(q_2-t_2) \frac{\scalar{t_2-r_2}{t_2+r_2}_\C}{|t_2|^2_\C |r_2|^2_\C}\right)\,. \label{decomp}
\end{align}
The parts corresponding to the first two terms on the right hand side need to be treated as singular integrals. We can estimate them arguing as we did in Lemma \ref{L.J2} to prove \eqref{1st sing} and \eqref{2nd sing}. On the other hand, the parts corresponding to the rest of the terms can be directly treated reasoning as we did in Lemma \ref{L.J2} to prove \eqref{rhs bd}. We show in detail how to deal with the parts corresponding to the first and third parenthesis on the right hand side. The corresponding part for the second one will then follow arguing as we are going to do with the first, and the rest will follow arguing as we are going to do with the third term.  

We first deal with the more difficult part corresponding to the first parenthesis on the right hand side of \eqref{decomp}. To that end, we introduce the auxiliary function 
$$
g: [0,1] \to \C^2\,, \quad \la \mapsto \frac{\la(p_1-q_1-r_1+t_1) + (1-\la)(p_2-q_2-r_2+t_2)}{|\la p_1 + (1-\la)p_2|_\C^2}\,,
$$
and point out that
$$
\frac{p_1-q_1-r_1+t_1}{|p_1|_{\C}^2}-\frac{p_2-q_2-r_2+t_2}{|p_2|_{\C}^2} = g(1) - g(0) = \int_0^1 g'(\la) \dd\la\,.
$$
Hence, it follows that
\begin{align}
    & \frac{p_1-q_1-r_1+t_1}{|p_1|_{\C}^2}-\frac{p_2-q_2-r_2+t_2}{|p_2|_{\C}^2} = \int_0^1 \frac{p_1-q_1-r_1+t_1-p_2+q_2+r_2-t_2}{|\la p_1+(1-\la)p_2|_{\C}^2} \dd \la \\
    & \quad - \int_0^1 \frac{\la(p_1-q_1-r_1+t_1) + (1-\la)(p_2-q_2-r_2+t_2)}{|\la p_1 + (1-\la)p_2|_{\C}^4} \Big(2\la|p_1-p_2|_\C^2+2\scalar{p_2}{p_1-p_2}_\C \Big) \dd \la\,.
\end{align}

We treat the part corresponding to each integral separately. First, setting
\begin{align}
& G(a+ib;\, w_1^j):= \digamma(a+ib;\, w_1^j) - \Gamma(a+ib) = w_1^j[\ell,a+ib] \dot{\Ga}(a+ib)^{\perp}\,, \quad j \in \{1,2\}\,, \\
& a_\la = a_\la(s_1+i\be) := \la \digamma'(s_1+i\be;\, w^1) + (1-\la) \digamma'(s_1+i\be;\, w^2)\,,\\
& b_\la = b_\la(s_1+i\be) := \la \ze(s_1+i\be;\, w^1) + (1-\la) \ze(s_1+i\be;\,w^{2})\,,
\end{align}
we get that
\begin{align*}
 & \frac{p_1-q_1-r_1+t_1-p_2+q_2+r_2-t_2}{|\la p_1+(1-\la)p_2|_{\C}^2}  \\
 & \quad = - \frac{G(s_1+i\be;\, w_1^1-w_1^2)-G(s_2+i\be;\, w_1^1-w_1^2) - (s_1-s_2) G'(s_2+i\be; \, w_1^1-w_1^2)}{|(s_1-\varsigma)a_\la + \ep b_\la|_\C^2} \\
 & \qquad + \frac{(s_1-\varsigma)(G'(s_1+i\be;\, w_1^1-w_1^2)-G'(s_2+i\be;\, w_1^1-w_1^2))}{|(s_1-\varsigma)a_\la + \ep b_\la|_\C^2} \\
 & \quad=: {\rm K}_G^1 (s_1,s_2, \varsigma,\be;\, w_1^1-w_1^2) + {\rm K}_G^2 (s_1,s_2, \varsigma,\be;\, w_1^1-w_1^2)\,. 
\end{align*}
Here, $G'$ denotes the derivative of $G$ with respect to $s$. Then, arguing as in \eqref{E.Digamma1}, we get that
\begin{equation}
\begin{aligned}
  &\left|{\rm K}_G^1(s_1,s_2, \varsigma,\be;\, w_1^1-w_1^2)\right| \lesssim \frac{|s_1-s_2|^{\frac32}}{|s_1-\varsigma|^2} \|w^1 - w^2\|_{\cY^4_\rho}, \quad \textup{for }|\be| \leq \rho\,,
  \end{aligned}
\end{equation}
and so that
\begin{equation} 
\begin{aligned}
& \left| \int_{|\varsigma-s_1|\geq 3|s_1-s_2|} \left( \int_0^1{\rm K}_G^1(s_1,s_2, \varsigma,\be;\, w_1^1-w_1^2)  \dd\la \right) \phi(s_1-\varsigma) w_4[\ell, \varsigma+i\be] \dd \varsigma \right| \\
& \qquad \lesssim \|w^1 - w^2\|_{\cY^4_\rho} \norm{w_4[\ell,\cdot+i\beta]}_{L^{\infty}(\T)} |s_1-s_2|^{\frac32} \int_{|\varsigma-s_1|\geq 3|s_1-s_2|}  \frac{1}{|s_1-\varsigma|^{2}} \dd\varsigma \\
& \qquad \lesssim |s_1-s_2|^{\frac12} \norm{w^1 - w^2}_{\cY^4_\rho} \,, \quad \textup{for } |\be|\leq \rho\,.
\end{aligned}
\end{equation}
On the other hand, let us point out that
$$
|G'(s_1+i\be;\, w_1^1-w_1^2) - G'(s_2+i\be;\, w_1^1-w_1^2)| \lesssim |s_1-s_{2}|^{\frac12} \norm{w^1-w^2}_{\cY^4_\rho}\,.
$$
Arguing then as in the proof of \eqref{E.Digamma3} and using Fubini's theorem, we get that
\begin{equation} 
\begin{aligned}
    & \bigg| \int_{|\varsigma-s_1|\geq 3|s_1-s_2|} \bigg( \int_0^1  {\rm K}_G^2(s_1,s_2, \varsigma,\be;\, w_1^1-w_1^2)   \dd \la \bigg) \phi(s_1-\varsigma) w_4[\ell, \varsigma+i\be] \dd \varsigma  \bigg| \\
    & \quad \leq \int_0^1 \bigg| \int_{\T}  {\rm K}_G^2(s_1,s_2, \varsigma,\be;\, w_1^1-w_1^2) \phi(s_1-\varsigma) \mathds{1}_{|s_1-\varsigma| \geq 3|s_1-s_2|} w_4[\ell, \varsigma+i\be] \dd \varsigma \bigg| \dd \la  \\[0.3cm]
    & \quad \lesssim |s_1-s_2|^{\frac12} \norm{w^1-w^2}_{\cY^4_\rho}, \quad \textup{for }|\be| \leq \rho\,.
\end{aligned}
\end{equation}
Hence, it follows that
\begin{equation} \label{E.doubledouble1}
    \begin{aligned}
        &\bigg| \int_{|\varsigma-s_1|\geq 3|s_1-s_2|} \bigg( \int_0^1 \frac{p_1-q_1-r_1+t_1-p_2+q_2+r_2-t_2}{|\la p_1+(1-\la)p_2|_{\C}^2} \dd \la \bigg)  \phi(s_1-\varsigma) w_4[\ell, \varsigma+i\be] \dd \varsigma  \bigg| \\[0.3cm]
        &\qquad \lesssim |s_1-s_2|^{\frac12} \norm{w^1-w^2}_{\cY^4_\rho}\norm{w_4[\ell,\cdot+i\beta]}_{C^\frac{1}{2}(\T)}\,, \quad \textup{for }|\be| \leq \rho\,.
    \end{aligned}    
\end{equation}

Next, observe that
\begin{equation}
\begin{aligned} \label{E.p1p2Square}
& |p_1-p_2|^2_\C = (s_1-\varsigma)^2 \big( ((w_2^1-w_2^2)[\ell,s_1+i\be])^2 +  ((w_1^1-w_1^2)[\ell,s_1+i\be])^2 |\ddot{\Ga}(s_1+i\be)|^2_\C \big) \\
& \quad + \ep^2 \left( \int_l^\ell (w_3^2-w_3^1)[\mu,s_1+i\be] \dd \mu \right)^2\\
& \quad + 2 \ep (s_1-\varsigma) \left( \int_l^\ell (w_3^2-w_3^1)[\mu,s_1+i\be] \dd \mu \right) (w_2^1-w_2^2)[\ell,s_1+i\be]\,,
\end{aligned}
\end{equation}
and that
\begin{equation} \label{E.p2Scalarp1p2}
\begin{aligned} 
& \scalar{p_{2}}{p_{1}-p_{2}}_\C = (s_1-\varsigma)^2 \Big((w_1^1-w_1^2)[\ell,s_1+i\be] \scalar{\dot{\Ga}(s_1+i\be)}{\ddot{\Ga}(s_1+i\be)^{\perp}}_\C \\
& \quad + w_1^2[\ell,s_1+i\be] (w_1^1-w_1^2)[\ell,s_1+i\be]|\ddot{\Ga}(s_1+i\be)|^2_\C + w_2^2[\ell,s_1+i\be] (w_2^1-w_2^2)[\ell,s_1+i\be] \Big) \\
& \quad + \ep^2 \left( \int_l^\ell (w_3^1-w_3^2)[\mu, s_1+i\be] \dd \mu \right) \left( \int_l^\ell w_3^2[\mu,s_1+i\be]  \dd \mu \right) \\
& \quad + \ep (s_1-\varsigma) \bigg( w_2^2[\ell,s_1+i\be] \left( \int_l^\ell (w_3^2-w_3^1)[\mu, s_1+i\be] \dd \mu \right)\\
& \quad + (w_2^2-w_2^1)[\ell,s_1+i\be] \left( \int_l^\ell w_3^2[\mu,s_1+i\be]  \dd \mu \right) \bigg)\,.
\end{aligned}
\end{equation}
Also, let us point out that
\begin{align}
& \int_0^1 \frac{\la(p_1-q_1-r_1+t_1) + (1-\la)(p_2-q_2-r_2+t_2)}{|\la p_1 + (1-\la)p_2|_{\C}^4} \Big(2\la|p_1-p_2|_\C^2+2\scalar{p_2}{p_1-p_2}_\C \Big) \dd \la \\
& \quad =  \int_0^1 \frac{\la(p_1-q_1-r_1+t_1)}{|(s_1-\varsigma)a_\la + \ep b_\la|_{\C}^4} \Big(2\la|p_1-p_2|_\C^2+2\scalar{p_2}{p_1-p_2}_\C \Big) \dd \la  \\
& \qquad + \int_0^1 \frac{(1-\la)(p_2-q_2-r_2+t_2)}{|(s_1-\varsigma)a_\la + \ep b_\la|_{\C}^4} \Big(2\la|p_1-p_2|_\C^2+2\scalar{p_2}{p_1-p_2}_\C \Big) \dd \la\,.
\end{align}
Taking into account \eqref{E.p1p2Square} and \eqref{E.p2Scalarp1p2}, the part corresponding to each integral above can be estimated arguing as in the proofs of \eqref{2nd sing} and \eqref{E.doubledouble1}. The key here is that
$$
|b_\la| \geq \frac12|l-\ell|\,, \quad \textup{for all } \la \in [0,1]\,.
$$
This can be proved by arguing exactly as in the proof of \eqref{E.lowerbll}. Once we have this estimate, we can systematically use Lemma \ref{SingInt} and get that
\begin{equation} \label{E.doubledouble2}
    \begin{aligned}
         &\bigg| \int_{|\varsigma-s_1|\geq 3|s_1-s_2|} \bigg( \int_0^1\frac{\la(p_1-q_1-r_1+t_1) + (1-\la)(p_2-q_2-r_2+t_2)}{|\la p_1 + (1-\la)p_2|_{\C}^4} \\
         & \hspace{3cm} \times \Big(2\la|p_1-p_2|_\C^2+2\scalar{p_2}{p_1-p_2}_\C \Big)  \dd \la \bigg)  \phi(s_1-\varsigma) w_4[\ell, \varsigma+i\be] \dd \varsigma  \bigg| \\[0.3cm]
        &\qquad \lesssim |s_1-s_2|^{\frac12} \norm{w^1-w^2}_{\cY^4_\rho}, \quad \textup{for }|\be| \leq \rho\,.
    \end{aligned}
\end{equation}
Then, combining \eqref{E.doubledouble1} and \eqref{E.doubledouble2} we conclude that
\begin{equation}
    \begin{aligned}
        & \bigg| \int_{|\varsigma-s_1|\geq 3|s_1-s_2|} \bigg(\frac{p_1-q_1-r_1+t_1}{|p_1|_{\C}^2}-\frac{p_2-q_2-r_2+t_2}{|p_2|_{\C}^2}\bigg) \phi(s_1-\varsigma)w_4[\ell, \varsigma+i\be] \dd \varsigma  \bigg| \\[0.3cm]
        &\qquad \lesssim |s_1-s_2|^{\frac12} \norm{w^1-w^2}_{\cY^4_\rho}, \quad \textup{for }|\be| \leq \rho\,.
    \end{aligned}
\end{equation}

We now deal with part corresponding to the third term on \eqref{decomp}. By direct computations, 
\begin{align}
    & (t_1-r_1) \frac{\scalar{p_1-r_1}{p_1+r_1}_{\C}}{|p_1|_{\C}^2|r_1|_{\C}^2}-(t_2-r_2) \frac{\scalar{p_2-r_2}{p_2+r_2}_{\C}}{|p_2|_{\C}^2|r_2|_{\C}^2} \\
    & \quad = ((t_1-t_2)-(r_1-r_2)) \frac{\scalar{p_1-r_1}{p_1+r_1}_\C}{|p_1|^2_\C |r_1|^2_\C} \\
    & \qquad + (t_2-r_2)  \scalar{p_1-r_1}{p_1+r_1}_\C \left( \frac{1}{|p_1|_\C^2 |r_1|^2_\C} - \frac{1}{|p_2|_\C^2 |r_2|^2_\C} \right) \\
    & \qquad + \frac{t_2-r_2}{|p_2|^2_\C |r_2|^2_\C} \Big( \scalar{p_1-r_1}{p_1+r_1}_\C - \scalar{p_2-r_2}{p_2+r_2}_\C \Big) \\
    & \quad = ((t_1-t_2)-(r_1-r_2)) \frac{\scalar{p_1-r_1}{p_1+r_1}_\C}{|p_1|^2_\C |r_1|^2_\C} \\
    & \qquad + (t_2-r_2) \scalar{p_1-r_1}{p_1+r_1}_\C \left( \frac{\scalar{p_2-p_1}{p_2+p_1}_\C}{|p_1|^2_\C|r_1|^2_\C|p_2|^2_\C} +  \frac{\scalar{r_2-r_1}{r_2+r_1}_\C}{|r_1|^2_\C|p_2|^2_\C|r_2|^2_\C} \right) \\
    & \qquad + \frac{t_2-r_2}{|p_2|^2_\C |r_2|^2_\C} \Big( \scalar{p_1-p_2}{p_1+p_2}_\C + \scalar{r_2-r_1}{r_2+r_1}_\C \Big)\,. \label{E.doubledoublereminder}
\end{align}
Then, arguing as in the proof of \eqref{bd diff}, we get that
$$
|t_1-t_2-r_1+r_2| \lesssim |s_2 - \varsigma|^{\frac32} \norm{w^1-w^2}_{\cY^4_\rho}\,.
$$
Likewise, it follows that
$$
|p_2-p_1| \lesssim \norm{w^1-w^2}_{\cY^4_\rho} (|s_1-\varsigma| + \ep |l-\ell| )\,, \quad |r_2-r_1| \lesssim \norm{w^1-w^2}_{\cY^4_\rho} (|s_2-\varsigma| + \ep |l-\ell| )\,.
$$
Moreover, observe that $|s_1-\varsigma| \geq 3|s_1-s_{2}|$ implies that
$$
\frac23 |s_1-\varsigma| \leq |s_2-\varsigma| \leq \frac43 |s_1-\varsigma|\,.
$$
Thus, we can estimate the right hand side on \eqref{E.doubledoublereminder} as 
\begin{align}
    & \bigg|((t_1-t_2)-(r_1-r_2)) \frac{\scalar{p_1-r_1}{p_1+r_1}_\C}{|p_1|^2_\C |r_1|^2_\C} \\
    & \quad + (t_2-r_2) \scalar{p_1-r_1}{p_1+r_1}_\C \left( \frac{\scalar{p_2-p_1}{p_2+p_1}_\C}{|p_1|^2_\C|r_1|^2_\C|p_2|^2_\C} +  \frac{\scalar{r_2-r_1}{r_2+r_1}_\C}{|r_1|^2_\C|p_2|^2_\C|r_2|^2_\C} \right) \\
    & \quad + \frac{t_2-r_2}{|p_2|^2_\C |r_2|^2_\C} \Big( \scalar{p_1-p_2}{p_1+p_2}_\C + \scalar{r_2-r_1}{r_2+r_1}_\C \Big) \bigg| \\
    &  \lesssim \frac{|s_1-\varsigma||s_1-s_2|^{\frac12} + |s_1-s_2|}{|s_1-\varsigma|^{\frac32}} \norm{w^1-w^2}_{\cY^4_\rho}\,,
\end{align}
and arguing exactly as in \eqref{rhs bd}, we conclude that
\begin{equation}
    \begin{aligned}
        & \bigg| \int_{|\varsigma-s_1|\geq 3|s_1-s_2|} \bigg((t_1-r_1) \frac{\scalar{p_1-r_1}{p_1+r_1}_{\C}}{|p_1|_{\C}^2|r_1|_{\C}^2} \\
        & \hspace{3cm}-(t_2-r_2) \frac{\scalar{p_2-r_2}{p_2+r_2}_{\C}}{|p_2|_{\C}^2|r_2|_{\C}^2}\bigg) \phi(s_1-\varsigma)w_4[\ell, \varsigma+i\be] \dd \varsigma  \bigg| \\[0.3cm]
        &\quad \lesssim |s_1-s_2|^{\frac12} \norm{w^1-w^2}_{\cY^4_\rho}, \quad \textup{for }|\be| \leq \rho\,.
    \end{aligned}
\end{equation}

The rest of the terms on the right hand side of \eqref{decomp} can be handled arguing in a similar way. We conclude that
\begin{equation} \label{E.J2LipschitzLipschitz-3}
    \begin{aligned}
        & \bigg| \int_{|\varsigma-s_1|\geq 3|s_1-s_2|} \bigg( \frac{p_1^{\perp}}{|p_1|_\C^2} - \frac{q_1^{\perp}}{|q_1|_\C^2} - \frac{r_1^{\perp}}{|r_1|_\C^2} + \frac{t_1^{\perp}}{|t_1|_\C^2} \\
        & \hspace{3cm}- \frac{p_2^{\perp}}{|p_2|_\C^2} + \frac{q_2^{\perp}}{|q_2|_\C^2} + \frac{r_2^{\perp}}{|r_2|_\C^2} - \frac{t_2^{\perp}}{|t_2|_\C^2} \bigg)  \phi(s_1-\varsigma)w_4[\ell, \varsigma+i\be] \dd \varsigma  \bigg| \\[0.3cm]
        &\quad \lesssim |s_1-s_2|^{\frac12} \norm{w^1-w^2}_{\cY^4_\rho}, \quad \textup{for }|\be| \leq \rho\,.
    \end{aligned}
\end{equation}
Taking into account \eqref{E.J2LipschitzBeginning}, the result follows combining \eqref{E.J2LipschitzLipschitz-1}, \eqref{E.J2LipschitzLipschitz-2} and \eqref{E.J2LipschitzLipschitz-3}.
\end{proof}

\begin{lemma} \label{L.J3LipschitzHolder}
    Let $C_2 > 0$ be as in Lemma \ref{bds d}. If $w^{1},w^{2}\in\cY^4_\rho$ {both satisfy} \eqref{ass small} with $\cC_2 \leq C_2$, \eqref{E.Zrho1} and \eqref{E.Zrho2}, then
\begin{equation} \label{E.J3HolderLipschitz}
\begin{aligned}
    \sup_{l,\ell \in [-1,1]} & \left\{ \sup_{|\be| \leq \rho} \bigg[ \int_\T \big( J_3(\cdot+i\be,\varsigma+i\be;\, w^1) - J_3(\cdot+i\be,\varsigma+i\be;\, w^2) \big) w_4^2[\ell,\varsigma+i\be] \dd \varsigma\bigg]_{C^{\frac12}(\T)} \right\} \\[0.3cm]
    & \lesssim \|w^1-w^2\|_{\cY^4_\rho}.
\end{aligned}
\end{equation}
\end{lemma}

\begin{proof}
The result immediately follows from the chain rule and the lower bound~{\eqref{bd d low} on the support of} $J_{3}$. 
\end{proof}

At this point, we can finally conclude the proof of Proposition \ref{main est}. 

\begin{proof}[Proof of \eqref{E.KLipschitz}]
 Taking into account \eqref{E.holomorphicKepsilon} and \eqref{splitting J}, \eqref{E.KLipschitz} immediately follows from the decomposition in \eqref{E.decompKLipschitz} and Lemmas \ref{L.LipschitzTrivial}, \ref{L.LipschitzLinfty}, \ref{L.J1LipschitzHolder}, \ref{L.J2LipschitzHolder} and \ref{L.J3LipschitzHolder}.
\end{proof}

\section{Convergence estimates}\label{Sec 7}

This section is devoted to proving Proposition \ref{Prop K conv} and Lemma \ref{lem K0}. We start by proving Lemma \ref{lem K0}. To that end, we start with the following preliminary result:

{
\begin{lemma}\label{lem lsc}
Let $0\leq\rho'<\rho$. Then the embedding
$\cX_\rho\hookrightarrow\cX_{\rho'}$ is compact.
Furthermore, if $\rho\geq0$ and $(u_n)\subset\cX_\rho$
converges distributionally on $\T$ to $u$, with
$\liminf_{n\to\infty}\norm{u_n}_\rho<\infty$, then
$u\in\cX_\rho$ and
\[
\norm{u}_\rho\leq\liminf_{n\to\infty}\norm{u_n}_\rho.
\]
\end{lemma}

\begin{proof}
With Cauchy's integral formula, we can bound every derivative
on each smaller closed strip in terms of $\norm{u}_\rho$.
Arzel\`a--Ascoli therefore gives convergence in $C^{1/2}$
on that strip along a subsequence. The limit is holomorphic
in its interior and real on $\T$, which proves compactness.

For lower semicontinuity, choose a subsequence such that
\[
\norm{u_n}_\rho\longrightarrow
L:=\liminf_{n\to\infty}\norm{u_n}_\rho<\infty.
\]
If $\rho>0$, the restrictions of $u_n$ to both boundary
components of $\T_\rho$ are uniformly bounded in $C^{1/2}$.
By Arzel\`a--Ascoli, a further subsequence converges uniformly
on both components. The maximum principle then gives uniform
convergence on the whole closed strip. The limit is continuous
on $\T_\rho$, holomorphic in its interior, and agrees with $u$
on $\T$ by distributional convergence. Passing to the limit
in the H\"older quotients on each horizontal line yields
$\norm{u}_\rho\leq L$. For $\rho=0$, the same conclusion
follows directly from Arzel\`a--Ascoli on $\T$ and the
lower semicontinuity of the H\"older norm.
\end{proof}
}

For later purposes, we also set 
\begin{align} 
&K_0^{\ell_1,\ell_2}[w](s):= K_0\big[w_1[\ell_1,\cdot],w_{4}[\ell_2,\cdot]\big](s)\,,\\
&K_\eps^{\ell,l}[w](s):=K_\eps\left[w_1[\ell,\cdot],w_2[\ell,\cdot],\int_l^\ell w_3[\mu,\cdot]\dd\mu,w_4[\ell,\cdot]\right](s)\,.
\end{align}
Moreover, through the rest of this section, we assume that $w \in \cZ_\rho^{\eps_0}$ for some fixed small $\eps_0>0$ and that it satisfies 
\begin{align} \label{E.smallness7}
\mathfrak g_\rho(w) < \mathcal{C}_0 \quad \textup{and} \quad \sup_{|l| \leq 1} \norm{w_4[l,\cdot]}_\rho \leq \mathcal{C}_3 < \infty\,.
\end{align}
Note that the constants $\mathcal{C}_0$ and $\mathcal{C}_3$ were chosen in Section \ref{S.proofMain}, with $\mathcal{C}_0$ being sufficiently small. Likewise, we assume that $0 < \ep \leq \ep_0$. Also, let us stress that, for simplicity, here and during the rest of the section, we omit the dependence on $t$. 

Finally, we set 
\begin{align}
  \overline{e}_{\bs}^{\ell}(s):=\dot{\Gamma}(s)+w_2[\ell,s] \dot{\Ga}(s)^{\perp} + w_1[\ell,s] \ddot{\Ga}(s)^{\perp} \quad\text{and}\quad\overline{e}_{\bn}^\ell(s)=  \overline{e}_{\bs}^\ell(s)^{\perp}\,,
\end{align}
for the tangent and normal of $\gamma_{\eps,\ell}(s)$.

To apply the previous kernel estimates with $w$ fixed,
for $0<\eps'\leq\eps_0$, we set
\[
w_1^{\eps',\ell}[\mu]
=w_1[\ell]+\frac{\eps'}{\eps_0}
 \big(w_1[\mu]-w_1[\ell]\big),
\qquad
w_2^{\eps',\ell}=\partial_s w_1^{\eps',\ell},
\qquad
w_3^{\eps',\ell}=w_3,
\qquad
w_4^{\eps',\ell}=w_4.
\]
Then, it follows that $w^{\eps',\ell}\in\cZ_\rho^{\eps'}$ and
$\mathfrak g_\rho(w^{\eps',\ell})\leq\mathfrak g_\rho(w)$.
Moreover,
$K_{\eps'}^{\ell,l}[w^{\eps',\ell}]
 =K_{\eps'}^{\ell,l}[w]$.
Thus Proposition~\ref{main est} gives
\begin{align}\label{bd kelll}
\norm{K_{\eps'}^{\ell,l}[w]}_\rho\lesssim1,
\end{align}
uniformly for $0<\eps'\leq\eps_0$ and $l\ne\ell$.

The proofs of both results (Proposition \ref{Prop K conv} and Lemma \ref{lem K0}) rely on the following two lemmas, whose proofs we postpone.

\begin{lemma}\label{conv lem 2}
{Fix data satisfying \eqref{E.smallness7} on a strip of width} $\rho\geq 0${, and let} $l,\ell\in[-1,1]$ with $l\ne\ell$. It holds that
\begin{align}
&\lim_{\eps \to 0^+}\norm{\bar{e}_{\bn}^\ell(\cdot)\cdot(K_0^{\ell,\ell}[w]-K_\eps^{\ell,l}[w])}_{L^\infty(\T)}= 0\,,\label{conv 1}\\
&\lim_{\eps \to 0^+}\norm{\bar{e}_{\bs}^\ell(\cdot)\cdot\left(K_0^{\ell,\ell}[w]-K_\eps^{\ell,l}[w]\right)-\frac{1}{2}\sign(l-\ell)w_4[\ell,\cdot]}_{L^\infty(\T)}= 0\,.\label{conv 2}
\end{align}
\end{lemma}

\begin{lemma}\label{norm deri lem}
There exists a constant $\cC_4 \in (0,\mathcal{C}_ 0]$, depending only on $\Ga$, such that, for all $l,\ell\in[-1,1]$ with $l\ne\ell$, and every $0 < \rho < \rho_0$, if $w \in \cZ_\rho$ satisfies 
$$
\mathfrak g_\rho(w) < \cC_4 \quad \textup{and} \quad \cC_5:= \sup_{|l|\leq 1} \norm{w_4[l,\cdot]}_\rho < \infty\,,
$$
then, for all $0 \leq \rho' < \rho$ and $0 < \ep \leq \ep_0$, it follows that \begin{align}
\norm{\de_\eps K_\eps^{\ell,l}[w]}_{{\rho'}}\lesssim A(\rho-\rho') \,. \label{bd deri eps}
\end{align}
Here, $A:\R^+\rightarrow \R^+$ is some positive, decreasing function and the implicit constant depends on $\mathcal{C}_5$. 
\end{lemma}

\begin{remark}
{The constant} $\cC_{4}$ {above is the same smallness threshold used in Proposition}~{\ref{Prop K conv} (decreased if necessary).}
\end{remark}

We first prove Lemma \ref{lem K0}.

\begin{proof}[Proof of Lemma \ref{lem K0}]
We only show how to conclude the boundedness estimate \eqref{bd K0} from Lemma \ref{conv lem 2}, the proof of the Lipschitz estimate \eqref{lipschitz K0} is the same.
Take a pair $(f_1,f_2)\in \mathcal{W}_\rho$.
We can lift this to a $w\in \mathcal{Z}_\rho^{\eps_0}$ by setting \begin{align}
    w_{1}[\mu]=f_{1}+\eps_0(\mu-\ell),\qquad w_{2}[\mu]=\partial_{s} f_{1},
\qquad w_{3}[\mu]=1,\qquad w_{4}[\mu]=f_{2}
\end{align}
for some fixed but arbitrary $\ell$.

Then Lemma \ref{conv lem 2} gives
\begin{align}
K_0[(f_1,f_2)]=K_{0}^{\ell,\ell}[w]
=\lim_{\eps\downarrow0}K_{\eps}^{\ell,l}[w](s)
 +\frac{\sign(l-\ell)}2\frac{\overline{e}_{\bs}^{\ell}(s)}{|\overline{e}_{\bs}^{\ell}(s)|^2}w_4[\ell,s].
\end{align}
for any $l \neq \ell$. Since we control $\norm{K_\eps^{\ell,l}[w]}_{\rho}$ uniformly in $\eps$ by \eqref{bd kelll}, we only need to bound the term $\frac{\overline{e}_{\bs}^{\ell}(s)}{|\overline{e}_{\bs}^{\ell}(s)|^2}w_4[\ell,\cdot]$. This in turn
 can be done using Lemma \ref{L.geoEst} in the same way as in the proof of Lemma \ref{L.real analytic} (after possibly making $\mathcal{C}_0$ smaller). 
\end{proof}

We nowmove on to the proof of Proposition \ref{Prop K conv} and prove the following result, which, taking into account Lemma \ref{lem K0}, directly implies Proposition \ref{Prop K conv} by integrating in $\ell$.

\begin{proposition} \label{pro conv}
Let $\cC_4$ be as in Lemma \ref{norm deri lem} and let $l, \ell, \ell_1 \in [-1,1]$ with $l \neq \ell$. For every $0 \leq \rho < \rho_0$, if $w \in \cZ_\rho^\ep$ satisfies 
$$
\mathfrak g_\rho(w) < \cC_4 \quad \textup{and} \quad \cC_5:= \sup_{|l|\leq 1} \norm{w_4[l,\cdot]}_\rho < \infty\,,
$$
then, for all $0 \leq \rho' < \rho <\rho_0$, it follows that
\begin{align}
&\norm{\overline{e}_{\bn}^{\ell_1}(\cdot)\cdot\left(K_0^{\ell_1,\ell}[w]-K_\eps^{\ell,l}[w]\right)}_{\rho'}\lesssim \ep A(\rho-\rho')\,,
\end{align}
and that
\begin{align}
&\norm{\overline{e}_{\bs}^{\ell_1}(\cdot)\cdot\left(K_0^{\ell_1,\ell}[w]-K_\eps^{\ell,l}[w]\right)-\frac{1}{2}\sign(l-\ell)w_4[\ell,\cdot]}_{\rho'}\lesssim \ep  A(\rho-\rho') \,.
\end{align}
Here, $A:\R^+\rightarrow \R^+$ is some positive, decreasing function and the implicit constants depend on $\mathcal{C}_5$. 
\end{proposition}

\begin{proof}[Proof of Proposition \ref{pro conv}] In this proof, we apply Lemmas~\ref{conv lem 2}
and~\ref{norm deri lem} with $\eps_0=\eps$,
keeping $w$ fixed as the kernel parameter varies. We first point out that the proof is reduced to showing the existence of some positive, decreasing function $A$ such that, for all $0 \leq \rho' < \rho$,
\begin{align} \label{conv 3}
&\norm{\bar{e}_{\bn}^\ell(\cdot)\cdot\left(K_0^{\ell,\ell}[w]-K_\eps^{\ell,l}[w]\right)}_{\rho'}\lesssim \ep A(\rho-\rho')\,,
\end{align}
and 
\begin{align} \label{conv 4}
&\norm{\bar{e}_{\bs}^\ell(\cdot)\cdot\left(K_0^{\ell,\ell}[w]-K_\eps^{\ell,l}[w]\right)-\frac{1}{2}\sign(l-\ell)w_4[\ell,\cdot]}_{\rho'}\lesssim \ep  A(\rho-\rho') \,.
\end{align}
Indeed, by Lemma \ref{lem K0}, we get that
\begin{align}
& \norm{K_0^{\ell_1,\ell}[w]-K_0^{\ell,\ell}[w]}_{\rho'} \lesssim \norm{w_1[\ell_1,\cdot]-w_1[\ell,\cdot]}_{\rho'} + \norm{\pd_s(w_1[\ell_1,\cdot]-w_1[\ell,\cdot])}_{\rho'} \\
& \qquad \lesssim \frac{1}{\rho-\rho'} \|w_1[\ell_1,\cdot]- w_1[\ell,\cdot]\|_{\rho} \lesssim \frac{\ep}{\rho-\rho'} \norm{\int_\ell^{\ell_1} (w_3[\mu,\cdot] )\dd\mu }_\rho \lesssim \frac{\ep}{\rho-\rho'}\,.
\end{align}
Likewise, observe that it follows directly from the definition that
\begin{align}
\norm{\overline{e}_{\bn}^{\ell_1}-\bar{e}_{\bn}^\ell}_{\rho'} +    \norm{\overline{e}_{\bs}^{\ell_1} - \bar{e}_{\bs}^\ell}_{\rho'} \lesssim \norm{w_1[\ell_1,\cdot]-w_1[\ell,\cdot]}_{\rho'} + \norm{\pd_s(w_1[\ell_1,\cdot]-w_1[\ell,\cdot])}_{\rho'} \lesssim \frac{\ep}{\rho-\rho'}\,.
\end{align}
Hence, if we show that \eqref{conv 3} and \eqref{conv 4} hold, the result follows from the triangle inequality and the fact that $\cX_\rho$ is an algebra. 

We now prove \eqref{conv 3} and \eqref{conv 4}. First, taking into account \eqref{bd K} and Lemma \ref{lem lsc}, we infer that the convergence in  \eqref{conv 1} and \eqref{conv 2} actually holds in $\cX_{\rho'}$ for all $0 \leq \rho' < \rho$. On the other hand, by Lemma \ref{norm deri lem}, for all $0 \leq \rho' < \rho$ and all $\ep' \in (0,\ep]$, it follows that
\begin{align}
\norm{K_\eps^{\ell,l}[w]-K_{\eps'}^{\ell,l}[w]}_{\rho'}\lesssim \ep A(\rho-\rho')\,,
\end{align}
where $A$ is some positive, decreasing function. Using again the triangle inequality, and sending $\eps'\to 0^+$, we conclude that \eqref{conv 3} and \eqref{conv 4} hold, and thus the result follows.  
\end{proof}

Finally, we prove Lemmas \ref{conv lem 2} and \ref{norm deri lem}.

\begin{proof}[Proof of Lemma \ref{conv lem 2}]
We will reuse $\digamma$ and $\zeta$, which were defined in \eqref{E.digamma} and \eqref{E.zeta}.
Note that $\digamma'(s)=\bar{e}_{\bs}^\ell$ and $\digamma'(s)^\perp=\bar{e}_{\bn}^\ell$.

We consider the single-layer potential
\begin{align}
\Psi(x):= \frac{1}{2\pi} \int_{\T} \log|x-\digamma(\varsigma)|w_4[\ell,\varsigma]\dd \varsigma\,,\label{def psi}
\end{align}
and point out that it is a harmonic function, except on the curve $\mathbf{C}:=\digamma(\cdot)$. Moreover, by definition, it follows that
\begin{align}
K_\eps^{\ell,l}[w](s)=\nabla^\perp \Psi(\digamma(s)+\eps\zeta(s))\,, \quad \textup{for all } 0 < \ep  \leq \ep_0 \,.
\label{psi eps}
\end{align}
\noindent Note that $\mathbf{C}$ is { a $C^{3/2}$ curve} with (non-normalized) tangent and normal vectors $\digamma'(s)$ and $\digamma'(s)^\perp$. Also, since $\Ga$ is parametrized by arclength and \eqref{E.smallness7} holds, we have that $\digamma' \neq 0$. Moreover, by the orientation chosen for the curve $\Ga$ and \eqref{E.smallness7}, we get that $\digamma'(s)^{\perp}$ is the inner normal of the region enclosed by $\mathbf{C}$. In particular, this implies that $\digamma(s)+\eps\zeta(s)$ lies inside the region enclosed by $\mathbf{C}$ if $l>\ell$ and outside if $l<\ell$.

By classical potential theory
(see e.g.\ \cite[Chapter 3]{Folland}), $\Psi$ is continuous
in $\R^2$, and its gradient has continuous one-sided
traces on $\mathbf{C}$. Their tangential components agree.
Understanding the tangential derivative on $\mathbf{C}$
as this common trace, it follows that
\begin{equation} \label{E.SL1}
\begin{aligned}
    & \lim_{\ep \to 0^+} \digamma'(s)^{\perp} \cdot K_\ep^{\ell,l}[w](s) = \lim_{\ep \to 0^+} \digamma'(s) \cdot \nabla \Psi(\digamma(s)+ \ep \ze(s)) \\
    & \qquad= \digamma'(s) \cdot \nabla\Psi(\digamma(s)) = \digamma'(s)^{\perp} \cdot K_0^{\ell,\ell}[w](s)\,,
\end{aligned}
\end{equation}
and this limit is uniform in $s\in\T$. {For} the normal derivative {we use the same inward unit normal vector}
$N=\digamma'^{\perp}/|\digamma'|$ on {both sides of} the {curve,} and {write ``}in{''} and {``out'' for} the {side from which the trace is taken. The jump relations are}
\begin{align}
(\partial_{N}\Psi)_{\rm in}
 &=\frac{w_4[\ell,s]}{2|\digamma'(s)|}
   -\frac{\digamma'(s)}{|\digamma'(s)|}\cdot K_{0}^{\ell,\ell}[w](s),\\
(\partial_{N}\Psi)_{\rm out}
 &=-\frac{w_4[\ell,s]}{2|\digamma'(s)|}
   -\frac{\digamma'(s)}{|\digamma'(s)|}\cdot K_{0}^{\ell,\ell}[w](s).
\end{align}
{The factor} $|\digamma'|^{-1}$ {converts parameter density to arclength density. Since} the {target approaches from the inside when} $l>\ell$, and {from} the {outside when} $l<\ell$, we {obtain}
\begin{equation}\label{E.SL2}
\lim_{h\downarrow0}\digamma'(s)\cdot K_{h}^{\ell,l}[w](s)
=\digamma'(s)\cdot K_0^{\ell,\ell}[w](s)
 -\frac{\sign(l-\ell)}{2}w_4[\ell,s].
\end{equation}
{This convergence} is {uniform on} $\T${, and together with} \eqref{E.SL1} this {proves} the {lemma.} 
\end{proof}

We now prepare the proof of Lemma~\ref{norm deri lem}. Fix $l\ne\ell$ and {set}
\[
\delta=l-\ell,\qquad d=|\delta|,\qquad \sigma=\sign\delta,\qquad
q(s)=\int_{0}^{1}w_{3}[\ell+\theta\delta,s]\,\dd\theta. 
\]
{Then} $\zeta=\delta q\dot\Ga^{\perp}$. {We use the normalized distance} $h=\eps d$ and {the coordinate map}
\[
y(s,h)=\bx(s,n(s,h)),\qquad n(s,h)=w_1[\ell,s]+\sigma hq(s).
\]
The conditions in the definition \eqref{def z} of $\mathcal{Z}_\rho$ imply that  
\[
n(s,h)=\left(1-\frac h{\eps d}\right)w_{1}[\ell,s]
       +\frac h{\eps d}w_{1}[l,s],
\]
{and the same identity holds for} $a:=\partial_{s} n$ {with} $w_{1}$ {replaced} by $w_{2}$. 
{Consequently} we see that \begin{align}
    \norm{n}_{\rho}+\norm{a}_\rho+\norm{q-1}_\rho\lesssim \mathcal{C}_0, \quad \norm{\de_s a}_{\rho'} \lesssim \frac{\cC_0}{\rho-\rho'}, \quad 0 \leq \rho' < \rho \label{smallness new var}
\end{align}
uniformly in $d$ and $0 \leq h \leq \ep_0 d$.

\begin{lemma}\label{L.matrixInverse}
{With this notation, let} $M=(\de_s y,\de_h y)$ and $B=M^{-1}${. If $\mathcal{C}_0$ is sufficiently small, then} $M$ is {invertible and $B\in \cX_{\rho'}^{2\times 2}$ for all $\rho'\in (0,\rho)$ and} 
\[
\|B\|_{\rho'}+\|\partial_{s}B\|_{\rho'}+\|\partial_{h}B\|_{\rho'}
\lesssim A(\rho-\rho'),
\]
{uniformly in} $l\ne\ell$ {and} $0<h\leq\eps_0 d${.} Here, $A: \R^+ \to \R^+$ is some positive, decreasing function. 
\end{lemma}

\begin{proof}
Recalling the definition \eqref{def coord} of $\bf{x}$, we see that \begin{align}\begin{aligned}
&\de_s y=\dot{\Gamma}+n(s,h)\ddot{\Gamma}^\perp+a(s,h)\dot{\Gamma}^\perp\\
&\de_h y=\sigma q(s)\dot{\Gamma}^\perp.\label{entries b1}
\end{aligned}\end{align}
We therefore see that the complexification of the determinant of $M$ is given by \begin{align}
\det M=\scalar{\left(\dot{\Gamma}+n(s,h)\ddot{\Gamma}^\perp+a(s,h)\dot{\Gamma}^\perp\right)^\perp}{\sigma q(s)\dot{\Gamma}^\perp}_{\C}=\sigma q(s)\left(1-n(s,h)\scalar{\ddot{\Gamma}}{\dot{\Gamma}^\perp}_{\C}\right), \label{entries b2}
\end{align}
where we recall that $\Gamma$ was parametrized by arclength. 

Now if $\mathcal{C}_0$ is small enough, then $|n(s,h)\scalar{\ddot{\Gamma}}{\dot{\Gamma}^\perp}_{\C}|<\frac{1}{2}$ and the determinant of the complexification of $M$ is invertible with a bounded inverse.
{The  inverse} and {its derivatives are}
\begin{align}\label{entries b3}
B=\frac{\operatorname{adj}M}{\det M},\qquad
\partial_{h}B=-B(\partial_{h}M)B,\qquad \partial_{s}B=-B(\partial_{s}M)B.\
\end{align}
We obtain from e.g.\ Lemma \ref{L.geoEst} that this lies indeed in $\cX_{\rho'}^{2\times 2}$. The estimate can easily be obtained from that lemma and Cauchy's integral formula.
\end{proof}

\begin{proof}[Proof of Lemma~\ref{norm deri lem}]
{
Let $\Psi$ be the potential defined in \eqref{def psi}, and set
\[
\psi(s,h):=\Psi(y(s,h)),
\qquad
V(s,h):=\nabla^\perp\Psi(y(s,h)).
\]
Then $V(s,h)=K_{h/d}^{\ell,l}[w](s)$. By Proposition~\ref{main est}, $V$ is
uniformly bounded in $\cX_\rho$ for $0<h\leq\ep_0 d$.
Moreover, on real coordinates,
\begin{align}
\psi_s&=V\cdot y_s^\perp,\label{deri stream 1}\\
\psi_h&=V\cdot y_h^\perp.\label{deri stream 2}
\end{align}
These identities provide holomorphic extensions of $\psi_s$
and $\psi_h$, uniformly bounded in $\cX_\rho$. Differentiating
in $s$ and applying Cauchy's estimates, we obtain bounds for
$\psi_{ss}$ and $\psi_{sh}$ on every smaller strip.

We would like to express $\psi_{hh}$ in terms of the other derivatives to be able to express $\de_\eps K_{\eps}^{\ell,l}$ in terms of these.

We estimate the coefficients arising when the Laplacian
is written in the coordinates $(s,h)$. Set
\[
c_{ij}:=\sum_{m=1}^2 B_{im}B_{jm},
\qquad
b_i:=\sum_{m,j=1}^2 B_{jm}\partial_jB_{im},
\]
where $\partial_1=\partial_s$ and $\partial_2=\partial_h$.
By the chain rule,
\[
b_1=\Delta_x\check s(y),
\qquad
b_2=\Delta_x\check h(y).
\]
where $\check{s}$ and $\check h$  denote the components of the inverse of $y$.

For real $s$ and $h>0$, the point $y(s,h)$ lies off the source
curve $\mathbf{C}$, so $\Psi$ is harmonic there. Consequently, rewriting the Laplacian in the coordinates $s$ and $h$ yields that
\begin{equation}
0=c_{11}\psi_{ss}+2c_{12}\psi_{sh}
+c_{22}\psi_{hh}+b_1\psi_s+b_2\psi_h.\label{exp psiss}
\end{equation}
To solve for $\psi_{hh}$, we need to check that the complexification of $c_{22}^{-1}$ is uniformly bounded. 

Using the formulas \eqref{entries b1} -\eqref{entries b3} for the entries of $B$, we see that \begin{equation}
B_{2\cdot}=\frac{\sigma}{q}\left(\dot{\Gamma}^\perp-\frac{a(s,h)}{1-n(s,h)\scalar{\ddot{\Gamma}}{\dot{\Gamma}^{\perp}}_{\C}}\dot{\Gamma}\right).
\end{equation} 
We therefore see that \begin{align}
c_{22}=\scalar{B_{2\cdot}}{B_{2\cdot}}_{\C}=\frac{\sigma^2}{q^2}\left(1+\left(\frac{a(s,h)}{1-n(s,h)\scalar{\ddot{\Gamma}}{\dot{\Gamma}^{\perp}}_{\C}}\right)^2\right).
\end{align}
It follows from \eqref{smallness new var} that \begin{align}
    \norm{c_{22}-1}_\rho\lesssim \mathcal{C}_0
\end{align}
if $\mathcal{C}_0$ is small enough, which in turn implies that $c_{22}$ is invertible on the strip $\T_\rho$. 

Hence, by Lemma~\ref{L.matrixInverse} and
the algebra property of $\cX_{\rho'}$, we have
\[
\sum_{i,j=1}^2\|c_{ij}\|_{\rho'}
+\sum_{i=1}^2\|b_i\|_{\rho'}
+\|c_{22}^{-1}\|_{\rho'}
\lesssim A(\rho-\rho'),
\]
uniformly in $l\ne\ell$ and $0<h\leq\ep_0 d$ for all $\rho'\in (0,\rho)$.

Solving \eqref{exp psiss} for $\psi_{hh}$, and using the bounds obtained above,
provides its holomorphic extension and yields
\begin{align}
\|\psi_{hh}\|_{\rho'}&\lesssim A(\rho-\rho'),
\label{goal 1}\\
\|\psi_{sh}\|_{\rho'}&\lesssim A(\rho-\rho').
\label{goal 2}
\end{align}
Finally, the chain rule gives
\[
\nabla\Psi(y)=B^T
\begin{pmatrix}\psi_s\\ \psi_h\end{pmatrix}.
\]
Differentiating in $h$ and using Lemma~\ref{L.matrixInverse}
together with the preceding estimates, we conclude that
\[
\|\partial_hV\|_{\rho'}\lesssim A(\rho-\rho').
\]
Since $h=\ep d$, differentiation with the kernel arguments
fixed gives
\[
\partial_\ep K_\ep^{\ell,l}[w]
=d\,\partial_hV(\cdot,\ep d).
\]
The bound $d=|l-\ell|\leq2$ proves \eqref{bd deri eps}.
}  
\end{proof}

\section*{Data availability statement}
This manuscript has no associated data. No datasets were generated or analyzed during the current study.  \vspace{-0.5cm}

\section*{Declaration} The authors wish to thank the anonymous referees for their valuable comments and corrections. This work has received funding from the European Research Council (ERC) under the European Union's Horizon 2020 research and innovation programme through the grant agreement~862342 (A.E. and D.M.). A.E. is also partially supported by the grant PID2022-136795NB-I00 of the Spanish Science Agency and the ICMAT--Severo Ochoa grant CEX2019-000904-S. A.J.F. is partially supported by the grants PID2023-149451NA-I00 of MCIN/AEI/10.13039/ 501100011033/FEDER, UE and Proyecto de Consolidaci\'on Investigadora 2022, CNS2022-135640, MICINN (Spain).  \vspace{-0.2cm}

\bibliographystyle{amsplain}

\end{document}